\documentclass[12pt,a4paper]{amsart}
\usepackage{verbatim,fancyvrb}
\usepackage{setspace}
\usepackage[left=2.5cm,right=2.5cm,top=3.0cm,bottom=3.0cm,includeheadfoot]{geometry}
\usepackage{latexsym}
\usepackage{amssymb}
\usepackage{amsthm}
\usepackage{amsmath}
\usepackage{tikz}
\usepackage{enumitem}
\usetikzlibrary{matrix}
\usetikzlibrary{positioning}
\usetikzlibrary{arrows}
\usepackage{graphicx}
\usepackage[font=small,labelfont=bf]{caption}
\usepackage[latin1]{inputenc}
\usepackage{hyperref}

\usepackage[english]{babel}

\theoremstyle{plain}
\newtheorem{thm}{Theorem}
\newtheorem{theorem}[thm]{Theorem}
\newtheorem{corollary}[thm]{Corollary}
\newtheorem{lemma}[thm]{Lemma}
\newtheorem{prop}[thm]{Proposition}
\newtheorem{proposition}[thm]{Proposition}

\newtheorem{problem}[thm]{Problem}

\newtheoremstyle{exm}
{9pt}{9pt}{}{}{\bfseries}{}{.5em}{}
\theoremstyle{exm}

\newtheoremstyle{rmk}
{9pt}{9pt}{}{}{\bfseries}{}{.5em}{}
\theoremstyle{rmk}
\newtheorem{rmk}[thm]{Remark}

\newtheoremstyle{question}
{9pt}{9pt}{}{}{\bfseries}{}{.5em}{}
\theoremstyle{question}

\numberwithin{equation}{section}
\numberwithin{thm}{section}
\numberwithin{figure}{section}

\theoremstyle{definition}
\newtheorem{definition}[thm]{Definition}
\newtheorem{defin}[thm]{Definition}
\newtheorem{example}[thm]{Example}
\newtheorem{remark}[thm]{Remark}

\newcommand{\field}[1]{\mathbb{#1}}
\newcommand{\C}{\field{C}}
\newcommand{\R}{\field{R}}

\newcommand{\Z}{\field{Z}}

\newcommand{\g}{\mathfrak{t}}

\newcommand{\comp}{\left(M,\omega,\Phi\right)}

\title[Compact monotone tall complexity one $T$-spaces]{Compact monotone tall
  complexity one $T$-spaces}

\author[I. Charton]{Isabelle Charton}
\address{Department of Mathematics, University of Haifa,
  199 Abba Khoushy Avenue, Mount Carmel, Haifa, 3498838, Israel.}
\email{icharton@math.haifa.ac.il}

\author[S. Sabatini]{Silvia Sabatini}
\address{Department Mathematik/Informatik, Universit\"at zu K\"oln,
  Weyertal 86-90, 50931 K\"oln, Germany.}
\email{sabatini@math.uni-koeln.de}

\author[D. Sepe]{Daniele Sepe}
\address{Escuela de Matem\'aticas, Universidad Nacional de Colombia sede
  Medell\'in, Calle 59A Norte \#63-20, Edif\'icio 43, Medell\'in, Colombia.}

\address{Instituto de Matem\'atica e Estat\'istica, Departamento de
  Matem\'atica Aplicada (GMA), Universidade Federal Fluminense, Campus
  Gragoat\'a, Rua Prof. Marcos Waldemar de Freitas Reis, s/n, S\~ao
  Domingos, Niter\'oi, RJ, 24210--201, Brazil.}
\email{dsepe@unal.edu.co}

\date{\today}

\begin{document}

\begin{abstract}
  In this paper we study
  compact monotone tall complexity one $T$-spaces. We use the
  classification of Karshon and Tolman, and the
  monotone condition, to prove that any two such
  spaces are isomorphic if and only if they have equal
  Duistermaat-Heckman measures. Moreover, we show that the moment
  polytope is Delzant and reflexive, and provide a complete description of
  the possible Duistermaat-Heckman measures. Whence we obtain a finiteness result that is analogous to that
  for compact monotone symplectic toric manifolds. Furthermore, we show
  that any such $T$-action can be extended to a toric $(T \times
  S^1)$-action. Motivated by a conjecture of Fine and Panov, we prove that any
  compact monotone tall complexity one $T$-space is equivariantly symplectomorphic to a Fano manifold
  endowed with a suitable symplectic form and a complexity one $T$-action.
\end{abstract}

\maketitle

\tableofcontents

\section{Introduction}
Fano manifolds play an important role in complex algebraic geometry
and beyond. A compact complex manifold is {\bf Fano} if its anticanonical
bundle is ample. Any such manifold is simply connected (see
\cite[Corollary 6.2.18]{isko_prok} and \cite[Remark 3.12]{lindsay}),
and its Todd genus equals one (see \cite[Section 1.8]{hirze} for a definition).
Moreover, in any complex dimension there are finitely many
topological types of Fano manifolds (see \cite{kollar}). The Fano
condition can be reformulated in K\"ahler terms: A compact complex manifold $(Y,J)$ is Fano if and only if 
there exists a K\"ahler form $\omega \in \Omega^{1,1}(Y)$ such that $
c_1 = [\omega]$, where $c_1$ is the first Chern class of $(Y,J)$ -- see, for instance, \cite[Sections 5.4 and
7]{ballmann}. This motivates the following definition\footnote{In the literature
  there are slight variations on this definition and sometimes these
  manifolds are also called
  symplectic Fano (see \cite[Remark 11.1.1]{mcduff_sal}).}: A
symplectic manifold $(M,\omega)$ is {\bf (positive) monotone} if there exists
(a positive) $\lambda \in \R$ such that
$$c_1 = \lambda [\omega], $$
where $c_1$ is the first Chern class of $(M,J)$ and $J$ is any
almost complex structure that is compatible with $\omega$. If
$(M,\omega)$ is positive monotone, then $\omega$
can be rescaled so as to be equal to $c_1$ in cohomology.

A driving question in symplectic topology is to determine whether
every compact positive monotone symplectic manifold is
diffeomorphic to a Fano manifold. The answer is affirmative in
(real) dimension up to four by work of McDuff (see
\cite{mcduff_structure}), and is negative starting from dimension twelve by work of Fine and Panov (see
\cite{fp_hyp,rez}). To the best of our knowledge, this is an open
problem in the remaining dimensions. 

Motivated by a conjecture due to
Fine and Panov (see \cite[Conjecture 1.4]{fp}), which is supported by recent
results (see \cite{cho,cho2,cho3,lp}), we study the above question in
the presence of a {\bf Hamiltonian torus action}: An action of a
compact torus $T$ on a symplectic
manifold $(M,\omega)$ by symplectomorphisms that is codified by a smooth $T$-invariant map $\Phi
: M \to \g^*$, called moment map (see Section
\ref{sec:defin-first-prop} for details). If the action is effective
and $M$ is connected, the triple
$(M,\omega, \Phi)$ is called a {\bf Hamiltonian $T$-space}. We remark
that a monotone Hamiltonian $T$-space is necessarily positive
monotone (see Proposition \ref{assumptions moment map c1}). The
following is the long-term question behind this paper.

\begin{problem}\label{prob:main}
  Find necessary and sufficient conditions for a compact monotone
  Hamiltonian $T$-space to be diffeomorphic to a Fano
  variety.
\end{problem}

A starting point to attack Problem \ref{prob:main} is to consider `large'
torus symmetries. This is codified precisely by the {\bf complexity}
of a Hamiltonian $T$-space $(M,\omega,\Phi)$, which is the
non-negative integer
$$ k:=\frac{1}{2}\dim M - \dim T. $$
\noindent
Intuitively, the lower the complexity, the larger the
symmetry. A Hamiltonian $T$-space of complexity $k$ is called a
{\bf complexity $k$ $T$-space}. Problem \ref{prob:main} has been already solved in complexity zero, i.e.,
for compact monotone symplectic toric manifolds. To recall the solution, we
say that two Hamiltonian $T$-spaces are {\bf isomorphic} if they
are symplectomorphic so that the moment maps are intertwined (see
Definition \ref{def hamiltonian space} for a precise statement). Let $(M,\omega,
\Phi)$ be a compact monotone symplectic toric manifold. By Delzant's classification (see \cite{delzant}), the
isomorphism class of $(M,\omega,
\Phi)$ is determined by
the moment polytope $\Phi(M) \subset \g^*$, which is {\bf Delzant} (see Section \ref{sec:moment-polyt-monot}). Moreover, if
without loss of generality we
assume that $c_1 = [\omega]$, then, up to translation,
$\Phi(M)$ is also {\bf reflexive} (see Definition
\ref{defn:reflexive}). Reflexive polytopes were introduced by Batyrev
in \cite{batyrev} in the study of toric Fano varieties and, like Fano
manifolds, enjoy special properties. For
instance, if $\ell \subset \g$ is the standard lattice,
then there are finitely many reflexive polytopes of full dimension in $\g^*$ up to the
standard action of $\mathrm{GL}(\ell^*)$ -- see Corollary
\ref{cor:finite}.

\begin{figure}[h]
\begin{center}
\includegraphics[width=2.5cm]{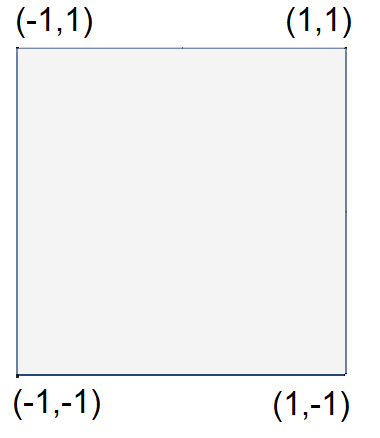}
\caption{The reflexive Delzant polytope in $\R^2$ that is the moment
  map image of the standard $T^2$-action on $\C P^2 \times \C P^2$
  endowed with the product of the Fubini-Study form on each factor.}
\label{Figure:ReflexiveSquare}
	\end{center}
\end{figure}

The combination of Delzant's classification and this
result above yields finiteness for compact monotone
symplectic toric manifolds up to the following notion of equivalence: Two Hamiltonian $T$-spaces $(M_1,\omega_1,\Phi_1)$ and
$(M_2,\omega_2,\Phi_2)$ are {\bf equivalent} if there exists a
symplectomorphism $\Psi : (M_1,\omega_1) \to (M_2,\omega_2)$ and an affine transformation $a \in \mathrm{GL}(\ell^*) \ltimes
\mathfrak{t}^*$ of $\mathfrak{t}^*$ such that $\Phi_2
\circ \Psi  = a \circ \Phi_1$. More precisely, the following holds.

\begin{theorem}\label{thm:finite_stm}
  For each $n\in \Z_{\geq 0}$, there are finitely many equivalence
  classes of compact symplectic toric manifolds of dimension $2n$ with first Chern class equal to the class
  of the symplectic form.
\end{theorem}

Moreover, by Delzant's classification and the K\"ahler description of
the Fano condition, the following result solves Problem \ref{prob:main} in complexity zero.

\begin{theorem}\label{thm:fano}
  If $\comp$ is a compact monotone symplectic toric manifold, then there exists an integrable
  complex structure $J$ on $M$ that is compatible with $\omega$ and 
  invariant under the torus action such that the K\"ahler manifold $(M,J)$ is Fano.
\end{theorem}

In fact, Theorem \ref{thm:fano} proves the stronger result that any
compact monotone symplectic toric manifold is
{\em symplectomorphic} to a Fano manifold (endowed with a suitable
symplectic form). \\

\subsection{The results} In this paper we solve Problem \ref{prob:main} for {\bf tall}
complexity one $T$-spaces, i.e., those for which no reduced space is a point (see
Definition \ref{def tall}). Such spaces have been
classified by Karshon and Tolman in a series of papers (see
\cite{kt1,kt2,kt3}). This classification is
more involved than that of compact symplectic
toric manifolds: For instance, there are several invariants, namely
the moment polytope, the genus, the painting and the
Duistermaat-Heckman measure (see Section \ref{sec:tall-comp-compl} for
more details),
and these invariants satisfy some compatibility conditions (see
\cite{kt3}). 

In order to attack Problem \ref{prob:main} in the above setting, first we
study the isomorphism classes of compact monotone tall
complexity one $T$-spaces. Our first main result states that, for
these spaces, the Duistermaat-Heckman measure determines all other
invariants. For our purposes, we codify this measure by the unique
continuous function that represents its Radon-Nikodym derivative with
respect to the Lebesgue measure on $\g^*$,
which we call the {\bf Duistermaat-Heckman} function (see Theorem \ref{thm:DH_function_cts} and Definition
\ref{Def: DHfucntion}).  

\begin{theorem}\label{thm:DH_classifies}
  Two compact monotone tall complexity one $T$-spaces are isomorphic if and only if their
  Duistermaat-Heckman functions are equal.
\end{theorem}

Our second main result is {\it finiteness} of compact monotone tall complexity one
$T$-spaces up to equivalence, which is the analog of
Theorem \ref{thm:finite_stm} . To this end, we observe that the moment polytope of
a space $\comp$ with $c_1 = \lambda [\omega]$ is a reflexive Delzant polytope if and only if $c_1 =
[\omega]$ and the moment map $\Phi$ satisfies the so-called
  weight sum formula (see Proposition
\ref{prop:mom_map_image} and Lemma
\ref{lemma:normalized_complexity_preserving}). If $\comp$ is
a compact monotone tall complexity one
$T$-space, then we may rescale $\omega$ and translate $\Phi$
so as to satisfy the above conditions (see Corollary
\ref{cor:rescaling} and Proposition \ref{prop:weight_sum}). The next
result is a crucial step towards establishing finiteness.

\begin{theorem}\label{thm:finitely_many_DH}
  Given a reflexive Delzant polytope $\Delta$, there exist finitely
  many isomorphism classes of compact monotone tall
  complexity one $T$-spaces with $\Phi(M) = \Delta$.
\end{theorem}

The following result
is a simple consequence of
Theorem \ref{thm:finitely_many_DH} and answers a question originally posed to us by Yael Karshon.

\begin{corollary}\label{prop:finitely_many}
  For each $n\in \Z_{\geq 0}$, there are finitely many equivalence classes of compact tall
  complexity one $T$-spaces of dimension $2n$ with first Chern class equal to the class
  of the symplectic form.
\end{corollary}

Our third main result concerns the {\it extendability} of a tall complexity one
$T$-action on a compact monotone symplectic manifold $(M,\omega)$ to a
toric $(T \times S^1)$-action. To the best of our
knowledge, there is no criterion to ensure such extendability for
compact tall complexity one $T$-spaces of dimension at least six. 

\begin{theorem}\label{thm:extension}
  If $\comp$ is a compact monotone tall
  complexity one $T$-space $\comp$, then the Hamiltonian $T$-action
  extends to a symplectic toric $(T \times S^1)$-action.
\end{theorem}

\begin{figure}[h]
	\begin{center}
		\includegraphics[width=14cm]{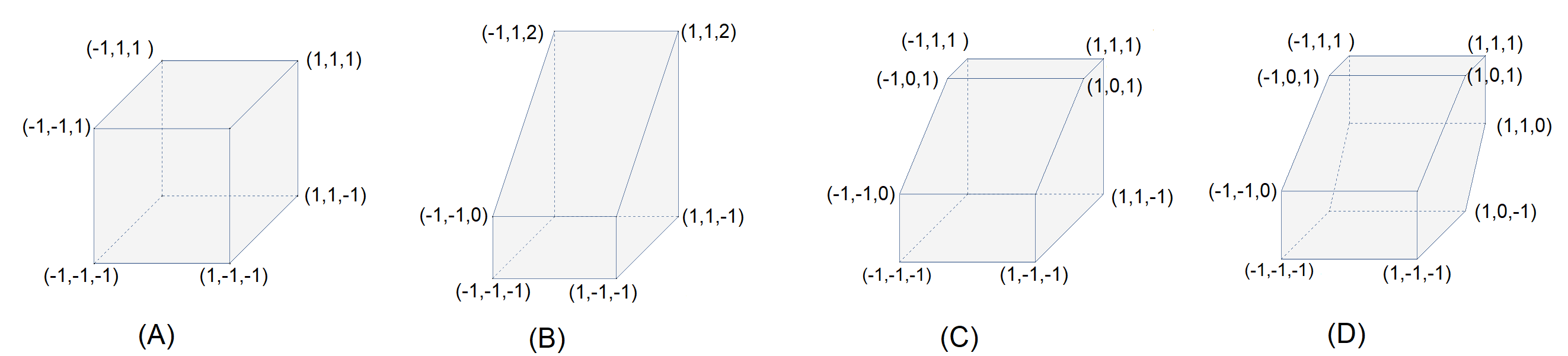}
		\caption{The reflexive Delzant polytopes that are the
                  moment map images of the toric extensions of a compact tall monotone 
                  complexity one $T$-space $\comp$ of dimension six whose moment map image is as in 
                  Figure \ref{Figure:ReflexiveSquare}. We remark that
                  the set consisting of the `height function' of each such polytope, i.e., the
                  difference between the $z$-coordinate of the `top' and
                  `bottom' boundary, is precisely the set of the possible
                  for the Duistermaat-Heckman functions of $\comp$.}
		\label{Figure: ExtensionReflexiveSquare}
	\end{center}
\end{figure}

Finally, our last main result is a {\em solution} to Problem \ref{prob:main} for tall complexity one $T$-spaces. We recall that, given a compact torus $T$, there exists a unique
complex Lie group $T_{\C}$ such that the Lie algebra of $T_{\C}$ is
the complexification of $\g$ and $T$ is a maximal compact subgroup of
$T_{\C}$ (see \cite[Proposition 4.1]{gs-kahler}). For instance, if $T = (S^1)^d$,
then $T_{\C} = (\C^*)^d$. The following result
is concerned with the existence of an integrable complex structure.

\begin{theorem}\label{thm:Fano}
  If $\comp$ is a compact monotone tall complexity one
  $T$-space, then there exists a $T$-invariant integrable
  complex structure $J$ on $M$ that is compatible with $\omega$ such that the complex manifold $(M,J)$ is Fano and the
  $T$-action extends to an effective holomorphic $T_{\C}$-action.
\end{theorem}

As an immediate consequence of Theorem \ref{thm:Fano}, the
following result solves a stronger version of Problem \ref{prob:main}
for compact tall complexity one $T$-spaces.

\begin{corollary}\label{cor:FP}
  Any compact monotone tall complexity one
  $T$-space is equivariantly symplectomorphic to a Fano manifold
  endowed with a suitable symplectic form and a complexity one $T$-action.
\end{corollary}

\subsection{Structure of the paper} In Section
\ref{sec:comp-hamilt-t} we recall fundamental properties of (compact)
Hamiltonian $T$-spaces. While many of the notions and results
presented therein are well-known to experts, we also set the
terminology and notation used throughout. Some basic properties of
Hamiltonian $T$-spaces are considered in Section
\ref{sec:defin-first-prop}, which describes in detail the local models
near any orbit, and introduces the notion of exceptional and regular
points and orbits. These are a generalization of the corresponding
concepts introduced by Karshon and Tolman in complexity one, which
play a key role in their classification. Moreover, we also discuss the notion of
exceptional and regular sheets, which are closely related to the
notion of x-ray (see \cite[Definition 2.1]{tolman_inven}). In Section
\ref{sec:global-invariants} we take a closer look at the invariants of
{\em compact} Hamiltonian $T$-spaces, starting with the so-called
Convexity Package and its consequences. A large part of Section \ref{sec:comp-hamilt-t}
is taken up by the {\em existence} of the Duistermaat-Heckman
function of a compact Hamiltonian $T$-space (see Theorem \ref{thm:DH_function_cts} and Definition
\ref{Def: DHfucntion}). We could not find an appropriate
reference for this result and, hence, included the material for completeness. We
focus on the complexity one case, showing that in this case there is a
polytope in $\g^* \times \R$ that encodes the Duistermaat-Heckman
function (see Corollary \ref{cor:DH_polytope}). In Section
\ref{sec:compl-pres-comp} we introduce compact complexity preserving
Hamiltonian $T$-spaces, a class of spaces that generalizes simultaneously compact symplectic toric
manifolds and compact tall complexity one $T$-spaces. For instance, we
prove that their moment polytopes are Delzant polytopes (see
Proposition \ref{prop:tall_delzant}), and their Duistermaat-Heckman
functions enjoy natural properties (see Corollary \ref{cor:DH_boundary_comp_pres}). These spaces may be of
independent interest. Finally, we review the classification
of compact tall complexity one $T$-spaces due to Karshon and Tolman in
Section \ref{sec:tall-comp-compl}.

In Section \ref{sec:monot-comp-hamilt} we prove some properties of Hamiltonian
$T$-actions on compact monotone symplectic manifolds. We show that the presence of a Hamiltonian $T$-action forces
monotonicity to be positive (see Proposition \ref{assumptions moment
  map c1}). Hence, the symplectic form of any compact monotone
Hamiltonian $T$-space $\comp$ can be rescaled so that it is equal to
$c_1$ in cohomology (see Corollary
\ref{cor:rescaling}). Moreover, we show that the moment map of any compact 
Hamiltonian $T$-space $\comp$ with $c_1 = [\omega]$ can be translated to satisfy the so-called
weight sum formula (see Proposition \ref{prop:weight_sum}). Hence, in
order to study compact monotone
Hamiltonian $T$-spaces, it suffices to consider those that satisfy both aforementioned
conditions, which we call
normalized monotone. That is the content of Section
\ref{sec:weight-sum-formula}. In Section \ref{sec:moment-polyt-tall},
we characterize normalized monotone complexity preserving Hamiltonian
$T$-spaces as being precisely those that are compact
monotone and have reflexive Delzant polytopes as moment map images
(see Proposition
\ref{prop:mom_map_image} and Lemma
\ref{lemma:normalized_complexity_preserving}).

Section \ref{sec:stuff-that-we} is the technical heart of the paper
and is where we prove our first main result, Theorem
\ref{thm:DH_classifies}. In Section \ref{sec:duist-heckm-funct}, we
recall that the genus of a compact monotone tall complexity one
$T$-space is zero, a result proved in \cite{ss}. Moreover, we show
that there is a facet of the moment polytope along
which the Duistermaat-Heckman function is constant and equal to the
minimal value (see Proposition \ref{prop:minimum_facet}). Such a
facet, which we call minimal, plays an important role in the arguments
of Sections \ref{sec:stuff-that-we} and \ref{sec:an-expl-real}. In
Section \ref{sec:the-image-of}, we characterize isolated fixed points
in normalized monotone tall complexity one
$T$-spaces (see Proposition \ref{prop:iso_fixed_pts_monotone}). This is
another fundamental result that we use extensively in Sections
\ref{sec:stuff-that-we} and \ref{sec:an-expl-real}. The space of
exceptional orbits and the painting of a normalized monotone tall complexity one
$T$-space are studied in Section \ref{sec:intern-walls-except}. We
show that the isotropy data associated to the former can be
reconstructed by `looking at the moment polytope' (see Remark
\ref{rmk:image_symplectic_slice} for a precise statement). Moreover,
we prove that the painting of a normalized monotone tall complexity one
$T$-space is trivial (see Definition
\ref{defn:trivial_painting} and Theorem
\ref{prop:trivial_painting}). In Section
\ref{sec:duist-heckm-funct-1}, we provide an explicit formula for the
Duistermaat-Heckman function of a normalized monotone tall complexity one
$T$-space (see Theorem \ref{thm:possibilities_DH}). It is given in terms of the number of connected components of the space
of exceptional orbits and an integer that can be associated to the
preimage of any vertex that lies on a given
minimal facet (see Lemma \ref{lemma:s_invariant}). Finally, in Section
\ref{sec:class-finit-result} we prove
Theorem \ref{thm:DH_classifies} by bringing the above results together. 

In Section \ref{sec:an-expl-real} we prove all our remaining main
results. Section \ref{sec:necess-cond-finit} addresses the question of
finding necessary conditions for a function to be the
Duistermaat-Heckman function of a normalized monotone tall complexity one
$T$-space (see Proposition \ref{prop:necessary}). This allows us to
prove the finiteness result, namely Theorem \ref{thm:finitely_many_DH} and Corollary
\ref{prop:finitely_many}. In Section \ref{sec:comb-real-result}, we
prove that the aforementioned necessary conditions are, in fact,
sufficient (see Theorem \ref{thm:comb_real} and Corollary
\ref{cor:realizability}). Our method to prove this result leads us
naturally to obtain the extendability result, namely Theorem
\ref{thm:extension}. Finally, in Section \ref{sec:comp-monot-tall}, we
prove Theorem \ref{thm:Fano}.

\subsection*{Acknowledgments} We would like to thank Yael
Karshon for posing the question of finiteness.
The authors were partially supported by SFB-TRR 191 grant Symplectic
Structures in Geometry, Algebra and Dynamics funded by the Deutsche
Forschungsgemeinschaft. D.S. was partially supported by FAPERJ grant JCNE E-26/202.913/2019
and by a CAPES/Alexander von Humboldt Fellowship for Experienced
Researchers 88881.512955/2020-01. D. S. would like to thank
Universit\"at zu K\"oln for the kind hospitality during a long stay. This study was financed in
part by the Coordenação de Aperfeiçoamento de Pessoal de Nível Superior -- Brazil
(CAPES) -- Finance code 001. I.C. would like to thank Instituto de Matemática Pura e Aplicada (IMPA), Rio de Janeiro,
for the support of a stay in Brazil. 

\newpage
\subsection{Conventions}\label{sec:conventions}
\subsubsection*{Tori} Throughout the paper, we identify the integral lattice of $S^1 \subset
\C$ with $\Z$. This means that the exponential map $\exp :
\mathrm{Lie}(S^1) = i\R \to S^1$ is given by $\exp(ix) = e^{2\pi i
  x}$, where $z \mapsto e^z$ is the standard complex exponential
function. Moreover, we often identify $\mathrm{Lie}(S^1)$
and its dual with $\R$ tacitly; we trust that this does not cause
confusion.

In this paper, $T$ is a compact torus  of
dimension $d$ with Lie algebra $\g$. We
denote its integral lattice by $\ell$, namely $ \ell=\ker(\exp\colon
\g \rightarrow T)$, and the dual of $\ell$ by $\ell^*$. Moreover, we denote by $\langle
\cdot, \cdot \rangle$ the standard pairing between $\g^*$ and
$\g$. Finally, we fix an inner product on $\g$ once and for all. 

\subsubsection*{Convex polytopes}
Let $\g$ be a real vector
space of dimension $d$ with full-rank
lattice $\ell \subset \g$. A {\bf (convex) polytope} $\Delta$ in
$\g^*$ is a subset that satisfies either of the following two
equivalent conditions:
\begin{itemize}[leftmargin=*]
\item $\Delta$ is the convex hull of a finite set of points, or
\item $\Delta$ is the bounded intersection of a finite set of (closed)
  half-spaces of $\g^*$,
\end{itemize}
(see \cite[Theorem 2.15]{ziegler}). Throughout the paper, we assume
that a polytope $\Delta$ has dimension equal to that of $\g^*$. We often
write $\Delta$ in its minimal representation, i.e.,
\begin{equation}
  \label{eq:3}
  \Delta=\bigcap_{i=1}^l \, \{w\in \g^* \mid \langle w,\nu_i \rangle \geq
  c_i\}\,
\end{equation}
where $\nu_i \in \g$ is the {\bf inward normal}, $c_i \in \R$, and the
affine hyperplane $\{w\in \g^* \mid \langle w,\nu_i \rangle =
c_i\}$ supports a {\bf facet} of $\Delta$ for $i=1,\ldots, l$. A finite non-empty intersection
of facets of $\Delta$ is a {\bf face} of $\Delta$. For convenience, we
think of $\Delta$ also as a face. The {\bf dimension} of a face $\mathcal{F}$ of
$\Delta$ is the dimension of the affine span of $\mathcal{F}$ in
$\g^*$. Faces that are 1-dimensional are called {\bf edges}, while
0-dimensional faces are {\bf vertices}.


\section{(Compact) Hamiltonian $T$-spaces and their invariants}\label{sec:comp-hamilt-t}

\subsection{Definition and first properties}\label{sec:defin-first-prop}

Let $(M,\omega)$ be a symplectic manifold of dimension $2n$. 
A smooth $T$-action $\psi\colon T \times M \to M$ is
\textbf{Hamiltonian} if it admits a 
{\bf moment map}, i.e., a smooth $T$-invariant map $\Phi\colon M \to
\g^*$ that satisfies
\begin{equation}\label{def moment map}
d \langle \Phi, \xi\rangle = -\iota_{\xi^\#} \omega\, \text{ for all }
\xi \in \g,
\end{equation}
where $\xi^\# \in \mathfrak{X}(M)$ denotes the vector field associated
to $\xi$. In this case, the diffeomorphism $\psi(t,
\cdot)\colon M \to M$ is a symplectomorphism
for each $t \in T$, 
i.e., it preserves $\omega$. 
For brevity we denote $\psi(t, p)$ by $t\cdot p$.
\begin{defin}\label{def hamiltonian space}
  \mbox{}
  \begin{itemize}[leftmargin=*]
\item A \textbf{(compact) Hamiltonian $T$-space} is a (compact)
  connected symplectic manifold $(M,\omega)$ endowed with an effective Hamiltonian
  $T$-action and a moment map $\Phi\colon M \to \g^*$. We denote such a
  space by $(M,\omega,\Phi)$.
\item Two Hamiltonian $T$-spaces $(M_1,\omega_1,\Phi_1)$ and
  $(M_2,\omega_2,\Phi_2)$ are {\bf isomorphic} if there exists a symplectomorphism $\Psi: (M_1,\omega_1) \to
  (M_2,\omega_2)$ such that $\Phi_2 \circ \Psi = \Phi_1$.
\end{itemize}
\end{defin}

\begin{remark}\label{rmk:equivariant}
  Since $T$ is connected, an isomorphism of
  Hamiltonian $T$-spaces is necessarily a $T$-equivariant diffeomorphism.
\end{remark}

Definition \ref{def hamiltonian
  space} includes $\dim T = 0$ (this is used, for
instance, in Theorem \ref{thm:DH_function_cts}). In this case, a Hamiltonian $T$-space is simply a symplectic
manifold and an isomorphism is simply a symplectomorphism.

\subsubsection{Orbital moment map and reduced spaces}\label{sec:orbital-moment-map}
We endow the quotient space $M/T$ with the quotient
topology. Since $\Phi$ is $T$-invariant, it descends to a continuous
map $\bar{\Phi} : M/T \to \g^*$
that is called the {\bf orbital moment
  map}.

\begin{remark}\label{rmk:orbital_moment_map}
  If $\Psi$ is an isomorphism between $(M_1,\omega_1,\Phi_1)$ and
  $(M_2,\omega_2,\Phi_2)$ , then there is a homeomorphism $\bar{\Psi}
  : M_1/T \to M_2/T$ such that $\bar{\Phi}_1 = \bar{\Phi}_2 \circ \bar{\Psi}$.
\end{remark}

The fibers of $\bar{\Phi}$ can be canonically identified with the
quotient of the fibers of $\Phi$ by the $T$-action. These are known as
{\bf reduced spaces}. If $\alpha \in \Phi(M)$ is a regular value of
$\Phi$, then the reduced space at $\alpha$, $\Phi^{-1}(\alpha)/T$, is
an orbifold that inherits a symplectic form $\omega_{\mathrm{red}}$
(see \cite{marsden_weinstein} and \cite[Lemma 3.9]{lerman_tolman}). 

\subsubsection{Complexity of a Hamiltonian
  $T$-space}\label{sec:compl-hamilt-t}
Since the $T$-action on
$M$ is effective and since orbits are
isotropic submanifolds of $(M,\omega)$, we have that $d\leq n$. The
difference $n-d$ is a simple, but important invariant of $\comp$.

\begin{definition}\label{defn:complexity}
  The {\bf complexity} of a Hamiltonian $T$-space $\comp$ is
  $$k:=\frac{1}{2}\dim M - \dim T.$$
\end{definition}

\begin{example}\label{exm:complexity_zero=toric}
  Complexity zero Hamiltonian $T$-spaces are {\bf symplectic toric
    manifolds}. Throughout the paper, we refer to torus actions of
  complexity zero as {\bf toric}.
\end{example}

Intuitively, the complexity of a Hamiltonian $T$-space is half of the
dimension of a reduced space at a regular value. 

\subsubsection{Local model and local normal form}\label{sec:local-model-local}
Given $p \in M$, its {\bf stabilizer} is the closed
subgroup $H:=\{t \in T \mid t \cdot p = p\}$. We set $h:= \dim H$. Since $T$ is abelian, any two points on the same orbit have equal
stabilizers. Hence, the stabilizer of an orbit is well-defined. If $\mathcal{O}$ denotes the $T$-orbit containing $p$, the
infinitesimal symplectic linear action of $H$ on $(T_pM,\omega_p)$ fixes $T_p
\mathcal{O}$. Thus there is a symplectic linear action of
$H$ on the quotient vector space $(T_p
\mathcal{O})^{\omega}/T_p\mathcal{O}$ endowed with the quotient linear
symplectic structure. We call the underlying Lie group homomorphism
the {\bf symplectic slice representation of $\boldsymbol{p}$}.

\begin{remark}\label{rmk:symp_slice_orbit}
  The symplectic slice representations of two points lying on the same
  orbit are naturally isomorphic. Hence, the symplectic slice
  representation of an orbit is well-defined. This allows us to `decorate' the
  quotient space $M/T$ by attaching the symplectic slice
  representation to every orbit (this data includes the
  stabilizer of the orbit).

  Let $\Psi$ be an isomorphism between $(M_1,\omega_1,\Phi_1)$ and
  $(M_2,\omega_2,\Phi_2)$ and let $\bar{\Psi} : M_1/T \to M_2/T$ be
  the homeomorphism given by Remark \ref{rmk:orbital_moment_map}. For
  any $p \in M_1$, $\bar{\Psi}([p])$ and $[p]$ have equal stabilizers
  and symplectic slice representations.
\end{remark}

Fix a $T$-invariant
almost complex structure on $(M,\omega)$; the existence of such a
structure is proved in \cite[Lemma
5.52]{mcduff-salamon}. We observe that $(T_p
\mathcal{O})^{\omega}/T_p\mathcal{O}$ has real dimension
$2(h+n-d)=2(h+k)$, where $k$ is the complexity of $\comp$. Hence, we
use the above almost complex structure to
identify $(T_p \mathcal{O})^{\omega}/T_p\mathcal{O}$ with $\C^{h+k}$
endowed with the standard symplectic form
\begin{equation}
  \label{eq:49}
  \frac{i}{2}\,\sum_{j=1}^{h+k} dz_j \wedge d\overline{z}_j.
\end{equation}
\noindent
Moreover, under this identification, the linear $H$-action is by 
unitary transformations.

Let $\rho : H \to U(\C^{h+k})$ be the
associated homomorphism of Lie groups. Since $H$ is abelian,
$\rho(H)$ is contained in a maximal torus of $U(\C^{h+k})$. We denote
the maximal torus of $U(\C^{h+k})$ consisting of diagonal
transformations by $(S^1)^{h+k}$. Hence,
we may assume that $\rho$ factors through
a Lie group homomorphism $H \to  (S^1)^{h+k}$ that we also denote by
$\rho$ by a slight abuse of
notation. We
write $\rho_j$ for the $j$th component of $\rho$, where $j = 1,\ldots,
h+k$. Let $d_e\rho_j$ denote the derivative at the identity of
$\rho_j$ and set
$d_e\rho_1:=2\pi i \alpha_1,\ldots, d_e\rho_{h+k}:=2\pi i
\alpha_{h+k}$. Hence, $\alpha_j \in
\ell^*_{\mathfrak{h}}$ for all $j=1,\ldots, h+k$, where $\mathfrak{h}$
is the Lie algebra of $H$ and $\ell_{\mathfrak{h}}$ denotes
the integral lattice in $\mathfrak{h}$. We call
$\alpha_1,\ldots, \alpha_{h+k}$ the {\bf isotropy weights of $\boldsymbol{p}$ (for the
  $\boldsymbol{H}$-action)}. The multiset of isotropy weights of $p$ does not depend on the
choice of $T$-invariant almost complex structure on
$(M,\omega)$. Moreover, this multiset encodes the action of the
identity component of $H$ on $\C^{h+k}$. Explicitly, 
\begin{equation}\label{action identity component}
  \exp(\xi)\cdot (z_1,\ldots,z_{h+k})=(e^{ 2\pi i \langle \alpha_1,\xi
  \rangle}z_1,\ldots,
  e^{ 2\pi i \langle \alpha_{h+k},\xi \rangle}z_{h+k})\quad\text{for 
    every}\;\;\xi\in \mathfrak{h}\, .
\end{equation}
In particular, if $H$ is connected, then the multiset of isotropy
weights of $p$ determine the symplectic slice representation up to
unitary isomorphisms. Finally, by Remark \ref{rmk:symp_slice_orbit}, the multisets of
isotropy weights of two points lying on the same orbit are equal. \\

From the stabilizer $H \leq T$ of $p$ and the Lie group homomorphism $\rho :
H \to (S^1)^{h+k}$, we construct a symplectic manifold together
with a Hamiltonian $T$-action and a moment map. This is the local model for a $T$-invariant neighborhood of
$\mathcal{O}$ in $\comp$. We do this in two equivalent ways, seeing as
one is more convenient for proofs and the other is more convenient for
calculations.

\subsubsection*{The abstract construction} Let $\Omega$ denote the
symplectic form on $T^*T \times \C^{h+k}$ given
by taking the sum of the pullbacks of the canonical symplectic form on
$T^*T$ and the standard symplectic form on $\C^{h+k}$ (see equation
\eqref{eq:49}). Let $H$ act (on the right) on $T^*T \times \C^{h+k}$
as follows: On $T^*T$ it acts by
the cotangent lift of (right) multiplication, while on $\C^{h+k}$ it
acts by $ z \cdot h := \rho(h^{-1})(z)$, for $h \in H$ and $z \in
\C^{h+k}$. By construction,
the $H$-action on $(T^*T \times \C^{h+k},\Omega)$ is
Hamiltonian. Let $\hat{\Phi} : T^*T \times \C^{h+k} \to
\mathfrak{h}^*$ be the moment map that sends $(0,0) \in T^*T \times
\C^{h+k} $ to the origin in $\mathfrak{h}^*$. Since this $H$-action is
free and proper, the quotient
$$ (T^*T \times \C^{h+k}) /\! /H:= \hat{\Phi}^{-1}(0)/H $$
\noindent
is a smooth manifold that inherits a symplectic form
$\omega_{\mathrm{red}}$ (see \cite[Proposition
5.40]{mcduff-salamon}).

Let $T$ act (on the left) on $T^*T \times \C^{h+k}$ as follows: On
$T^*T$ it acts by the
cotangent lift of (left) multiplication, while on $\C^{h+k}$ it acts
trivially. This $T$-action is Hamiltonian and commutes with the above
$H$-action. Hence, it induces a Hamiltonian $T$-action on $((T^*T
\times \C^{h+k}) /\! /H, \omega_{\mathrm{red}})$. As a moment map for
this $T$-action we take the one that sends $[0,0] \in (T^*T \times
\C^{h+k}) /\! /H$ to the origin in $\g^*$ and we denote it by
$\Phi_{\mathrm{red}}$. The desired local model is the triple $((T^*T
\times \C^{h+k}) /\! /H, \omega_{\mathrm{red}}, \Phi_{\mathrm{red}})$.

\subsubsection*{The explicit construction}
The choice of inner product on $\mathfrak{t}$ induces an
inner product on $\mathfrak{t}^*$ which, in turn, determines an
isomorphism $\g^* \simeq \mathrm{Ann}(\mathfrak{h}) \oplus
\mathfrak{h}^*$. Moreover, we choose a trivialization $T^*T \cong T
\times \g^*$. With these choices, we fix an identification $T^*T
\times \C^{h+k}$ with $T \times \mathrm{Ann}(\mathfrak{h}) \times
\mathfrak{h}^* \times \C^{h+k} $. Under this identification, the above
(right) $H$-action is given by 
$$ (t,\alpha,\beta,z) \cdot h = (th,\alpha,\beta,\rho(h^{-1})z), $$ 
while the above moment map
$\hat{\Phi}$ is given by
$$ (t,\alpha,\beta,z) \mapsto \beta - \Phi_H(z), $$
\noindent
where $\Phi_H : \C^{h+k} \to \mathfrak{h}^*$ is the homogeneous
moment map for the (left) $H$-action on $\C^{h+k}$ given by $h \cdot z
= \rho(h)z$. The map 
\begin{equation*}
  \begin{split}
    T \times \mathrm{Ann}(\mathfrak{h}) \times \C^{h+k} & \to
    \hat{\Phi}^{-1}(0) \\
    (t,\alpha,z) &\mapsto (t,\alpha,\Phi_H(z),z)
  \end{split}
\end{equation*}
is an $H$-equivariant diffeomorphism, where the left hand side is
equipped with the (right) $H$-action on $T \times
\mathrm{Ann}(\mathfrak{h}) \times \C^{h+k}$ given by
\begin{equation}
  \label{eq:45}
  (t,\alpha,z) \cdot h = (th, \alpha, \rho(h^{-1})z).
\end{equation}
\noindent
Hence the quotient
\begin{equation}
  \label{eq:47}
  Y:=T \times_H
  \mathrm{Ann}(\mathfrak{h}) \times \C^{h+k}
\end{equation}
\noindent
is diffeomorphic to $(T^*T \times
\C^{h+k}) /\! /H$. We denote by $\omega_Y$ (respectively $\Phi_Y$) the pullback of
$\omega_{\mathrm{red}}$ (respectively $\Phi_{\mathrm{red}}$) under the above diffeomorphism. The (left)
$T$-action on $Y$ is given by
\begin{equation}
  \label{eq:48}
  s \cdot [t,\alpha,z] = [st,\alpha,z],
\end{equation}
\noindent
while the moment map $\Phi_Y$ takes the form
\begin{equation}
  \label{eq:43}
  \Phi_Y([t,\alpha,z]) = \alpha + \Phi_H(z).
\end{equation}
\noindent
The stabilizer of $[1,0,0] \in Y$ is $H$ and the symplectic slice
representation of $[1,0,0]$ is $\rho : H \to (S^1)^{h+k}$. Moreover, if the $T$-action is effective, the
complexity of $(Y,\omega_Y,\Phi_Y)$ is equal to $k$. We refer to
$(Y,\omega_Y,\Phi_Y)$ as the {\bf local model of
  $\boldsymbol{p}$}. By
Remark \ref{rmk:symp_slice_orbit}, the local models at two points lying
on the same orbit are equal.

\begin{remark}\label{rmk:local_model_two_interpretations}
  By a slight abuse of terminology, we also refer to $((T^*T
  \times \C^{h+k}) /\! /H, \omega_{\mathrm{red}},
  \Phi_{\mathrm{red}})$ as the local model of $p$. Moreover,
  throughout the paper we sometimes state results in terms of $Y$ but use $(T^*T
  \times \C^{h+k}) /\! /H$ in the proof. We trust that this does not
  cause confusion. 
\end{remark}

As a consequence of the local normal form
theorem for Hamiltonian actions of compact Lie groups due to Guillemin-Sternberg \cite{gs-local} and Marle 
\cite{marle}, any Hamiltonian
$T$-space is isomorphic to a local model near an orbit. More
precisely, the following holds.

\begin{thm}\label{local normal form}
Let $\comp$ be a Hamiltonian $T$-space. Given $p \in M$, let
$(Y,\omega_Y,\Phi_Y)$ be the local model of $p$. There exist
$T$-invariant open neighborhoods $U\subset 
M$ of $p$ and $V \subset Y$ of $[1,0,0]$ and an isomorphism between
$(U,\omega,\Phi)$ and $(V,\omega_Y,\Phi_Y + \Phi(p))$ that maps $p$ to $[1,0,0]$.
\end{thm}

\begin{corollary}\label{cor:effective}
  Let $\comp$ be a Hamiltonian $T$-space. For any $p \in M$ with
  stabilizer $H$, the
  homomorphism $\rho : H \to (S^1)^{h+k}$ is injective.
\end{corollary}

\begin{proof}
  Fix $p \in M$ and notation as above. Let $(Y,\omega_Y,\Phi_Y)$ be
  the local model of $p$. Since the $T$-action on $M$ is assumed
  to be effective, so is the $T$-action on $Y$ by Theorem \ref{local
    normal form} (see Remark \ref{rmk:equivariant}). By definition of
  $Y$ and of the $T$-action on $Y$, the $T$-action on
  $Y = T \times_H \mathrm{Ann}(\mathfrak{h}) \times \C^{h+k}$ is
  effective if and only if the Lie group homomorphism $\rho : H \to
  (S^1)^{h+k}$ is injective, as desired.
\end{proof}

\begin{remark}\label{rmk local normal form}
  Given $p \in M$, let $H$ be its stabilizer and let
  $\{\alpha_j\}$ be the multiset of isotropy weights of $p$. By
  Corollary \ref{cor:effective}, the $H$-action on $(T_p \mathcal{O})^{\omega}/T_p\mathcal{O}$ is effective. In particular, so is the action
  by its identity component. Using equation \eqref{action identity component}, it follows that the $\Z$-span of $\{\alpha_j\}$ equals
  $\ell^*_{\mathfrak{h}}$.
\end{remark}

\begin{remark}\label{rmk:fixed_point}
  The above discussion simplifies significantly if $p$ is a {\bf fixed
    point}, i.e., if $H=T$. In this case, $h =d$ so that
  $h+k = n$, and the isotropy
  weights $\alpha_1,\ldots , \alpha_n$ lie in $\ell^*$. Moreover,
  $Y = \C^n$, the symplectic form $\omega_Y$ is the standard
  symplectic form on $\C^n$, the $T$-action on $Y$ is given by
  \begin{equation}\label{action Cn}
    \exp(\xi)\cdot (z_1,\ldots,z_n)=(e^{ 2\pi i \langle \alpha_1 , \xi
    \rangle}z_1,\ldots,
    e^{ 2\pi i \langle \alpha_n , \xi \rangle}z_n)\quad\text{for 
      every}\;\;\xi\in \g\, ,
  \end{equation}
  \noindent
  and the moment map $\Phi_Y : Y \to \g^*$ is given by
  \begin{equation}
    \label{eq:33}
    \Phi_Y(z)=\pi \sum_{j=1}^n \alpha_j |z_j|^2,
  \end{equation}
  \noindent
  where $z = (z_1,\ldots,z_n) \in \C^n$.
\end{remark}

\subsubsection{Regular and exceptional local models}\label{sec:exceptional-orbits}
Theorem \ref{local normal form} allows us to understand a Hamiltonian $T$-space in a $T$-invariant open neighborhood of a
point $p$ by studying the local model of $p$. For this reason, in this subsection we
take a closer look at local models. \\

In what follows, we fix a non-negative integer $k$ and a closed subgroup $H \leq T$. As above, we set $d:= \dim T$ and $h:=\dim H$. We also fix
an injective Lie group homomorphism $\rho : H
\hookrightarrow (S^1)^{h+k}$. We denote the subspace
of fixed points of the $H$-action induced by $\rho$ by $(\C^{h+k})^H$. As in Section
\ref{sec:local-model-local}, we use $k$, $H$ and $\rho$ to construct
Hamiltonian $T$-spaces
$$((T^*T
\times \C^{h+k}) /\! /H, \omega_{\mathrm{red}}, \Phi_{\mathrm{red}})
\cong (Y,\omega_Y,\Phi_Y)$$
that we refer to as the {\bf
  local model determined by $\boldsymbol{k}$, $\boldsymbol{H}$ and
  $\boldsymbol{\rho}$} (see equations \eqref{eq:47}, \eqref{eq:48} and
\eqref{eq:43} for the definition of $Y$, the $T$-action and $\Phi_Y$
respectively). We remark that $k$ is the complexity of $
(Y,\omega_Y,\Phi_Y)$, that $H$ is the stabilizer of $p:=[1,0,0]$ and
that $\rho$ is the symplectic slice representation of $p$.

\begin{lemma}\label{lemma:max_dim_rep}
  Let $k$ be a non-negative integer, let $H \leq T$
  be a closed subgroup and let $\rho : H \hookrightarrow (S^1)^{h+k}$
  be an injective Lie group homomorphism, where $h = \dim H$. Set $s:=
  \dim_{\C} (\C^{h+k})^H$. There exists an isomorphism of
  Hermitian vector spaces $\C^{h+k} \simeq \C^{h+k-s}\oplus \C^s$
  such that
  \begin{equation}
    \label{eq:50}
    \rho = (\rho',1) : H \hookrightarrow 
    (S^1)^{h+k-s} \times 1 \hookrightarrow (S^1)^{h+k-s} \times (S^1)^s \simeq
    (S^1)^{h+k}
  \end{equation}
  \noindent
  and $(\C^{h+k-s})^H = \{0\}$. Moreover, $s \leq k$ and, if $s =
  k$, then $\rho$ is an isomorphism between $H$ and $(S^1)^h$ and
  the $H$-action on $\C^h$ is toric.
\end{lemma}

\begin{proof}
  To simplify notation, set $V :=  (\C^{h+k})^H$. The standard
  Hermitian product on $\C^{h+k}$ induces an isomorphism of Hermitian
  vector spaces $\C^{h+k} \simeq V^{\perp} \oplus V$, and both $V$ and
  $V^{\perp}$ are endowed with the restriction of the standard
  Hermitian product. By definition $\rho(H)$ fixes $V$ pointwise and
  is a subgroup of $U(h+k)$. Hence, $\rho$ splits as the
  direct sum of two $H$-representations that we denote by
  $\rho_V : H \to U(V)$ and $\rho': H
  \to U(V^{\perp})$. By construction,
  $\rho_V$ is the trivial representation. Moreover,
  $\rho'(H)$ is contained in a maximal torus of $U(V^{\perp})$ (see
  Section \ref{sec:local-model-local}). Since the Hermitian vector
  spaces $V$ and $V^{\perp}$ are isomorphic to $\C^s$ and to
  $\C^{h+k-s}$ respectively and since $\{0\}= (\C^{h+k-s})^H$, this proves the first statement. To prove
  the second, we observe that $\rho': H \to U(V^{\perp})$ is injective, since $\rho$ is injective. Since maximal tori in $U(V^{\perp})$
  have dimension equal to $h+k-s$ and
  since $h = \dim H$, it follows at once that $s\leq k$. Finally, if $s = k$,
  then, since $\rho'$ is injective, it follows that the map $H
  \hookrightarrow (S^1)^h$ is an isomorphism of Lie groups.
\end{proof}

Lemma \ref{lemma:max_dim_rep} allows us to `decompose' local models as
follows (see \cite[Lemma 4.4]{kt1} for a proof in the case $k=s=1$).

\begin{proposition}\label{prop:decom_local_model}
  Let $k$ be a non-negative integer, let $H \leq T$
  be a closed subgroup and let $\rho : H \hookrightarrow (S^1)^{h+k}$
  be an injective Lie group homomorphism, where $h = \dim H$. Let $s
   = \dim_{\C}(\C^{h+k})^H$ and $\rho' : H \hookrightarrow (S^1)^{h+k-s}$ be as
  in the statement of Lemma \ref{lemma:max_dim_rep}, so that
  $(\C^{h+k-s})^H=\{0\}$. The local
  model determined by $k$, $H$ and $\rho$ is isomorphic to
  $$ (Y' \times \C^s,\mathrm{pr}^*_1\omega_{Y'} + \mathrm{pr}^*_2\omega_0, \mathrm{pr}_1^*\Phi_{Y'}), $$
  \noindent
  where $(Y',\omega_{Y'},\Phi_{Y'})$ is the local model determined by
  $k-s$, $H$ and $\rho'$, $\omega_0$ is the standard symplectic form
  $\C^s$, and $\mathrm{pr}_j$ is the projection from $Y' \times \C^s$
  to its $j$th component, for $j=1,2$.
\end{proposition}

\begin{proof}
  In this proof we use the abstract construction of local models (see
  Section \ref{sec:local-model-local} and Remark
  \ref{rmk:local_model_two_interpretations}). Fix an isomorphism
  $\C^{h+k} \simeq \C^s \oplus \C^{h+k-s}$ as in the statement of
  Lemma \ref{lemma:max_dim_rep}. This induces a symplectomorphism between $(T^*T \times \C^{h+k}, \Omega)$
  and $((T^*T \times \C^{h+k-s}) \times \C^s,\mathrm{pr}^*_1\Omega' +
  \mathrm{pr}^*\omega_0)$, where $\Omega'$ denotes the symplectic form
  on $T^*T \times \C^{h+k-s}$ and, as above, $\mathrm{pr}_j$ is the projection from $(T^*T \times \C^{h+k-s}) \times \C^s$
  to its $j$th component, for $j=1,2$. We endow $(T^*T \times
  \C^{h+k-s}) \times \C^s$ with the following (right) $H$-action and
  (left) $T$-action:
  \begin{itemize}[leftmargin=*]
  \item ($H$-action): On $T^*T \times \C^{h+k-s}$ we consider the (right) $H$-action used
    to construct the local model determined by $k-s$, $H$ and $\rho'$,
    while $H$ acts trivially on $\C^s$, and
  \item ($T$-action): On $T^*T \times \C^{h+k-s}$ we consider the (left) $T$-action used
    to construct the local model determined by $k-s$, $H$ and $\rho'$,
    while $T$ acts trivially on $\C^s$.
  \end{itemize}
  The above symplectomorphism is both $H$- and $T$-equivariant. Hence,
  there is a $T$-equivariant symplectomorphism between the symplectic
  quotients with respect to the $H$-actions, i.e., an isomorphism
  between the resulting Hamiltonian $T$-spaces. Once the abstract
  constructions are identified with the explicit ones as in Section
  \ref{sec:local-model-local}, this yields the desired isomorphism of
  Hamiltonian $T$-spaces.
\end{proof}

\begin{remark}\label{rmk:regular_local_model}
  Proposition \ref{prop:decom_local_model} is particularly simple if
  $s = k$, i.e., if $\dim_{\C} (\C^{h+k})^H$ is maximal. In this case, $(Y', \omega', \Phi_{Y'})$ is
  {\em toric}.
\end{remark}

Motivated by Theorem \ref{local normal form} and by Remark \ref{rmk:regular_local_model}, we introduce the
following dichotomy.

\begin{definition}\label{defn:regular_exceptional}
  Let $\comp$ be a complexity $k$ $T$-space and let $p \in M$ be a
  point with stabilizer $H$. Let $\mathcal{O}$ be the orbit containing
  $p$. We say that $p$ is {\bf regular} if
  $\dim_{\C}(\C^{h+k})^H = k$ and {\bf exceptional} otherwise, where
  $H$ acts on $(T_p\mathcal{O})^{\omega}/T_p \mathcal{O} \simeq \C^{h+k}$ via the symplectic slice representation at $p$.
\end{definition}

\begin{remark}\label{rmk:exc_orbits}
  We observe that Definition \ref{defn:regular_exceptional} extends
  also to orbits (see Remark \ref{rmk:symp_slice_orbit}); hence,
  throughout the paper, we also refer to {\bf regular} and {\bf
    exceptional} orbits.
\end{remark}

\begin{definition}
  Let $\comp$ be a Hamiltonian $T$-space. We define the set of {\bf
    exceptional orbits} $M_{\mathrm{exc}}$ as the subset of $M/T$
  consisting of exceptional orbits.
\end{definition}

\begin{example}\label{exm:fixed_point_reg}
  A fixed point $p$ in a complexity $k$ $T$-space $\comp$ is regular if
  and only if it lies on a fixed submanifold of dimension $k$. We observe that the latter condition is
  equivalent to the $T$-action on the normal bundle to the fixed
  submanifold that contains $p$ being toric. 
\end{example}

The techniques of Example \ref{exm:fixed_point_reg} and Lemma
\ref{lemma:max_dim_rep}, together with Theorem \ref{local normal form}, can be used to prove the following result.

\begin{lemma}\label{lemma:regular_connected_stab}
  \mbox{}
  \begin{itemize}[leftmargin=*]
  \item Any point with
    trivial stabilizer, as well as any point in a complexity zero
    $T$-space, is regular.
  \item If a Hamiltonian $T$-space has positive complexity, then any isolated fixed
    point is exceptional.
  \item The subset of regular points in a Hamiltonian $T$-space is open.
  \item If a point $p$ is regular, then the stabilizer of $p$ is connected. 
  \end{itemize}
\end{lemma}

We conclude this section with two results in complexity one. 

\begin{lemma}\label{lemma:fixed_point_exceptional}
  Let $\comp$ be a complexity one $T$-space. A point $p \in M^T$ is
  isolated if and only if it is exceptional.
\end{lemma}
\begin{proof}
  By Lemma \ref{lemma:regular_connected_stab}, if $p \in M^T$ is
  isolated, then it is exceptional. Conversely, if $p \in M^T$
  is regular, then, by Theorem \ref{local normal form}
  and Example \ref{exm:fixed_point_reg}, $p$ is not isolated.
\end{proof}

Finally, as an immediate consequence of Theorem \ref{local normal
  form} and of Proposition \ref{prop:decom_local_model}, the following
characterization of exceptional points holds (cf. \cite[Lemma 4.4]{kt1}).

\begin{corollary}\label{cor:exc=exc}
  Let $\comp$ be a complexity one $T$-space. A point $p \in M$ is
  exceptional if and only if every nearby orbit in the
  same moment fiber has a strictly smaller stabilizer.
\end{corollary}

In particular, Definition \ref{defn:regular_exceptional} agrees with
the definition of exceptional orbits of \cite{kt1} in complexity one,
and is the appropriate notion for our purposes.

\subsubsection{Regular and exceptional sheets}\label{sec:exceptional-sheets}
Let $\comp$ be a Hamiltonian $T$-space. For any closed subgroup $H \leq T$, the action of $H$ on $M$ is also
Hamiltonian: Let
$\mathfrak{h}$ be the Lie algebra of $H$ and let $i: \mathfrak{h} \hookrightarrow \g$ denote the
inclusion. Then a moment
map for the $H$-action is given by the composition 
\begin{equation}\label{mm subgroup}
M \stackrel{\Phi}{\longrightarrow} \g^* \stackrel{i^*}{\longrightarrow} \mathfrak{h}^*, 
\end{equation}
where $i^*\colon \g^* \to \mathfrak{h}^*$ is the dual of $i$. We 
denote the set of fixed points of the $H$-action by
$$M^H = \{ p \in M \mid h \cdot p = p \, \text{ for all } h \in
H\}.$$
\noindent
The following result is used throughout and a proof can be found in \cite[Lemma 5.53]{mcduff-salamon}.

\begin{lemma}\label{Lem:fixedSubmainfolds} 
  Let $\comp$ be a Hamiltonian $T$-space and let $H \leq T$ be
  closed. Then any connected component of $M^H$ is an embedded
  symplectic submanifold of $(M,\omega)$.
\end{lemma}

We refer to the connected components of $M^T$ as the \textbf{fixed submanifolds} of $\comp$.

\begin{rmk}\label{rmk weights of F}
  If $N\subset M^T$ is a fixed submanifold, then for any $p,p'\in N$,
  the isotropy weights of $p$ are equal to those at $p'$. Hence, the 
  \textbf{isotropy weights of the fixed submanifold $\boldsymbol{N}$}
  are well-defined. Moreover, they can be used to determine $\dim N$: By Theorem \ref{local normal form},
  $\dim N$ equals twice the number of isotropy weights of $N$ that are
  equal to 0.
\end{rmk}

Let $H \leq T$ be a closed subgroup that is the stabilizer of a point
$p \in M$ and let $N \subset M^H$ denote the connected component that
contains $p$. We set $\omega_N:= \omega|_N$. By Lemma
\ref{Lem:fixedSubmainfolds}, $(N,\omega_N)$ is an embedded 
symplectic submanifold of $(M,\omega)$. Moreover, since $T$ is
abelian and connected, and since $N$ is connected, $N$ is
$T$-invariant. The $T$-action on $N$ induces a $T':=T/H$-action on $N$
that is effective because the stabilizer of $p$ for the $T$-action
equals $H$. This $T'$-action on $N$ is Hamiltonian. To see this, we
construct a moment map starting from $\Phi$ and the point $p$. To this
end, let $\mathrm{pr} : T \to T'$ be the quotient map. We denote its
derivative at the identity by $\mathrm{pr}_*: \g \to \g'$ and the dual
of $\mathrm{pr}_*$ by $\mathrm{pr}^* : (\g')^* \to \g^*$. By
construction, $\mathrm{pr}^*$ is an isomorphism between $(\g')^*$ and
$\mathrm{Ann}(\mathfrak{h})$. Since $N \subset M^H$ is connected and
since $\Phi$ is a moment map, $\langle \Phi|_N,\xi
\rangle = \langle \Phi(p), \xi \rangle$ for all $\xi \in
\mathfrak{h}$. Consequently, $\Phi(N) \subset \Phi(p) +
\mathrm{Ann}(\mathfrak{h})$. Hence, there exists a unique smooth map
$\Phi_N : N \to (\g')^*$ such that
\begin{equation}
  \label{eq:44}
  \mathrm{pr}^* \circ \Phi_N = \Phi|_N - \Phi(p). 
\end{equation}
\noindent
By \eqref{eq:44}, since $\Phi : M \to \g^*$ is
a moment map for the $T$-action on $M$ and since
$(N,\omega_N)$ is a symplectic submanifold of $(M,\omega)$, it follows shows that $\Phi_N$ is a
moment map for the $T'$-action on $N$. Since the $T'$-action on $N$ is
effective, this shows that
$(N,\omega_N,\Phi_N)$ is a Hamiltonian $T'$-space.

\begin{definition}\label{defn:sheet}
  Let $\comp$ be a Hamiltonian $T$-space and let $H \leq T$ be
  a subgroup that occurs as a stabilizer of some point $p \in M$. The
  Hamiltonian $T'$-space $(N,\omega_N,\Phi_N)$ constructed above is
  called a {\bf sheet stabilized by $\boldsymbol{H}$}. Whenever we wish to emphasize the role of $p$,
  we say that $(N,\omega_N,\Phi_N)$
  is the {\bf sheet through $\boldsymbol{p}$}.
\end{definition}

Any fixed submanifold $N \subset M^T$ gives rise to a sheet that we
denote simply by $(N,\omega_N)$. 

\begin{remark}\label{rmk:convention_moment_map_sheet}
  For our purposes, it is useful to allow some freedom in the choice
  of moment map associated to a sheet. Modulo $\mathrm{pr}^*$, $\Phi_N$ and
  $\Phi|_N$ only differ by a constant. Depending on the context, we
  use either moment map. We trust that this does
  not cause confusion.
\end{remark}

Given a sheet $(N,\omega_N,\Phi_N)$ of $\comp$, we are interested in understanding
the complexity of $(N,\omega_N,\Phi_N)$ in relation to that of
$\comp$. By Definition \ref{defn:complexity}, the complexity of
$(N,\omega_N,\Phi_N)$ is $  \frac{1}{2}\dim N - (d-h)$.

\begin{proposition}\label{prop:regular_exceptional}
  Let $\comp$ be a Hamiltonian $T$-space and let $(N,\omega_N,\Phi_N)$
  be a sheet stabilized by $H \leq T$. The complexity of $(N,\omega_N,\Phi_N)$ is at most
  that of $\comp$. Moreover, the following
  are equivalent:
  \begin{enumerate}[label=(\arabic*),ref=(\arabic*),leftmargin=*]
  \item \label{item:21} The complexity of $(N,\omega_N,\Phi_N)$ is less than
    that of $\comp$.
  \item \label{item:24} Any $p \in N$ with stabilizer $H$ is exceptional.
  \item \label{item:26} The fiberwise $H$-action on the symplectic normal bundle to $N$ has
    positive complexity.
  \end{enumerate}
\end{proposition}

\begin{proof}
  Let $p \in N$ be a point that is stabilized by $H$ and let
  $(Y,\omega_Y,\Phi_Y)$ be the local model of $p$. Since the above
  conditions can be checked locally, by Theorem
  \ref{local normal form}, it suffices to prove the result in the case
  $\comp = (Y,\omega_Y,\Phi_Y)$ and $p = [1,0,0]$. Recall that $Y = T \times_H
  \mathrm{Ann}(\mathfrak{h}) \times \C^{h+k}$, where $k$ is the
  complexity of $(Y,\omega_Y,\Phi_Y)$, and that the $T$-action on $Y$ is given by
  equation \eqref{eq:48}. The submanifold $Y^H$ is given by
  \begin{equation}
    \label{eq:52}
     T \times_H \mathrm{Ann}(\mathfrak{h}) \times (\C^{h+k})^H \simeq
     T/H \times \mathrm{Ann}(\mathfrak{h}) \times (\C^{h+k})^H, 
  \end{equation}
  \noindent
  where $(\C^{h+k})^H$ is the fixed point set for the $H$-action on
  $\C^{h+k}$. Since $(\C^{h+k})^H$ is a subspace of $\C^{h+k}$, it
  follows that $Y^H$ is connected. Since $N \subseteq Y^H$ is a
  connected component of $Y^H$, $N = Y^H$. Moreover, by
  \eqref{eq:52}, the dimension of $N$ equals $2(d-h) +
  2\dim_{\C}(\C^{h+k})^H$. By Lemma \ref{lemma:max_dim_rep},
  $\dim_{\C}(\C^{h+k})^H \leq k$. Hence, the complexity of
  $(N,\omega_N,\Phi_N)$ is at most $k$. Moreover, since
  $(\C^{h+k})^H$ is a complex subspace of $\C^{h+k}$, the $H$-action
  on the symplectic
  normal vector space to $N$ at $p$ can be identified with the linear
  $H$-action on $\C^{h+k}/(\C^{h+k})^H$.

  Suppose that the complexity of $(N,\omega_N,\Phi_N)$ is less than that of $(Y,\omega_Y,\Phi_Y)$. Since $\dim N
  = 2(d-h) +
  2\dim_{\C}(\C^{h+k})^H$, it follows that $\dim_{\C}(\C^{h+k})^H<k$, so that
  $p$ is exceptional. Therefore, \ref{item:21} implies \ref{item:24}. If $p$ is exceptional, so that $\dim_{\C}(\C^{h+k})^H<k$, then $\dim_{\C}
  \C^{h+k}/(\C^{h+k})^H > h$. Hence, the linear $H$-action on
  $\C^{h+k}/(\C^{h+k})^H$ has positive
  complexity. Therefore, \ref{item:24} implies \ref{item:26}. Finally,
  if the $H$-action on the symplectic normal vector space to $N$ at
  $p$ has positive complexity, then $\dim_{\C}
  \C^{h+k}/(\C^{h+k})^H > h$. The latter is equivalent to $\dim_{\C}(\C^{h+k})^H<k$, so that $\dim N
  < 2(d-h) + 2k$. Hence, the complexity of $(N,\omega_N,\Phi_N)$ is less than
  that of $(Y,\omega_Y,\Phi_Y)$. Therefore,
  \ref{item:26} implies \ref{item:21}.
\end{proof}

\begin{remark}\label{rmk:regular_equivalent}
  The proof of Proposition \ref{prop:regular_exceptional} can be
  adapted to prove equivalence of the following conditions:
  \begin{enumerate}[label=(\arabic*'),ref=(\arabic*'),leftmargin=*]
  \item \label{item:27} The complexity of $(N,\omega_N,\Phi_N)$ is
    equal to that of $\comp$.
  \item \label{item:28} Any $p \in N$ with stabilizer $H$ is regular.
  \item \label{item:29} The fiberwise $H$-action on the symplectic normal bundle to $N$ has
    complexity zero.
  \end{enumerate}
  We observe that, by Proposition \ref{prop:regular_exceptional},
  either \ref{item:21} or \ref{item:27} needs to hold for any sheet.
\end{remark}

\begin{remark}\label{rmk:reg_irr_fixed}
  If $(N,\omega_N)$ is a fixed submanifold in $\comp$, properties \ref{item:21} and
  \ref{item:24} of Proposition \ref{prop:regular_exceptional}
  (respectively properties \ref{item:27} and \ref{item:28} of Remark \ref{rmk:regular_equivalent}) simplify as follows: The
  dimension of $N$ is less than (respectively equal to) twice the
  complexity of $\comp$, and {\em all} points in $N$ are exceptional
  (respectively regular). Moreover, $\dim N \leq 2k$,
  where $k$ is the complexity of $\comp$.
\end{remark}

Motivated by Proposition \ref{prop:regular_exceptional} and Remark \ref{rmk:regular_equivalent}, we introduce
the following terminology.

\begin{definition}\label{defn:reg_irreg}
  A sheet $(N,\omega_N,\Phi_N)$ in $\comp$ is {\bf exceptional}
  if it satisfies any of the conditions
  of Proposition \ref{prop:regular_exceptional}, and {\bf regular} otherwise.
\end{definition}

Exceptional sheets enjoy the following stronger characterization.

\begin{lemma}\label{lemma:exceptional}
  A sheet $(N,\omega_N,\Phi_N)$ in $\comp$ is exceptional if and only if
  every point in $N$ is exceptional.
\end{lemma}

\begin{proof}
  If every point in $N$ is exceptional, then
  $(N,\omega_N,\Phi_N)$ is exceptional by Proposition
  \ref{prop:regular_exceptional}. Conversely, suppose that
  $(N,\omega_N,\Phi_N)$ is exceptional. By contradiction, suppose that $p \in N$ is regular. By Lemma \ref{lemma:regular_connected_stab}, every point in
  $N$ that is sufficiently close to $p$ is also regular. By the
  principal orbit theorem (see \cite[Theorem 2.8.5]{dk}), the set
  of points in $N$ that is stabilized by $H$ is dense. Hence, there exist
  points in $N$ that are stabilized by $H$ that are arbitrarily close
  to $p$. However, by Proposition \ref{prop:regular_exceptional}, any
  such point is exceptional, a contradiction.
\end{proof}

\begin{remark}\label{rmk:not_exceptional_sheet}
  It is not necessarily true that, if $(N,\omega_N,\Phi_N)$ is a
  regular sheet, then all points in $N$ are regular. A counterexample is as follows: Let $\comp$ be a Hamiltonian
  $T$-space of positive complexity that contains an isolated fixed
  point $p$. Since $T$ is compact and
  abelian, and
  since $M$ is connected, the
  principal orbit theorem (see \cite[Theorem 2.8.5]{dk}) implies that
  there exists an open and dense subset of $M$ whose points have
  trivial stabilizer. Hence, taking $H = \{e\}$, it follows that
  $\comp$ is a regular sheet. However, by Lemma \ref{lemma:regular_connected_stab}, $p$ is exceptional.
\end{remark}

\subsection{Invariants of compact Hamiltonian $T$-spaces}\label{sec:global-invariants}
In this section we recall some fundamental results about {\em
  compact}\footnote{In fact, many of the theorems presented in this
  section hold under the weaker assumption that the moment map be
  proper as a map to a convex open subset of $\g^*$. However, this
  degree of generality goes beyond the scope of this paper.}
Hamiltonian $T$-spaces.

\subsubsection{Convexity package and its consequences}\label{sec:conn-fibers-conv}
We start with the following foundational result  (see
\cite{atiyah,gs,Sjamaar}).

\begin{thm}[{\bf Convexity Package}]\label{thm:con_pac}
  Let $\comp$ be a compact Hamiltonian $T$-space.
  \begin{enumerate}[label=(\roman*), ref=(\roman*), leftmargin=*]
  \item\label{item:3} {\bf (Connectedness)} The fibers of the moment map are connected.
  \item\label{item:4} {\bf (Stability)} The moment map is open as a map onto its
    image.
  \item\label{item:10} {\bf (Convexity)} The moment map image is the convex hull of the images of 
    the fixed submanifolds.
  \end{enumerate}
\end{thm}

We remark that, since the action is effective, the moment map
image of a compact Hamiltonian $T$-space is a polytope that has
dimension equal to $\g^*$.

\begin{definition}\label{defn:mom_polytope}
  Let $\comp$ be a compact Hamiltonian $T$-space. The image $\Phi(M)$
  is called the {\bf moment polytope}.
\end{definition}

By Theorem \ref{thm:con_pac}, the moment polytope of a compact
Hamiltonian $T$-space is convex. In fact, more is true and, in order
to prove this, we need to recall a few notions. We say that a polytope
$\Delta \subset \g^*$
is {\bf rational} if any edge $e \subset \Delta$ is of the form
\begin{equation}
  \label{eq:17}
   e = \{v + t \alpha \mid t \in [0,l]\} \quad \text{for some } v \in
   \g^*, \, \alpha \in \ell^* \text{ and } l \in \R_{>0}.
\end{equation}
\noindent
Moreover, a
subset $C \subseteq \g^*$ is a {\bf cone} if, for all $v \in C$
and all $\lambda \in \R_{\geq 0}$, $\lambda v \in C$. A cone in $\g^*$ is {\bf proper} if it does not contain any subspace of $\g^*$ of
positive dimension. The following result provides a local description
of the moment polytope of a compact Hamiltonian $T$-space near the
image of a fixed submanifold.

\begin{corollary}\label{Cor:ConeFace}
Let $\comp$ be a compact Hamiltonian $T$-space of dimension $2n$. Let $N$ be a fixed submanifold and let $\alpha_{1},\dots 
\alpha_{n}$ be the isotropy weights of $N$ (see Remark \ref{rmk weights of F}). Consider the cone 
\begin{equation}\label{cone def}
  \mathcal{C}_N = \R_{\geq 0}\text{-span} \left\lbrace
    \alpha_{1},\dots \alpha_{n}\right\rbrace \subseteq \g^*, 
\end{equation}
\noindent
and let $\mathcal{H}_N \subseteq \g^*$ be the maximal subspace that is contained in $\mathcal{C}_N$.

\begin{enumerate}[label=(\roman*),ref=(\roman*), leftmargin=*]
\item \label{item:1} There exist an open neighborhood $V$ of $\Phi(N)$ in $\Phi(M)$
  and an open neighborhood $W$ of $0$ in $\mathcal{C}_N$ such that $V
  = W + \Phi(N)$. In particular, $\Phi(M)$ is rational.
\item \label{item:2} The intersection
  \begin{align*}
    \left( \mathcal{H}_N + \Phi(N)\right) \cap \Phi(M)
  \end{align*}
  is a face of $\Phi(M)$ and the dimension of this
  face equals the dimension of \(\mathcal{H}_N\).
\end{enumerate}
In particular, \(\Phi(N)\) is a vertex of \(\Phi(M)\) if and only if the cone \(\mathcal{C}_N\) is proper.
\end{corollary}

\begin{figure}[h]
	\begin{center}
		\includegraphics[width=7cm]{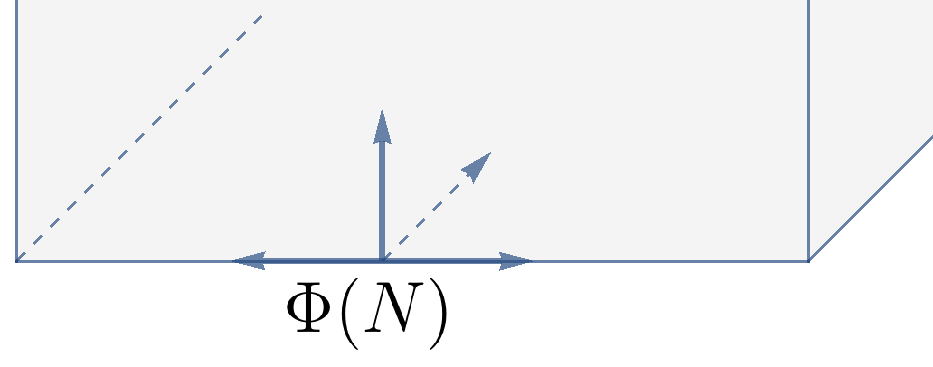}
		\caption{The arrows are the non-zero isotropy weights of the fixed
                  submanifold $N$, the shaded area equals
                  $\mathcal{C}_N$ and the straight line spanned by the
                  collinear isotropy weights is $\mathcal{H}_N$.}
		\label{Edge}
	\end{center}
\end{figure}
\begin{proof}
  For any $p \in N$, by Theorem \ref{local normal
    form}, there exist a $T$-invariant open 
  neighborhood $U_p$ of $p$ and an open neighborhood $W_p$ of $0$ in
  $\mathcal{C}_N$ such that $\Phi(U_p) = W_p + \Phi(N)$. Moreover, by
  Theorem \ref{thm:con_pac}, $\Phi(U_p)$ is an open neighborhood of
  $\Phi(p)$ in $\Phi(M)$. Since $N$ is compact, there exist finitely
  many $p_1,\ldots, p_r \in N$ such that $N$ is contained in
  $\bigcup_{j=1}^r U_{p_j}$. Set $W:= \bigcap_{j=1}^r W_{p_j}$. By
  construction, $W + \Phi(N) = \bigcap_{j=1}^r \Phi(U_{p_j})$ is an open neighborhood of $\Phi(N)$
  in $\Phi(M)$. Since the cone $\mathcal{C}_N$ is convex and rational, and
  since vertices of the moment polytope are the image of fixed submanifolds
  by Theorem \ref{thm:con_pac}, the moment polytope is rational. This proves part \ref{item:1}.

  To prove part \ref{item:2}, observe that any cone is the product of
  the maximal subspace that it contains with a proper cone. Thus we
  can write $\mathcal{C}_N = \mathcal{H}_N \times
  \mathcal{C}'_N$ for some proper cone $\mathcal{C}'_N$. Hence, without loss of generality, we may take an
  open neighborhood $W$ in $\mathcal{C}_N$ 
  as in the statement of part \ref{item:1} to be of the form
  $W_{\mathcal{H}} \times W'$, where $W_{\mathcal{H}}$ (respectively $W'$) is an open
  neighborhood of $0$ in $\mathcal{H}_N$ (respectively
  $\mathcal{C}'_N$). Therefore, the desired result holds `locally' by
  part \ref{item:1}; convexity of
  $\Phi(M)$ (see Theorem \ref{thm:con_pac}) implies that it is true `globally'.
\end{proof}

Until the end of the section, we deduce some consequences of Theorem
\ref{thm:con_pac} that we use throughout the paper. We start with the following sufficient condition
for a sheet to be exceptional (see Definition \ref{defn:reg_irreg}).

\begin{lemma}\label{lemma:exceptional_not_boundary}
  Let $\comp$ be a compact Hamiltonian $T$-space and let
  $(N,\omega_N,\Phi_N)$ be a sheet stabilized by a non-trivial
  subgroup $H$. If $\Phi(N)$ is not contained in
  the boundary of $\Phi(M)$, then $(N,\omega_N,\Phi_N)$ is exceptional.
\end{lemma}

\begin{proof}
  We prove the contrapositive. Let $(N,\omega_N,\Phi_N)$ be a regular sheet. It suffices
  to show that the set of regular values of $\Phi|_N$ that are
  contained in $\Phi(N)$ is contained in the boundary of
  $\Phi(M)$. Let $x \in \Phi(N)$ be a regular value for $\Phi|_N$. By
  Theorem \ref{local normal form},
  there exists $p \in \Phi|_N^{-1}(x)$ that has trivial stabilizer
  for the $T/H$-action, i.e., the stabilizer of $p$ for the
  $T$-action is $H$. By Remark \ref{rmk:regular_equivalent}, $p$ is
  regular and, hence, the $T/H$-action on the symplectic normal bundle
  to $N$ is toric. By the local normal form (Theorem \ref{local normal
    form}), and by openness of the moment map (Theorem
  \ref{thm:con_pac}), it follows that $\Phi(x)$ lies in the boundary
  of $\Phi(M)$.
\end{proof}

For compact Hamiltonian $T$-spaces, the existence of exceptional sheets
is intimately connected to the existence of exceptional fixed points. More precisely, the following holds. 

\begin{lemma}\label{lemma:necessary_sufficient_existence_exceptional}
  A compact Hamiltonian $T$-space $\comp$ contains an exceptional sheet if and
  only if it contains an exceptional fixed point. 
\end{lemma}

\begin{proof}
  Suppose first that $(N,\omega_N,\Phi_N)$ is an exceptional sheet of
  $\comp$. Since $N$ is compact, it contains a fixed point $p \in
  M^T$. By Lemma \ref{lemma:exceptional}, $p$ is exceptional. Conversely, suppose that $p \in
  M^T$ is exceptional. By definition, the sheet
  $(N,\omega_N,\Phi_N)$ through $p$ is exceptional. 
\end{proof}

Next we deduce some general results about isotropy weights of isolated
fixed points that are used in one of the key results of our paper,
Proposition \ref{prop:iso_fixed_pts_monotone}.

\begin{lemma}\label{lemma:sheet_weight}
  Let $\comp$ be a compact complexity $k$ $T$-space and let $p \in
  M^T$ be isolated. For any isotropy weight $\alpha$ of $p$,
  there exists a sheet $(N_{\alpha},\omega_{\alpha},\Phi_{\alpha})$ with the following
  properties:
  \begin{itemize}[leftmargin=*]
  \item the point $p$ lies in $N_{\alpha}$,
  \item the sheet $(N_{\alpha},\omega_{\alpha},\Phi_{\alpha})$ is
    stabilized by the codimension 1 subgroup
    \begin{equation}
      \label{eq:36}
      H_{\alpha} :=\exp(\{\xi \in \mathfrak{t} \mid \langle \alpha,\xi \rangle
      \in \mathbb{Z} \}), 
    \end{equation}
  \item the dimension of $N_{\alpha}$ is at most $2(k+1)$, 
  \item the moment map image $\Phi_{\alpha}(N_{\alpha})$ is contained
    in the affine line $\Phi(p) + \R \langle \alpha \rangle$ and intersects the open
    half-ray $\Phi(p) + \R_{>0} \langle \alpha \rangle$, and
  \item there exists $q_{\alpha} \in M^T \cap N_{\alpha}$ such that $\Phi(q_{\alpha}) =
    \Phi_{\alpha}(q_{\alpha})$ is a global extremum of $\Phi_{\alpha}(N_{\alpha})$ with $\Phi(q_{\alpha})
    \in \Phi(p) + \R_{>0} \langle \alpha \rangle$, and $-\alpha$ is
    an isotropy weight of $q_{\alpha}$.
  \end{itemize}
\end{lemma}

\begin{proof}
  Let $\alpha_1,\ldots, \alpha_n
  \in \ell^*$ be the isotropy weights of $p$. Without loss of generality, we may assume that $\alpha=\alpha_n$. By the local normal form
  of Theorem \ref{local normal form}, we may identify a $T$-invariant
  open neighborhood
  of $p$ in $M$ with a $T$-invariant open neighborhood of $0 \in \C^n$
  so that the action becomes that of \eqref{action Cn} and
  the moment map is given by \eqref{eq:33}. Since $p \in M^T$ is
  isolated, $\alpha_j \neq 0$ for all $j$. Hence, since $\Z\langle
  \alpha_1,\ldots, \alpha_n \rangle = \ell^*$, there can be at most
  $k$ isotropy weights that are multiples of
  $\alpha_n$. Therefore the subgroup $H_{\alpha}$ of \eqref{eq:36}
  stabilizes a subspace of $\C^n$ that is of real dimension at most
  $2(k+1)$. Moreover, by \eqref{action Cn}, the subgroup $H_{\alpha}$
  is the stabilizer of some point $p'\in M$. The
  sheet $(N,\omega_N,\Phi_N)$ through $p'$ in the sense of Definition
  \ref{defn:sheet} satisfies the desired conditions.
\end{proof}

For our purposes, it is useful to introduce the following
terminology.

\begin{definition}\label{defn:sheet_along_weight}
  Let $\comp$ be a compact Hamiltonian $T$-space and let $p \in M^T$ be
  isolated. Given an isotropy weight $\alpha$ of $p$, we
  say that the sheet $(N_{\alpha},\omega_{\alpha},\Phi_{\alpha})$
  of
  Lemma \ref{lemma:sheet_weight} is {\bf along $\boldsymbol{\alpha}$}.
\end{definition}

\begin{lemma}\label{lemma:special_weight}
  Let $\comp$ be a compact Hamiltonian $T$-space. Let $\mathcal{F}$ be
  the facet of $\Phi(M)$ supported on $ \{w \in \g^* \mid \langle w, \nu
  \rangle = c\} $. If $p \in M^T$ satisfies $\langle \Phi(p), \nu \rangle
  > c$, then there exists an isotropy weight $\alpha$ of $p$
  with $\langle \alpha, \nu \rangle < 0$.
\end{lemma}

\begin{proof}
  Let  $\alpha_1,\ldots,\alpha_n$ be the isotropy weights of $p$ and
  suppose that $\langle \alpha_j, \nu \rangle \geq 0$ for all $j=1,\ldots, n$. By part
  \ref{item:1} of Corollary \ref{Cor:ConeFace}, an open neighborhood
  of $\Phi(p)$ in $\Phi(M)$ equals an open neighborhood of $\Phi(p)$
  in $\Phi(p) + \R_{\geq 0} \langle
  \alpha_1,\ldots,\alpha_n\rangle$. Hence, since $\langle \alpha_j,\nu
  \rangle \geq 0$ for all $j$, an open neighborhood $V$
  of $\Phi(p)$ in $\Phi(M)$ is contained in the half-space $\{w \in
  \mathfrak{t}^* \mid \langle w, \nu \rangle \geq \langle \Phi(p), \nu \rangle
  \}$. Since $\langle \Phi(p), \nu \rangle
  > c$ and since $\Phi(M)$ has a facet supported on $ \{w \in \g^* \mid \langle w, \nu
  \rangle = c\} $, this is a contradiction.
\end{proof}

Motivated by Lemma \ref{lemma:special_weight}, we introduce the following
terminology.

\begin{definition}\label{definition:downward_pointing}
  Let $\comp$ be a compact Hamiltonian $T$-space and let $\mathcal{F}$ be
  the facet of $\Phi(M)$ supported on $ \{x \in \g^* \mid \langle x, \nu
  \rangle = c\}$. For any $p \in M^T$ with $\langle \Phi(p), \nu \rangle
  > c$, we say that the isotropy weight $\alpha$ of $p$ of
  Lemma \ref{lemma:special_weight} is {\bf
    ($\boldsymbol{\mathcal{F}}$-)downward pointing}.
\end{definition}

Combining Lemmas \ref{lemma:sheet_weight} and
\ref{lemma:special_weight}, we obtain the following result.

\begin{corollary}\label{cor:downward_pointing_sheet}
  Let $\comp$ be a compact Hamiltonian $T$-space and let $\mathcal{F}$ be
  the facet of $\Phi(M)$ supported on $ \{x \in \g^* \mid \langle x, \nu
  \rangle = c\}$. Let $p \in M^T$ be isolated with $\langle \Phi(p), \nu \rangle
  > c$ and let $\alpha$ be an isotropy weight of $p$ that
  is $\mathcal{F}$-downward pointing. Let
  $(N_{\alpha},\omega_{\alpha},\Phi_{\alpha})$ be the sheet along
  $\alpha$. There exists $q_{\alpha} \in M^T \cap N_{\alpha}$
  such that
  \begin{itemize}[leftmargin=*]
  \item $\Phi(q) = \Phi_{\alpha}(q_{\alpha})$ is a global extremum of
    $\Phi_{\alpha}$,
  \item $-\alpha$ is an isotropy weight of $q_{\alpha}$, and
  \item $\langle \Phi(q_{\alpha}), \nu\rangle <
  \langle \Phi(p),\nu \rangle$.
  \end{itemize}
\end{corollary}

\begin{proof}
  Taking $q_\alpha$ as in Lemma \ref{lemma:sheet_weight}, we need to
  prove only the last property. This follows immediately by observing that $\Phi(q) \in \Phi(p) + \R_{>0}
  \langle \alpha \rangle$ and that $\alpha$ is $\mathcal{F}$-downward pointing.
\end{proof}

To conclude this section, we look at the preimage of faces of the
moment polytope. Given a face $\mathcal{F}$ of $\Phi(M)$, we set
\begin{equation}
  \label{eq:23}
\mathfrak{h}_{\mathcal{F}}:= \left\lbrace \xi \in \g \mid
   \left\langle x-y, \xi\right\rangle=0 \text{ for all }  
 x,y\in \mathcal{F} \right\rbrace.
 \end{equation}
By part \ref{item:1} of Corollary \ref{Cor:ConeFace},
$\mathfrak{h}_{\mathcal{F}}$ is a rational subspace of $\g$ of
dimension equal to the codimension of $\mathcal{F}$ in $\Phi(M)$. Hence 
\(\exp(\mathfrak{h}_{\mathcal{F}})\) is subtorus of \(T\). The subset
$M_{\mathcal{F}}:=\Phi^{-1}(\mathcal{F}) \subset M$ is $T$-invariant; we set
\begin{equation}
  \label{eq:15}
  H_{\mathcal{F}}:= \{ t \in T \mid t \cdot p = p
  \text{ for all } p
  \in M_{\mathcal{F}} \}.
\end{equation}

\begin{lemma}[Theorem 3.5.12 in \cite{Nicolaescu}]\label{Lem:FacesSubmanifolds}
  Let $\comp$ be a compact Hamiltonian $T$-space and let \(\mathcal{F}\) be a face of $\Phi(M)$. Then 
  \(M_{\mathcal{F}}=\Phi^{-1}(\mathcal{F})\) is a connected component of
  \(M^{H_{\mathcal{F}}}\) and the Lie algebra of $
  H_{\mathcal{F}}$ equals $\mathfrak{h}_{\mathcal{F}}$.
\end{lemma}

\begin{remark}\label{rmk:preimage_connected}
  Connectedness of $M_{\mathcal{F}}$ can also be proved using the fact
  that
  the Convexity Package implies that the preimage of any convex set is
  connected (see \cite[Section 2]{kt3}).
\end{remark}

\begin{corollary}\label{cor:preimage_face}
  Let $\comp$ be a compact Hamiltonian $T$-space and let
  \(\mathcal{F}\) be a face of $\Phi(M)$. There exists
  a connected, open and dense subset of $M_{\mathcal{F}}$ whose points
  have stabilizer equal to $H_{\mathcal{F}}$.
\end{corollary}

\begin{proof}
  By the principal orbit theorem (see \cite[Theorem
  2.8.5]{dk}) and since $T$ is abelian, there exists a subgroup $H$
  of $T$ and a connected, open and dense subset of $M_{\mathcal{F}}$
  such that $H$ is the stabilizer of any point in this subset. Hence,
  $H_{\mathcal{F}} \subseteq H$. However, since $H$ is the stabilizer
  of orbits of principal type and since $T$ is abelian, $H$ is
  contained in the stabilizer of any point in
  $M_{\mathcal{F}}$. Therefore, $H \subseteq H_{\mathcal{F}}$. 
\end{proof}

By Corollary \ref{cor:preimage_face}, the preimage of
a face $\mathcal{F}$ of $\Phi(M)$ gives rise to a sheet in the sense
of Definition \ref{defn:sheet} that is stabilized by
$H_{\mathcal{F}}$. We denote it by 
\begin{equation}
  \label{eq:11}
  (M_{\mathcal{F}},\omega_{\mathcal{F}},\Phi_{\mathcal{F}}).
\end{equation}
\noindent
In particular, since the complexity of $
(M_{\mathcal{F}},\omega_{\mathcal{F}},\Phi_{\mathcal{F}})$ is at most that of $\comp$ by Proposition \ref{prop:regular_exceptional}, if
the codimension of $\mathcal{F}$ is $r$, then $\dim M_{\mathcal{F}} \leq 2n - 2r$.

\subsubsection{The Duistermaat-Heckman measure and its density
  function}\label{sec:duist-heckm-meas}
In this section we take a close look at an invariant of compact
Hamiltonian $T$-spaces that is central to this paper. We start by
recalling the following notion.

\begin{definition}\label{defn:DH_measure}
  Let $\comp$ be a compact Hamiltonian $T$-space of dimension $2n$. The {\bf Duistermaat-Heckman measure of
    $\boldsymbol{\comp}$} is the pushforward of the (normalized)
  Liouville measure, i.e., for any Borel set $U\subset\mathfrak{t}^*$,
  \begin{equation}
    \label{eq:13}
    m_{DH}(U)\colon=\frac{1}{(2\pi)^n}\int_{\Phi^{-1}(U)}
    \frac{\omega^n}{n!}.
  \end{equation}
\end{definition}

The Duistermaat-Heckman measure is absolutely continuous with respect to the Lebesgue 
measure on $\mathfrak{t}^*$ (see \cite[equation (3.4)]{dh}). Therefore its Radon-Nikodym derivative with respect 
to the Lebesgue measure is a Lebesgue integrable function $f_\text{DH}\colon \mathfrak{t}^* \rightarrow 
\R$ that is uniquely defined up to a set of measure zero. Without loss
of generality, we henceforth assume that $f_{\text{DH}}$ vanishes identically
on $\g^* \smallsetminus \Phi(M)$.

In \cite{dh}, Duistermaat and Heckman give an explicit representative of the restriction of
$f_{\text{DH}}$ to the intersection of the moment polytope with the
set of regular values of $\Phi$. In order to
state this result, we denote the set of regular
values of $\Phi$ contained in $\Phi(M)$ by
$\Phi(M)_{\mathrm{reg}}$. Moreover, we recall that for any $x\in
\Phi(M)_{\mathrm{reg}}$, the reduced space $M_x$ is an orbifold that
inherits a symplectic form that we denote by $\omega_x$ (see Section
\ref{sec:orbital-moment-map}). 

\begin{thm}[Proposition 3.2 in \cite{dh}]\label{THM:DH}
  Let $\comp$ be a compact complexity $k$ $T$-space. The restriction of the
  Radon-Nikodym derivative of the Duistermaat-Heckman measure of
  $\comp$ to $\Phi(M)_{\mathrm{reg}}$ can be chosen to be
  equal to the function $\Phi(M)_{\mathrm{reg}} \to \R$ 
  \begin{equation}
    \label{eq:1}
      x \mapsto \frac{1}{(2\pi)^k}\int_{M_x} \frac{\omega_x^k}{k!} =:
      \frac{1}{(2\pi)^k} \mathrm{Vol}(M_x), \quad x \in \Phi(M)_{\mathrm{reg}}.
  \end{equation}
\end{thm}

\begin{remark}\label{rmk:polynomial}
  By \cite[Corollary 3.3]{dh}, the restriction of the
  function \eqref{eq:1} to each connected component of $\Phi(M)_{\mathrm{reg}}$
  is a polynomial of degree at most $k$. Moreover, if this
  polynomial has positive degree, the coefficients
  of the monomials of top degree are integral. This is because the cohomology classes
  $[\omega_x]$ vary linearly with $x$ on such a connected component
  and the variation is controlled by a cohomology class
  with integral coefficients (see \cite[Theorem 1.1]{dh}).
\end{remark}

\begin{rmk}\label{rmk:thm_enough}
  Theorem \ref{THM:DH} is sufficient to calculate the
  Duistermaat-Heckman measure, as the set of singular values has
  measure zero by Sard's theorem. 
\end{rmk}

By Remark \ref{rmk:polynomial}, the function of
\eqref{eq:1} is continuous. The aim of this section is to prove Theorem \ref{thm:DH_function_cts}
below, which is probably well-known to experts (cf. \cite[Theorem
3.2.11]{GLS} for linear symplectic actions on vector spaces and
\cite[Theorem A and Remark 1.2.10]{paradan}). However,
since we use it extensively throughout the paper, we include a proof for completeness.

\begin{theorem}\label{thm:DH_function_cts}
  Given a compact Hamiltonian $T$-space $\comp$, there exists a unique
  continuous function $DH \comp: \Phi(M) \to \R$ that extends the
  function of \eqref{eq:1}.  
\end{theorem}

\begin{remark}\label{rmk:DH_no_internal_walls}
  By Remark \ref{rmk:polynomial}, if the interior of the moment polytope consists solely of regular
  values, then Theorem \ref{thm:DH_function_cts}
  is trivial. In this case, $DH\comp$ is the restriction
  of a polynomial of degree at most the complexity of
  $\comp$. For instance, the desired function in the
  case of compact symplectic toric
  manifolds is the indicator function of the moment polytope, since
  reduced spaces are connected by Theorem \ref{thm:con_pac}, and since $k = \dim M_x
  = 0$ for all $x \in \Phi(M)_{\mathrm{reg}}$. 
\end{remark}

Theorem \ref{thm:DH_function_cts} allows us to introduce the following notion.

\begin{definition}\label{Def: DHfucntion}
Let $\comp$ be a compact Hamiltonian $T$-space. We call the continuous
map $DH \comp : \Phi (M) \to \R$ given by Theorem
\ref{thm:DH_function_cts} {\bf the Duistermaat-Heckman
  function} of $\comp$.
\end{definition}

\begin{example}\label{exm:induced_DH}
  Let $\comp$ be a compact Hamiltonian $T$-space and let $H \subset T$
  be a subtorus. Choose a complementary subtorus $K$ of $T$ so that
  $T = H \times K$. This induces an identification $\g^* \simeq \mathfrak{h}^* \oplus
  \mathfrak{k}^*$. We write the Lebesgue measure on $\g^*$ as $dxdy$,
  where $dx$ (respectively $dy$) is the Lebesgue measure on
  $\mathfrak{h}^*$ (respectively $\mathfrak{k}^*$). Let $\pi : \g^*
  \to \mathfrak{h}^*$ be the projection induced by the inclusion $H
  \subset T$. The $H$-action is Hamiltonian with moment map $\Phi':= \pi \circ \Phi$. Since $DH\comp $ is continuous, by Fubini's theorem, the Duistermaat-Heckman function of
  $(M,\omega, \Phi')$ is given by
  \begin{equation}
    \label{eq:19}
      DH (M,\omega, \Phi') (x) = \int\limits_{\Delta_x}DH\comp(x,y) dy
      \quad \text{for all } x \in \Phi'(M),
  \end{equation}
  \noindent
  where $\Delta_x = \pi^{-1}(x) \cap \Phi(M)$.

  Suppose further that $\comp$ is a compact symplectic toric manifold and that
  $H$ has codimension 1. For any
  $x \in \Phi'(M)$, $\Delta_x$ is an interval in $\{x\} \times
  \mathfrak{k}^* \simeq \{x\} \times \R$ that we write as $\{(x,y) \in \mathfrak{h}^*\oplus \R \mid y \in
  [p_{min}(x), p_{max}(x)] \}$. By Remark \ref{rmk:DH_no_internal_walls}, 
  \begin{equation}
    \label{eq:20}
    DH (M,\omega, \Phi') (x) = p_{\max}(x) - p_{\min}(x) \quad
    \text{for all } x \in \Phi'(M). 
  \end{equation}
  We observe that, since $\Phi(M)$ is a convex polytope, the
  difference $p_{max} - p_{min}$ is concave (cf. Proposition
  \ref{prop:concave} below).

\end{example}

Our proof of Theorem \ref{thm:DH_function_cts} uses extensively ideas from \cite[Section
3.5]{GLS}. Before proceeding with the proof, we need to recall some
facts about singular values of the moment map.

\subsubsection*{Intermezzo 1: chambers of the moment map}
We begin by relating singular points of the moment map of a
Hamiltonian $T$-space (that is not necessarily compact), to sheets
arising from one dimensional stabilizers (see Definition
\ref{defn:sheet}). If $K \leq T$ is a closed one-dimensional
subgroup that occurs as the stabilizer of some point in $M$, then
every $p \in M^K$ is a singular point of $\Phi$. In fact, the converse
also holds.

\begin{lemma}\label{lemma:walls}
  Let $\comp$ be a Hamiltonian $T$-space. If $K \leq T$ is a
  one-dimensional closed
  subgroup that occurs as the stabilizer of some point
  in $M$ then $\Phi(M^K)$ is contained in the set of singular values
  of $\Phi$. Conversely, if $x \in \g^*$ is a
  singular value of $\Phi$, then there exists a one-dimensional closed
  subgroup $K \leq T$ that occurs as the stabilizer of some point
  in $M$ such that $x \in \Phi(M^K)$.
\end{lemma}

\begin{proof}
  Since the action is Hamiltonian, the first statement
  holds. Conversely, let $p \in M$ be a singular point with $\Phi(p)
  =x$. Let $H$ be the stabilizer of $p$, let $h = \dim H$ and let $k$
  be the complexity of $\comp$. Since the $T$-action is
  Hamiltonian and since $p$ is a singular point of $\Phi$, $\dim H \geq 1$. If $\dim H = 1$, there is nothing to prove. Hence, suppose that
  $\dim H \geq 2$. By Theorem \ref{local normal form}, it suffices to prove the result for the Hamiltonian $T$-action
  on the local model of $p$. Hence, we may assume that $\comp$ is the
  local model $(Y,\omega_Y,\Phi_Y)$ at $p$ and that $p =
  [1,0,0]$. Moreover, by Corollary \ref{cor:effective}, the symplectic
  slice representation $\rho :  H \to
  (S^1)^{h+k}$ of $p$ is injective. The
  stabilizer of any point in $Y = T
  \times_H \mathrm{Ann}(\mathfrak{h}) \times \C^{h+k}$ is a subgroup of
  $H$. In fact, if $\tilde{H} \leq H$ is the stabilizer of some
  point in $Y$, we have that
  $$ Y^{\tilde{H}} = T \times_H \mathrm{Ann}(\mathfrak{h}) \times
  (\C^{h+k})^{\tilde{H}}, $$
  \noindent
  where $H$ acts on $\C^{h+k}$ via
  the symplectic slice representation of $p$. Since the $H$-action on $\C^{h+k}$ is linear,
  $(\C^r)^{\tilde{H}} $ is a subspace for any subgroup $\tilde{H}
  \leq H$. Hence, to prove the result, it suffices to show that there exists a
  one-dimensional closed subgroup $K \leq H$ that occurs as a
  stabilizer of a point in $\C^{h+k}$. 

  Let $\eta_1,\ldots, \eta_{h+k} \in \mathfrak{h}^*$ be the isotropy
  weights of $p$. Since the symplectic slice representation of $p$ is
  injective, the $H$-action on $\C^{h+k}$ is effective. This implies that $\eta_1,\ldots, \eta_{h+k}$
  span $\mathfrak{h}^*$. Hence we may assume that there exists $s \geq
  1$ such that the span of $\eta_{s+1},\ldots, \eta_{h+k}$ has codimension
  1. Since the $H$-action on $\C^{h+k}$ is Hamiltonian with moment map $\Phi_H(z) = \frac{1}{2}\sum\limits_{j=1}^{h+k}
  |z_j|^2 \eta_j$, it can be checked directly that all
  points in
  \begin{equation*}
     \{z = (z_1,\ldots,z_r) \in \C^{h+k} \mid z_j = 0 \text{ if }
     j=1,\ldots,s, \, z_j \neq 0 \text{ if } j \geq s+1 \}, 
  \end{equation*}
 \noindent
 have stabilizers of dimension one, as desired.
\end{proof}

In other words, the set of singular values of the moment map equals
the union of the moment map image of all sheets that are stabilized by
some one dimensional closed subgroup
of $T$ (see Definition \ref{defn:sheet}). Each such image is contained
in some affine hyperplane (see Remark
\ref{rmk:convention_moment_map_sheet} and the discussion preceding
it). If, in addition, $M$ is compact there are two important
consequences. First, there are only finitely many subgroups of $T$ that occur
as the stabilizer of some point in $M$ (see \cite[Theorem
10.9.1]{gs-supersymmetry}). Second, since the action is
Hamiltonian, $\Phi$ has some singular
point; hence, by Lemma \ref{lemma:walls}, there exists a
one dimensional closed subgroup of $T$ that that occurs
as the stabilizer of some point in $M$. Let $K_1,\ldots, K_r \leq T$ be the
collection of such one dimensional closed subgroups. Since $M$ is
compact, by Lemma \ref{Lem:fixedSubmainfolds}, $M^{K_i}$ is a compact submanifold of $M$
for each $i=1,\ldots,
r$ and, therefore, has finitely
many connected components $N_{i1},\ldots, N_{is_i}$. For each $i,j$,
we denote the corresponding sheet by $(N_{ij},\omega_{ij},\Phi_{ij})$
and we set $\Delta_{ij}:= \Phi_{ij}(N_{ij})$. 
We observe that the union of the $\Delta_{ij}$'s includes the union of all
facets. More precisely, the following holds.

\begin{lemma}\label{lemma:facets_are_included}
  Let $\comp$ be a compact Hamiltonian $T$-space. For any facet
  $\mathcal{F}$ of $\Phi(M)$, there exist indices $i,j$ as above so
  that $\mathcal{F} = \Delta_{ij}$. 
\end{lemma}

\begin{proof}
  Let $(M_{\mathcal{F}},\omega_{\mathcal{F}},\Phi_{\mathcal{F}})$ be the
  sheet corresponding to $\mathcal{F}$ (see Corollary \ref{cor:preimage_face} and
  \eqref{eq:11}), and let $H_{\mathcal{F}}$ its stabilizer. By Lemma
  \ref{Lem:FacesSubmanifolds}, since the codimension of $\mathcal{F}$
  in $\g^*$ is one, the dimension of $H_{\mathcal{F}}$ is also one. By
  Corollary \ref{cor:preimage_face}, it follows that
  $(M_{\mathcal{F}},\omega_{\mathcal{F}},\Phi_{\mathcal{F}})$ is one
  of the sheets constructed above.
\end{proof}

The complement of the union of the $\Delta_{ij}$'s in the moment polytope is precisely $\Phi(M)_{\mathrm{reg}}$. We call the
closure of a connected component of this complement a {\bf chamber of
  $\boldsymbol{\Phi(M)}$}. These chambers partition the moment polytope into
subpolytopes, i.e., the following properties hold (see Figure \ref{Figure ChambersFlagVariety}):

\begin{enumerate}[label=(\arabic*),ref=(\arabic*),leftmargin=*]
\item \label{item:16} Each chamber is a polytope in $\g^*$ of full dimension and any two chambers
intersect in a common face. 
\item \label{item:11} Let $F$ be a facet of a chamber. The set of
  points $x\in F$ that are regular values for all the moment maps
  $\Phi_{ij}$ of the sheets
  $(N_{ij},\omega_{ij},\Phi_{ij})$ constructed above is dense in
  $F$. Moreover this set is contained in the interior of
  $F$. Conversely, if $x$ is a singular value that is a regular value
  for all the moment maps
  $\Phi_{ij}$ of the sheets
  $(N_{ij},\omega_{ij},\Phi_{ij})$ constructed above, then $x$ lies in
  the interior of a facet of a chamber.
\item \label{item:15} If two chambers $\mathfrak{c}$ and
  $\mathfrak{c}'$ intersect in a face $F$, then there exists a
  sequence of chambers $\mathfrak{c}_0 = \mathfrak{c},
  \mathfrak{c}_1,\ldots, \mathfrak{c}_s = \mathfrak{c}'$ with the
  property that the intersection of $\mathfrak{c}_l$ and
  $\mathfrak{c}_{l+1}$ is a facet that contains $F$ for all
  $l=0,\ldots, s-1$.
\end{enumerate}

\begin{figure}[htbp]
	\begin{center}
		\includegraphics[width=4.5cm]{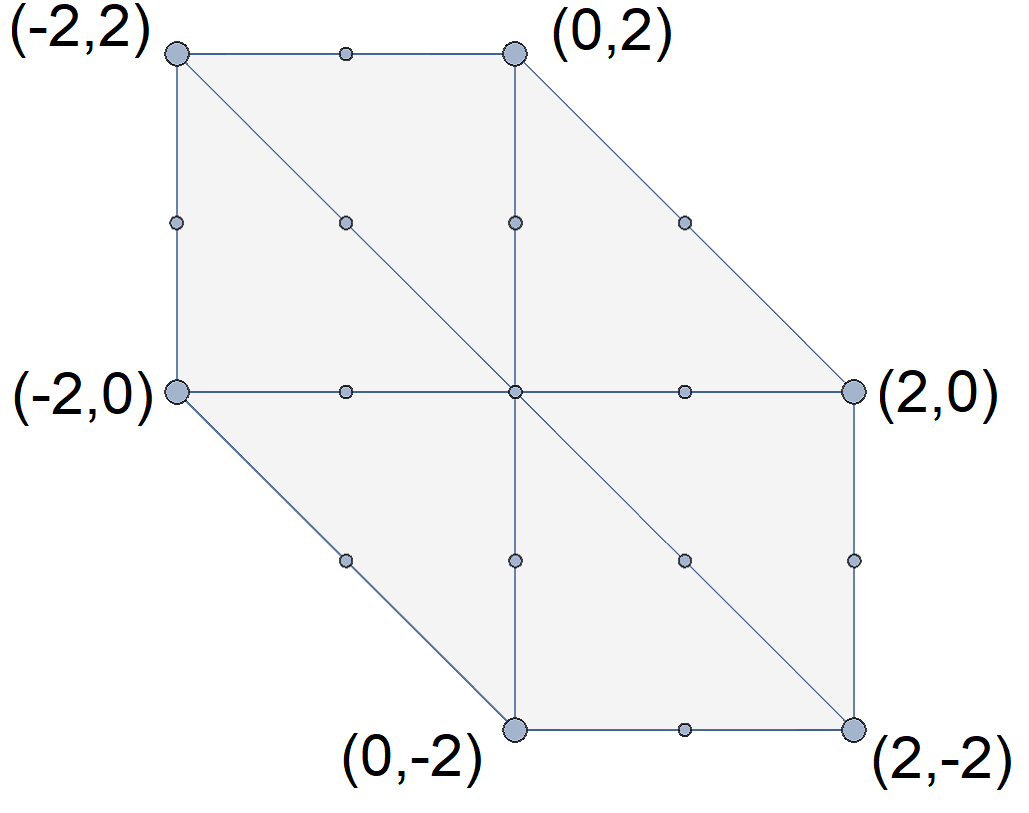}
		\caption{The six triangles above are the chambers for
                  the moment polytope for the standard
                  $(S^1)^2$-action on the complete flag variety in
                  $\C^3$ with first Chern
                  class equal to the class of the symplectic form.}
		\label{Figure ChambersFlagVariety}
	\end{center}
\end{figure}

Properties \ref{item:16} -- \ref{item:15} follow from the above
definition of the $N_{ij}$'s and from Theorems \ref{local normal form}
and \ref{thm:con_pac} (see also \cite[Section 14]{gs-inve}). \\

By Remark \ref{rmk:DH_no_internal_walls}, Theorem
\ref{thm:DH_function_cts} follows at once if the interior of $\Phi(M)$
consists entirely of regular values. The following result describes
precisely when this happens in terms of exceptional sheets.

\begin{lemma}\label{lemma:one_chamber}
  Let $\comp$ be a compact Hamiltonian $T$-space. There is
  precisely one chamber of $\Phi(M)$ if and only if there are no
  exceptional sheets.
\end{lemma}
\begin{proof}
  If $\Phi(M)$ has no exceptional sheets then all sheets are contained
  in the boundary of $\Phi(M)$ by Lemma
  \ref{lemma:exceptional_not_boundary}. This implies that the union of
  the $\Delta_{ij}$'s constructed above is contained in the boundary of
  $\Phi(M)$, which equals the union of the facets of $\Phi(M)$ since
  $\Phi(M)$ is a convex polytope. By Lemma
  \ref{lemma:facets_are_included}, the union of all the facets of
  $\Phi(M)$ is contained in the union of the $\Delta_{ij}$'s. Hence,
  the union of the $\Delta_{ij}$'s equals the boundary of
  $\Phi(M)$. Since $\Phi(M)$ is a convex polytope of full dimension, the complement of the boundary of
  $\Phi(M)$ in $\Phi(M)$ equals the interior of $\Phi(M)$, which is
  connected. Hence, there is precisely one chamber of $\Phi(M)$.

  Conversely, if $\Phi(M)$ has at least two chambers, then
  at least one of the $\Delta_{ij}$'s constructed above cannot be
  contained in the boundary of $\Phi(M)$. By Lemma
  \ref{lemma:exceptional_not_boundary}, the corresponding sheet is exceptional.
\end{proof}

Seeing as the case of only one chamber has already been proved, in
what follows (namely, in Intermezzo 2 and in the proof of Theorem
\ref{thm:DH_function_cts} below), we assume that there are at least
two chambers.

\subsubsection*{Intermezzo 2: the wall-crossing formula} The main tool that we use
in the proof of Theorem \ref{thm:DH_function_cts} is the so-called \textbf{wall-crossing
  formula} for compact Hamiltonian $T$-spaces (see \cite{bp}, \cite[Section
3.5]{GLS} and \cite{paradan}). We recall it here for completeness and
we draw on the above Intermezzo for notation. Moreover, in this
subsection we consider the closure of the complement of the moment map
as a chamber.

Let
$\mathfrak{c}_{\pm}$ be two chambers in $\Phi(M)$ that intersect in a
facet $F$. Let $\xi \in \ell$ be the primitive
element that is normal to the hyperplane supporting $F$ and that points out
of $\mathfrak{c}_-$ (see Figure \ref{Wall}). We fix a point $x \in F$ that
has the property that is a regular value for all for all the moment maps
  $\Phi_{ij}$ of the sheets
  $(N_{ij},\omega_{ij},\Phi_{ij})$ constructed in the above Intermezzo;
such a point exists by property \ref{item:11}.

\begin{figure}[htbp]
\begin{center}
\includegraphics[width=6cm]{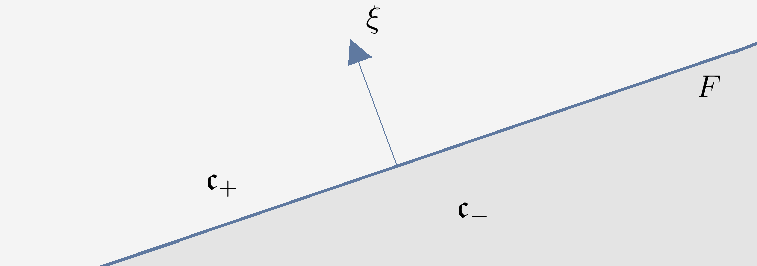}
\caption{Two chambers, $\mathfrak{c}_-$ and $\mathfrak{c}_+$, separated
by a hyperplane with normal $\xi$.}
\label{Wall}
\end{center}
\end{figure}

We use the fixed inner
product on $\g$ to choose a complementary subspace $\mathfrak{k}$ to
the span of $\xi$, i.e., $\g = \langle \xi \rangle \oplus
\mathfrak{k}$. Hence, viewing $\g$ as the space of homogeneous polynomials of
degree one on $\g^*$, we can view polynomials on $\g^*$ as being
generated by $\xi$ and polynomials on $\mathfrak{k}^*$. Moreover,
since $\xi \in \ell$, we have that $\exp (\langle \xi \rangle)$ is a
circle in $T$ that we denote by $S^1$. The subspace $\mathfrak{k}$ is
isomorphic to the Lie algebra of the quotient $T/S^1$. In what follows, we use this identification tacitly since
it is compatible with the identification of Remark \ref{rmk:convention_moment_map_sheet}.

Let $f_{\pm}: \g^* \to \R$ be the polynomial that,
when restricted to the interior of $\mathfrak{c}_{\pm}$, equals
\eqref{eq:1}. Since $x$ is a regular value of
$\Phi_{ij}$, it lies in the interior of a
chamber $\mathfrak{c}_{ij}$ of $\Phi_{ij}(N_{ij})$ for all $i,j$. Hence, there is a corresponding polynomial
$f_{ij}: \mathfrak{k}^* \to \R$ that, when restricted to the interior of the
$\mathfrak{c}_{ij}$, equals \eqref{eq:1}. (If $x \notin
\Phi_{ij}(N_{ij})$, this polynomial is identically zero.) 

Finally, for each $i,j$, we let $\kappa_{ij}$ be half the codimension
of $N_{ij}$ in $M$ and, if $i,j$ are such that $x \in \Phi_{ij}(N_{ij})$, we let $\alpha_{ij1},\ldots,
\alpha_{ij\kappa_{ij}} \in \Z$ be the isotropy weights for the $S^1$-action on the normal bundle to $N_{ij}$.

\begin{theorem}[Wall-crossing formula]\label{Thm: GLS-wall-crossing-formula}
  Set the notation as above. For all $y \in \g^*$ we have that
  \begin{equation}
    \label{eq:2}
    f_+(y)- f_-(y) = \sum\limits_{\{i,j \mid x \in \Phi_{ij}(N_{ij})\}}\xi^{\kappa_{ij}-1}(y-x)\left(\prod\limits_{s=1}^{\kappa_{ij}}
      \alpha_{ijs}\right)^{-1}\left[\frac{f_{ij}(y-x)}{(\kappa_{ij}-1)!}
      + P_{ij}(y-x)\right],
  \end{equation}
  \noindent
  where $P_{ij}$ is a polynomial depending on $i,j$ that is divisible
  by $\xi$. 
  
\end{theorem}

\begin{remark}\label{rmk:gls_formula}
  \mbox{}
  \begin{itemize}[leftmargin=*]
  \item We stress that Theorem \ref{Thm: GLS-wall-crossing-formula}
    also holds if, say, $\mathfrak{c}_-$ is the closure of the complement of
    the moment map image. 
  \item The polynomials $P_{ij}$ have been computed explicitly (see
    \cite[Theorem 1.2.9]{paradan}). They depend on the symplectic
    reduction of $N_{ij}$ at $x$ by the $T/S^1$-action.
  \end{itemize}
\end{remark}

\subsubsection*{Back to the Duistermaat-Heckman function of $\comp$}
With the above Intermezzos, we have all the ingredients to proceed with the proof of Theorem \ref{thm:DH_function_cts}. 

\begin{proof}[Proof of Theorem \ref{thm:DH_function_cts}]
  First, we prove existence of a continuous extension. Let $\mathfrak{c}_1,\ldots, \mathfrak{c}_l$ be the
  chambers of $\Phi(M)$ and, for each $i = 1,\ldots, l$, let $f_i:\g^*
  \to \R$ be the polynomial that equals
  \eqref{eq:1} when restricted to the interior of
  $\mathfrak{c}_i$. The result is proved if
  we show that the map given by
  \begin{equation}
    \label{eq:14}
    x \mapsto f_i(x) \qquad \text{if } x \in \mathfrak{c}_i
  \end{equation}
  \noindent
  is well-defined, i.e., given two chambers $\mathfrak{c}_i$ and
  $\mathfrak{c}_j$ such that $\mathfrak{c}_i\cap \mathfrak{c}_j \neq
  \emptyset$, the following holds:
   \begin{equation}
    \label{eq:7}
    f_i(x) = f_j(x) \text{ for all } x \in \mathfrak{c}_i \cap \mathfrak{c}_j.
  \end{equation}
  \noindent
  Clearly equation \eqref{eq:7} holds if $i=j$, so suppose that $i
  \neq j$. Set $F := \mathfrak{c}_i \cap \mathfrak{c}_j$; this is
  a face of both $\mathfrak{c}_i$ and $\mathfrak{c}_j$ (see property \ref{item:16}).

  We consider
  first the special case that $F$ is a facet. Since the set of points in $F$ that
  are regular values for all the sheets
  $(N,\omega_N,\Phi_N)$ constructed in Intermezzo 1 is
  dense in $F$ (see property \ref{item:11}), and since both $f_i$ and
  $f_j$ are continuous, it suffices to check that equation \eqref{eq:7} holds
  for those points. Let $x \in F$ be such a point. If we show
  that the codimension of each $N$ such
  that $x \in \Phi(N)$ is at least four, then we are done by
  Theorem \ref{Thm: GLS-wall-crossing-formula} (cf. \cite[Remark 1.2.10]{paradan}). Since $F$ is a
  facet of two distinct chambers, it is not contained in the boundary
  of $\Phi(M)$. In particular, since $x$ lies in the interior of $F$,
  it does not lie in the boundary of $\Phi(M)$. Therefore, if $N$
  is such that $x \in \Phi(N)$, then $\Phi(N)$ is
  not contained in the boundary of $\Phi(M)$. By Lemma
  \ref{lemma:exceptional_not_boundary}, $(N,\omega,\Phi)$
  is an exceptional sheet. Thus the complexity of
  $(N,\omega,\Phi)$ is strictly less than that of
  $\comp$ by Proposition \ref{prop:regular_exceptional}. Since the dimension of the
  torus that acts effectively on $N$ is one less than that of
  $T$, it follows that the codimension of $N$ in $M$ is at least
  four, as desired.

  If $F$ is not a facet, then by property \ref{item:15} above we can find chambers $\mathfrak{c}_0 =
  \mathfrak{c}_i, \mathfrak{c}_1,\ldots, \mathfrak{c}_s =
  \mathfrak{c}_j$ such that $\mathfrak{c}_l$ and $\mathfrak{c}_{l+1}$
  intersect in a facet that contains $F$ for all $l=0,\ldots, s-1$ (see property \ref{item:15}). The result then follows
  immediately by the above special case.

  Uniqueness of the continuous extension follows immediately by
  observing that $\Phi(M)_{\mathrm{reg}}$ is a dense subset
  of $\Phi(M)$ and that the function of \eqref{eq:1} is continuous. 
\end{proof}

The Duistermaat-Heckman function is an invariant of the isomorphism
class of a compact Hamiltonian $T$-space that plays an important role in this
paper. As an illustration, the following result describes the restriction of the
Duistermaat-Heckman function of $\comp$ to any facet of the moment
polytope. 

\begin{proposition}\label{prop:DH_boundary}
  Let $\comp$ be a compact Hamiltonian $T$-space. If
  $\mathcal{F}$ is a facet of $\Phi(M)$, then
  \begin{equation}
    \label{eq:12}
    DH \comp|_{\mathcal{F}} =
    \begin{cases}
      DH(M_{\mathcal{F}},\omega_{\mathcal{F}},\Phi_{\mathcal{F}}) &
      \text{if } \dim M_{\mathcal{F}} = \dim M - 2, \\
      0 & \text{otherwise,}
    \end{cases}
  \end{equation}
  \noindent
  where $(M_{\mathcal{F}},\omega_{\mathcal{F}},\Phi_{\mathcal{F}})$
  is as in \eqref{eq:11}.
\end{proposition}

\begin{proof}
  In this proof we denote the closure
  of the complement of the moment map image by
  $\mathfrak{c}_-$ (cf. the paragraph
  preceding Theorem \ref{Thm: GLS-wall-crossing-formula} and Remark
  \ref{rmk:gls_formula}). By Lemma \ref{lemma:facets_are_included}, 
  $(M_{\mathcal{F}},\omega_{\mathcal{F}},\Phi_{\mathcal{F}})$ is one
  of the sheets constructed in Intermezzo 1; in particular, the
  quotient $T/H_{\mathcal{F}}$ is isomorphic to $S^1$, where
  $H_{\mathcal{F}}$ is the stabilizer of $(M_{\mathcal{F}},\omega_{\mathcal{F}},\Phi_{\mathcal{F}})$. Let $x \in \mathcal{F}$
  be a regular value of $\Phi_{\mathcal{F}}$. Since $M_{\mathcal{F}} =
  \Phi^{-1}(\mathcal{F})$ (see Lemma \ref{Lem:FacesSubmanifolds}), it
  follows that
  $(M_{\mathcal{F}},\omega_{\mathcal{F}},\Phi_{\mathcal{F}})$ is the
  only sheet constructed in Intermezzo 1 that contains $x$ in its
  moment map image. In particular, $x$ is a regular value for all the
  moment maps of all the sheets constructed in Intermezzo 1. Hence,
  there exists a chamber $\mathfrak{c}_+$ and a facet $F$ of
  $\mathfrak{c}_+$ such that $x$ lies in the interior of $F$.

  Let $\xi \in \ell$ be the primitive normal to the hyperplane
  supporting $F$ that points out of $\mathfrak{c}_-$. Let $f_+ : \g^* \to \R$ denotes the polynomial that, when restricted
  to $\mathfrak{c}_+$, equals \eqref{eq:1}. Since $f_- \equiv 0$, by Theorem \ref{Thm:
    GLS-wall-crossing-formula},
  \begin{equation}
    \label{eq:8}
     f_+(y)= \xi^{\kappa_{\mathcal{F}}-1}(y-x) \left(\prod\limits_{s=1}^{\kappa_{\mathcal{F}}}
      \alpha_{s}\right)^{-1}\left[\frac{f_{\mathcal{F}}(y-x)}{(\kappa_{\mathcal{F}}-1)!}
      + P_{\mathcal{F}}(y-x)\right] \quad \text{for all } y \in \g^*,
  \end{equation}
  \noindent
  where $\kappa_{\mathcal{F}}$ is half of the codimension of
  $M_{\mathcal{F}}$ in $M$, $\alpha_1,\ldots,
  \alpha_{\kappa_{\mathcal{F}}} \in \Z$ are the isotropy weights of
  the $S^1$-action on the normal bundle to $M_{\mathcal{F}}$,
  $f_{\mathcal{F}}$ is the polynomial that equals \eqref{eq:1} when restricted to the
  chamber of $\Phi_{\mathcal{F}}$ containing $x$, and $P_{\mathcal{F}}$ is the polynomial
  associated to $(M_{\mathcal{F}},\omega_{\mathcal{F}},\Phi_{\mathcal{F}})$ in
  \eqref{eq:2}. By \eqref{eq:8}, if $\kappa_{\mathcal{F}} \geq 2$, then $f_+(x) = 0$. On the
  other hand, if $\kappa_{\mathcal{F}} = 1$, since $P_{\mathcal{F}}$
  is divisible by $\xi$ (see Theorem \ref{Thm:
    GLS-wall-crossing-formula}, then the right hand side
  of \eqref{eq:8} evaluated at $x$ equals $ f_{\mathcal{F}}(0)/\alpha_1$.
  Since the $S^1$-action is effective and $\xi$ is chosen to point out
  of $\mathfrak{c}_-$, it follows that $\alpha_1 = 1$ and, hence,
  $f_+(x) = f_{\mathcal{F}}(0)$. Since $f_+$ and $f_{\mathcal{F}}$ are
  restrictions of $DH\comp$ and $DH
  (M_{\mathcal{F}},\omega_{\mathcal{F}},\Phi_{\mathcal{F}}) $
  respectively, we have shown that, if $x$ is a regular value of
  $\Phi_{\mathcal{F}}$, then
  \begin{equation}
    \label{eq:54}
    DH\comp(x) =
    \begin{cases}
      DH (M_{\mathcal{F}},\omega_{\mathcal{F}},\Phi_{\mathcal{F}})(x) &
      \text{if } \dim M_{\mathcal{F}} = \dim M - 2 \\
      0 & \text{otherwise.}
    \end{cases}
  \end{equation}
  \noindent
  Since $\Phi_{\mathcal{F}}(M_{\mathcal{F}})_{\mathrm{reg}}$ is
  dense in $\Phi_{\mathcal{F}}(M_{\mathcal{F}})$, and since $DH\comp$
  and $DH (M_{\mathcal{F}},\omega_{\mathcal{F}},\Phi_{\mathcal{F}})$
  are continuous and defined on $\mathcal{F}$, equation \eqref{eq:54}
  implies the desired result.
\end{proof}

\subsubsection*{The Duistermaat-Heckman function of a complexity one $T$-space}
In this section we prove some properties of the Duistermaat Heckman
function of a compact complexity one $T$-space $\comp$. We start with the
following special property that fails to hold in higher complexity
(see \cite{kar_not_log}).

\begin{proposition}\label{prop:concave}
  The Duistermaat-Heckman function $DH\comp : \Phi(M) \to \R$ of a compact complexity one
  $T$-space $\comp$ is concave. 
\end{proposition}

\begin{proof}
  By Theorem \ref{thm:DH_function_cts}, $DH\comp$ is
  continuous. Hence, it suffices to check that the restriction of $DH
  \comp$ to the interior of $\Phi(M)$ is concave. Since the complexity
  of $\comp$ is one, by Remark \ref{rmk:polynomial}, $DH \comp$ is piecewise
  linear. Thus the restriction of $DH
  \comp$ to the interior of $\Phi(M)$ is concave if and only if it is
  log-concave. The result then follows immediately from \cite[Theorem 1.1]{ck}.
\end{proof}

As observed in \cite[Corollary 3.8]{ss}, continuity and concavity of
$DH \comp$ (see Theorem \ref{thm:DH_function_cts}
and Proposition \ref{prop:concave}), together with convexity of $\Phi(M)$ (see Theorem \ref{thm:con_pac}), immediately
imply the following result.

\begin{corollary}\label{cor:min_DH_function}
  The minimum of the Duistermaat-Heckman function of a compact
  complexity one $T$-space is attained at a vertex.
\end{corollary}

Next we prove two results relating $DH \comp$ with exceptional orbits
and singular values of $\Phi$.

\begin{lemma}\label{lemma:no_isolated_implies_affine_DH}
  Let $\comp$ be a compact complexity one $T$-space. If
  there are no isolated fixed points, then $M_{\mathrm{exc}} =
  \emptyset$ and the Duistermaat-Heckman
  function $DH \comp : \Phi(M) \to \R$ is the restriction of an affine
  function. 
\end{lemma}

\begin{proof}
  Since there are no isolated fixed points, Lemma
  \ref{lemma:fixed_point_exceptional} implies that there are no
  exceptional fixed
  points. Hence, by Lemma
  \ref{lemma:necessary_sufficient_existence_exceptional}, there are no
  exceptional sheets. By Lemma \ref{lemma:exceptional},
  $M_{\mathrm{exc}} = \emptyset$. Moreover, by Lemma \ref{lemma:one_chamber}, there
  is only one chamber of $\Phi(M)$. The result then follows by Remark \ref{rmk:DH_no_internal_walls}.
\end{proof}

\begin{lemma}\label{lemma:constant}
  Let $\comp$ be a compact complexity one $T$-space. If the
  Duistermaat-Heckman function $DH\comp$ is constant, then there are
  no singular values in the interior of $\Phi(M)$.
\end{lemma}
\begin{proof}
  We prove the contrapositive. Suppose that there is a singular value
  in the interior of $\Phi(M)$. Hence there exist two chambers
  $\mathfrak{c}_-$ and $\mathfrak{c}_+$ of $\Phi(M)$ that intersect in
  a facet $F$ that is not contained in the boundary of $\Phi(M)$. As
  in Intermezzo 2, we let $\xi \in \ell$ be a primitive element that
  is normal to the hyperplane supporting $F$ and points out of $\mathfrak{c}_-$. Moreover, we fix $x \in F$ that is a regular value for
  all the sheets $(N_{ij},\omega_{ij},\Phi_{ij})$ constructed in
  Intermezzo 1; we observe that $x$ lies in the interior of $F$ (see
  property \ref{item:11}).

  Since $F$ is not contained in the boundary
  of $\Phi(M)$, neither is $x$. In particular, if $x \in
  \Phi_{ij}(N_{ij})$, then $\Phi_{ij}(N_{ij})$ is not contained in the
  boundary of $\Phi(M)$. Hence, $(N_{ij},\omega_{ij},\Phi_{ij})$ is
  exceptional by Lemma \ref{lemma:exceptional_not_boundary}. By
  Proposition \ref{prop:regular_exceptional}, the complexity of
  $(N_{ij},\omega_{ij},\Phi_{ij})$ is strictly less than that of
  $\comp$. Since the complexity of $\comp$ is one, the complexity of
  $(N_{ij},\omega_{ij},\Phi_{ij})$ is zero. Hence, the codimension of
  $N_{ij}$ in $M$ equals 4. Therefore, the
  lowest order term in $\xi$ in the right hand side of equation
  \eqref{eq:2} is 
  \begin{equation}
    \label{eq:5}
    \xi \left(\sum\limits_{\{i,j \mid x \in \Phi_{ij}(N_{ij})\}} (\alpha_{ij1}\alpha_{ij2})^{-1}f_{ij} \right),
  \end{equation}
  \noindent
  where, for each $i,j$, the polynomial$f_{ij}$ and the integers
  $\alpha_{ij1},\alpha_{ij2}$ are as in the
  discussion leading up to Theorem \ref{Thm:
    GLS-wall-crossing-formula}. Since $x$ lies in the interior of
  $\Phi(M)$, it follows that $\alpha_{ij1}\alpha_{ij2} < 0$ and
  $f_{ij}(x) > 0$ for each $i,j$. In particular, the polynomial in
  equation \eqref{eq:5} is not identically zero and, therefore,
  neither is the right hand side of equation \eqref{eq:2}.

  Let $f_{\pm} : \g^* \to \R$ be the polynomial that, when restricted
  to $\mathfrak{c}_{\pm}$, equals \eqref{eq:1}. By \eqref{eq:5} and Theorem \ref{Thm:
    GLS-wall-crossing-formula}, $f_{+}$ and $f_-$ are not
  equal. Hence, since $f_{\pm}$ is the restriction of the
  Duistermaat-Heckman function $DH\comp$ to $\mathfrak{c}_{\pm}$,
  $DH\comp$ is not constant, as desired.
\end{proof}

The next result describes $DH\comp$ near a vertex of $\Phi(M)$ that
corresponds to a fixed surface. To this end, let $v \in
\Phi(M)$ be a vertex and let $\Sigma = \Phi^{-1}(v)$ be a fixed
surface. Let $\alpha_1,\ldots,\alpha_n$ be the isotropy weights of
$\Sigma$ (see Remark \ref{rmk weights of F}), labeled so that
$\alpha_n = 0$. Since $v$ is a vertex, by part \ref{item:1} of
Corollary \ref{Cor:ConeFace}, a sufficiently small neighborhood of $v$
in $\Phi(M)$ is of the form
$$ \Big\{ v + \sum\limits_{i=1}^{n-1}t_i \alpha_i \mid t_i \geq 0
\text{ sufficiently small} \Big\}. $$
\noindent
Moreover, by part \ref{item:2} of
Corollary \ref{Cor:ConeFace}, for each $i=1,\ldots, n-1$, the edge $e_i$ of
$\Phi(M)$ that is incident to $v$ is contained in the half-line
$\{v+t_i\alpha_i \mid t_i \geq 0\}$. Let $(M_i,\omega_i,\Phi_i)$ be
the sheet corresponding to the edge $e_i$ as in \eqref{eq:11} with
stabilizer $H_i$. By
Proposition \ref{prop:tall_vertex}, $\dim M_i = 4$ and $(M_i,\omega_i,\Phi_i)$ is a
compact complexity one Hamiltonian $T/H_i$-space. In what follows, we
identify $T/H_i \simeq S^1$. Let $N$ be the normal bundle of $\Sigma$ in $M$. Then $N$ splits
$T$-equivariantly as a direct sum
\begin{equation}
  \label{eq:25}
  N = L_1 \oplus \ldots \oplus L_{n-1}, 
\end{equation}
\noindent
where $L_i$ denotes the normal bundle of $\Sigma$ in $M_i$. 

\begin{lemma}\label{Lemma:DHnearfixedSurface}
  Let \(\comp\) be a compact complexity one $T$-space of dimension
  \(2n\) and let $v \in \Phi(M)$ be a vertex such that \(\Sigma =
  \Phi^{-1}(v)\) is a fixed surface. Let \(\alpha_{1}, \ldots,
  \alpha_{n-1}\) be the non-zero isotropy 
  weights of $\Sigma$. For all
  $t_1,\ldots, t_{n-1} \geq 0$ sufficiently small,
  \begin{equation}
    DH \comp \left(v+ \sum\limits_{i=1}^{n-1}t_i\alpha_{i}\right)=  \int\limits_\Sigma \omega-\sum\limits_{i=1}^{n-1}t_i c_1(L_i) \left[ 
      \Sigma\right],
  \end{equation}
  \noindent
  where $L_1,\ldots, L_{n-1}$ are as in \eqref{eq:25} and $c_1(L_i)$
  is the first Chern class of $L_i$ for $i=1,\ldots, n-1$. In
  particular, if $DH \comp$ attains its minimum at $v$, then
  \begin{equation}\label{eq: value c1 on Li}
    c_1(L_i)[\Sigma] \leq 0 \qquad \text{for all } i=1,\ldots, n-1.
  \end{equation}
\end{lemma}

\begin{proof}
  Since $v$ is a vertex and the complexity of $\comp$ is one, the restriction of $DH\comp$ to a
  sufficiently small neighborhood of $v$ is an affine function. Thus
  there exist real numbers $\beta_0,\beta_1,\ldots, \beta_{n-1}$ such
  that, for all
  $t_1,\ldots, t_{n-1} \geq 0$ sufficiently small,
  \begin{equation}
    \label{eq:26}
    DH \comp \left(v+ \sum\limits_{i=1}^{n-1}t_i\alpha_{i}\right) =
    \beta_0 + \sum\limits_{i=1}^{n-1} t_i\beta_i.
  \end{equation}
  \noindent
  In order to determine the constants $\beta_0,\beta_1,\ldots,
  \beta_{n-1}$, it suffices to understand the restriction of $DH
  \comp$ to elements of the form $v + t_j\alpha_j$ for $j=1,\ldots,
  n-1$. Fix $i=1,\ldots, n-1$. By Corollary
  \ref{cor:DH_boundary_comp_pres}, 
  \begin{equation}
    \label{eq:28}
    DH\comp (v+t_i\alpha_i) = DH(M_i,\omega_i,\Phi_i)(v +
    t_i\alpha_i).
  \end{equation}
  \noindent
  By \cite[Lemma 2.19]{karshon}, for
  all $t_i \geq 0$ sufficiently small,
  \begin{equation}
    \label{eq:29}
     DH(M_i,\omega_i,\Phi_i)(v +
     t_i\alpha_i) = \int\limits_{\Sigma} \omega_i - t_i
     c_1(L_i)[\Sigma].
  \end{equation}
  \noindent
  The result follows immediately by comparing equations \eqref{eq:26},
  \eqref{eq:28} and \eqref{eq:29}.
\end{proof}

The following result is an immediate consequence of piecewise
linearity of the Duistermaat-Heckman function of a compact complexity
one $T$- space and of Proposition \ref{prop:concave}.

\begin{corollary}\label{cor:DH_polytope}
  Let $\comp$ be a compact complexity one $T$-space. The subset
  \begin{equation}
    \label{eq:21}
    \{(x,t) \in \mathfrak{t}^* \times \R \mid x \in \Phi(M) \, , \, t
    \in [0,DH\comp(x)]\}
  \end{equation}
  \noindent
  is a convex polytope in $\g^* \times \R$.
\end{corollary}

The convex polytope of Corollary \ref{cor:DH_polytope} and the
Duistermaat-Heckman function of a compact complexity one $T$-space
are equivalent in the sense that knowing one allows to reconstruct the
other. To emphasize the combinatorial nature of the problem we study,
we introduce the following notion.

\begin{definition}\label{Def: DH-polytope}
  The convex polytope \eqref{eq:21} of a compact complexity one
  $T$-space $\comp$ is called the {\bf Duistermaat-Heckman polytope} of $\comp$.
\end{definition}

\begin{remark}\label{rmk:DH_polytope_induced}
  Suppose that $\comp$ is a compact complexity one $T$-space obtained
  by restricting a complexity zero $T \times S^1$-action on
  $(M,\omega)$ to the subtorus $T \times \{1\}$. The moment
  polytope of the complexity zero action need not agree with the
  Duistermaat-Heckman polytope of $\comp$. However, the two are
  related as follows. If $\Delta$ denotes the moment polytope of the
  complexity zero $T \times S^1$-action, we have that
  $$ \Delta = \{ (x,y)\in \Phi'(M) \times \R \mid y \in
  [p_{min}(x), p_{max}(x)] \}, $$
  \noindent
  whereas, combining Example \ref{exm:induced_DH} and
  \eqref{eq:21}, the Duistermaat-Heckman polytope of $\comp$ equals
  $$     \{(x,t) \in \mathfrak{t}^* \times \R \mid x \in \Phi(M) \, , \, t
  \in [0,p_{max}(x) - p_{min}(x)]\}. $$
  \noindent
  Finally, we observe that the latter can be obtained from the former
  by applying a piecewise integral affine transformation of $\g^*
  \oplus \R$, where the lattice is $\ell^* \oplus \Z$ (see Figure \ref{HexagonHouse}).
\end{remark}

\begin{figure}[htbp]
	\begin{center}
		\includegraphics[width=8cm]{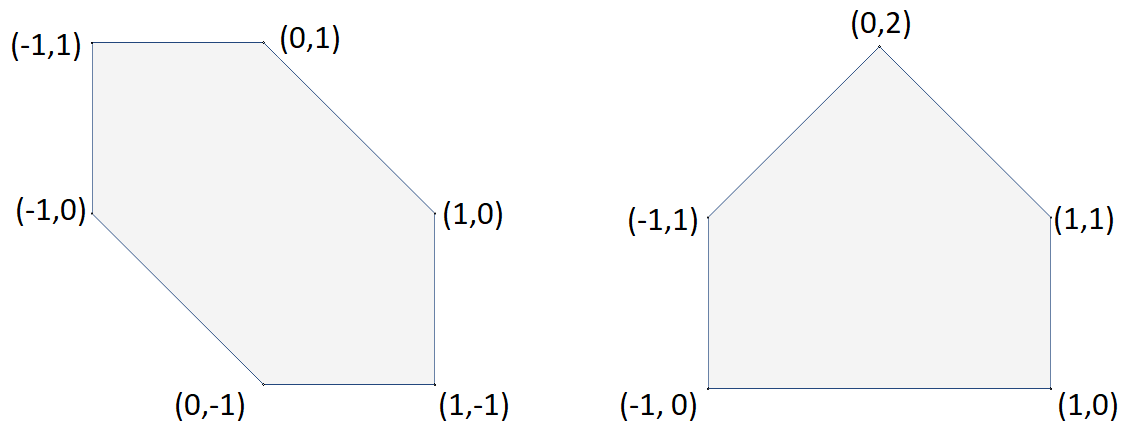}
		\caption{An example in which the moment polytope of
                  the compact symplectic toric manifold does not agree
                  with the Duistermaat-Heckman
                  polytope of the complexity one space constructed as
                  in Remark \ref{rmk:DH_polytope_induced}.}
		\label{HexagonHouse}
	\end{center}
\end{figure}

\subsection{Compact complexity preserving Hamiltonian $T$-spaces}\label{sec:compl-pres-comp}
In this section, we introduce a class of Hamiltonian $T$-spaces that
enjoy special properties, which are also enjoyed by compact
symplectic toric manifolds (see Corollary \ref{cor:comp_preserving},
Proposition \ref{prop:tall_delzant} and Corollary
\ref{cor:DH_boundary_comp_pres}). We begin
with the following result, which extends \cite[Corollary 3.1, 
Lemma 3.3 and Remark 3.4]{ss}.

\begin{proposition}\label{prop:tall_vertex}
  Let $\comp$ be a compact complexity $k$ $T$-space. If $N \subset
  M^T$ is a fixed submanifold with $\dim N = 2k$, then $\Phi(N)$ is a
  vertex of $\Phi(M)$. Moreover, for every face $\mathcal{F}$ of
  $\Phi(M)$ that contains $\Phi(N)$, the sheet
  $(M_{\mathcal{F}},\omega_{\mathcal{F}},\Phi_{\mathcal{F}})$ is
  stabilized by a connected subgroup and has complexity equal to $k$.
\end{proposition}

\begin{proof}
  Since $\dim N= 2k$, by Remark \ref{rmk:reg_irr_fixed},
  $(N,\omega_N)$ is regular and every point in $N$ is regular. Hence,
  by Remark \ref{rmk:regular_equivalent}, the
  $T$-action on the normal bundle to $N$ is toric. By Corollary
  \ref{Cor:ConeFace}, $\Phi(N)$ is a vertex of $\Phi(M)$.

  Let $\mathcal{F}$ be a face of $\Phi(M)$ that contains $\Phi(N)$ and
  let $p \in N \cap M_{\mathcal{F}}$. By Corollary
  \ref{cor:preimage_face}, there exist points arbitrarily close to $p$
  that have stabilizer equal to $H_{\mathcal{F}}$, the stabilizer of
  $(M_{\mathcal{F}},\omega_{\mathcal{F}},\Phi_{\mathcal{F}})$. Since
  $p$ is regular and all stabilizers in a regular local model
  are connected, $H_{\mathcal{F}}$ is connected by Theorem \ref{local
    normal form}. Finally, we observe that, since $(N,\omega_N)$
  is a fixed submanifold and $N \subseteq
  M_{\mathcal{F}}$, $(N,\omega_N)$ is a
  sheet in $(M_{\mathcal{F}},\omega_{\mathcal{F}},\Phi_{\mathcal{F}})$. Since
  $\dim N = 2k$ and $N \subset M^{T/H_{\mathcal{F}}}_{\mathcal{F}}$,
  by Proposition \ref{prop:regular_exceptional}, the complexity of
  $(M_{\mathcal{F}},\omega_{\mathcal{F}},\Phi_{\mathcal{F}})$ is at
  least $k$. On the other hand,
  $(M_{\mathcal{F}},\omega_{\mathcal{F}},\Phi_{\mathcal{F}})$ is a
  sheet in $\comp$ and the complexity of the latter is $k$. Hence, by
  Proposition \ref{prop:regular_exceptional}, the complexity of
  $(M_{\mathcal{F}},\omega_{\mathcal{F}},\Phi_{\mathcal{F}})$ is at
  most $k$. 
\end{proof}

\begin{corollary}\label{cor:comp_preserving}
  Let $\comp$ be a compact complexity $k$ $T$-space. The following are equivalent:
  \begin{enumerate}[label=(\roman*),ref=(\roman*),leftmargin=*]
  \item \label{item:14} for each face $\mathcal{F}$ of $\Phi(M)$, the
    complexity of the sheet $(M_{\mathcal{F}},\omega_{\mathcal{F}},\Phi_{\mathcal{F}})$ equals $k$;
  \item \label{item:12} for each face $\mathcal{F}$ of $\Phi(M)$ of codimension $r$, $M_{\mathcal{F}}$ has maximal dimension,
    i.e., $\dim M_{\mathcal{F}} = 2n-2r$;
  \item \label{item:13} for each vertex $v$ of $\Phi(M)$, $\Phi^{-1}(v)$ has maximal
    dimension, i.e., $\dim \Phi^{-1}(v) = 2k$.
  \end{enumerate}
\end{corollary}

\begin{proof}
  For any face $\mathcal{F}$ of
  $\Phi(M)$ of codimension $r$, the complexity of
  $(M_{\mathcal{F}},\omega_{\mathcal{F}},\Phi_{\mathcal{F}})$ equals
  that of $\comp$ if and only if $\dim M_{\mathcal{F}} = \dim M -
  2r$. This shows that \ref{item:14} and \ref{item:12} are equivalent. Since vertices are faces of maximal codimension, \ref{item:12} implies
  \ref{item:13}. The converse follows from Proposition \ref{prop:tall_vertex}.
\end{proof}


Motivated by Corollary \ref{cor:comp_preserving}, we introduce the
following terminology.

\begin{definition}\label{defn:complexity_preserving}
  A compact Hamiltonian $T$-space is said to be {\bf complexity
    preserving} if it satisfies any (and hence all) of the conditions
  \ref{item:14} -- \ref{item:13} in Corollary \ref{cor:comp_preserving}.
\end{definition}

compact complexity preserving Hamiltonian $T$-spaces generalize compact
symplectic toric manifolds.

\begin{remark}\label{rmk:complexity_preserving_restriction}
  By Corollary \ref{cor:comp_preserving}, if $\comp$ is a compact complexity preserving Hamiltonian $T$-space,
  then, for every face $\mathcal{F}$ of $\Phi(M)$, so is $(M_{\mathcal{F}},\omega_{\mathcal{F}},\Phi_{\mathcal{F}})$.
\end{remark}

The following result describes a property of the Duistermaat-Heckman
function of compact complexity preserving Hamiltonian $T$-spaces and is an immediate consequence of Propositions \ref{prop:DH_boundary},
\ref{prop:tall_vertex} and Remark
\ref{rmk:complexity_preserving_restriction}.

\begin{corollary}\label{cor:DH_boundary_comp_pres}
  Let $\comp$ be a compact complexity $k$ $T$-space and suppose that $N
  \subset M^T$ is a fixed submanifold with $\dim N = 2k$. For
  any face $\mathcal{F}$ of $\Phi(M)$ containing $\Phi(N)$, 
  \begin{equation}
    \label{eq:27}
    DH \comp|_{\mathcal{F}} =
     DH(M_{\mathcal{F}},\omega_{\mathcal{F}},\Phi_{\mathcal{F}}), 
  \end{equation}
  \noindent
  where  $(M_{\mathcal{F}},\omega_{\mathcal{F}},\Phi_{\mathcal{F}})$
  is the sheet given in \eqref{eq:11}. In particular, if $\comp$ is complexity preserving, then
  \eqref{eq:27} holds for all faces of $\Phi(M)$.
\end{corollary}

To conclude this section, we prove the following result, which we need
in Section \ref{sec:duist-heckm-funct}.

\begin{lemma}\label{lemma:comp_pres_no_iso}
  Let $\comp$ be a compact complexity preserving Hamiltonian $T$-space
  of positive complexity. If there are no singular values of $\Phi$ contained in
  the interior of $\Phi(M)$,
  then the action has no isolated fixed points. 
\end{lemma}

\begin{proof}
  Suppose that $p \in M^T$ is isolated. By condition \ref{item:13} in Corollary \ref{cor:comp_preserving},
  $\Phi(p)$ is not a vertex. Hence, if
  $\mathcal{F}$ is the face of smallest dimension in which $\Phi(p)$
  lies, then $\dim \mathcal{F} \geq 1$. By part \ref{item:1}
  of Corollary \ref{Cor:ConeFace}, an open neighborhood of
  $\Phi(p)$ in $\Phi(M)$ can be identified with an open neighborhood of $(0,0) \in
  \R^{\dim \mathcal{F}} \times \R^{d-\dim \mathcal{F}}$ in $\R^{\dim
    \mathcal{F}} \times \mathcal{C}'_p$, where $\mathcal{C}'_p$ is the
  proper cone in the proof of part \ref{item:2} of Corollary
  \ref{Cor:ConeFace}. We observe that, since $\dim \mathcal{F} \geq
  1$, the subset $\{0\} \times \mathcal{C}'_p$ intersects the interior
  of $\R^{\dim \mathcal{F}} \times \mathcal{C}'_p$. Choose $d - \dim
  \mathcal{F}$ linearly independent isotropy weights of $p$ that span
  $\mathcal{C}'_p$ and, if needed, complete this set with $\dim \mathcal{F}
  - 1$ linearly
  independent isotropy weights of $p$ whose span is contained in $\R^{\dim
    \mathcal{F}}$. The span of these isotropy weights
  $\alpha_1,\ldots, \alpha_{d-1}$ satisfies
  \begin{equation}
    \label{eq:31}
    \left(\Phi(p) + \R_{\geq 0} \langle \alpha_1,\ldots, \alpha_{d-1} \rangle \right) \cap
    \mathrm{Int}(\Phi(M)) \neq \emptyset,
  \end{equation}
  \noindent
  where $\mathrm{Int}(\Phi(M))$ denotes the interior of
  $\Phi(M)$.

  By Theorem \ref{local normal form}, we may identify a $T$-invariant
  neighborhood of $p$ with a $T$-invariant neighborhood of $0 \in
  \C^n$ so that $\Phi$ becomes the map
  \begin{equation}
    \label{eq:9}
     (z_1,\ldots, z_n) \mapsto \pi\sum\limits_{i=1}^n \alpha_i |z_i|^2
     +\Phi(p). 
  \end{equation}
  \noindent
  Moreover, by part \ref{item:1} of Corollary \ref{Cor:ConeFace}, an
  open neighborhood of $\Phi(p)$ in $\Phi(M)$ can be identified with
  an open neighborhood of $\Phi(p)$ in the image of the map of
  equation \eqref{eq:9}.

  By \eqref{eq:31}, the affine hyperplane $ \Phi(p) + \R \langle \alpha_1,\ldots, \alpha_{d-1} \rangle $ intersects $\mathrm{Int}(\Phi(M))$. All values in this intersection
  have a one-dimensional stabilizer: this is because 
  $$ \dim\left(\mathrm{Ann}\left(\R \langle \alpha_1 \rangle \right) \cap
    \ldots \cap \mathrm{Ann}\left(\R \langle \alpha_{d-1}
      \rangle\right)\right) = 1, $$
  \noindent
  since $\alpha_1,\ldots, \alpha_{d-1}$ are linearly
  independent. Hence, there is a singular value of $\Phi$ in
  $\mathrm{Int}(\Phi(M))$, a contradiction.
\end{proof}

\subsubsection{Moment polytopes for compact complexity preserving
  Hamiltonian $T$-spaces}\label{sec:moment-polyt-monot}
In this subsection, we characterize the moment map image of complexity
preserving compact Hamiltonian $T$-spaces (see Proposition
\ref{prop:tall_delzant} below). To this end, given a polytope $\Delta
\subset \g^*$ and a vertex $v \in \Delta$, we say that $\Delta$ is
{\bf smooth} at $v$ if

\begin{itemize}[leftmargin=*]
\item there are exactly $d$ edges $e_1,\ldots, e_d$ that are incident
  to $v$, and
\item there exists a basis $\alpha_1,\ldots, \alpha_d$ of $\ell^*$
  such that $\alpha_i$ is a tangent vector to the edge $e_i$ for all
  $i=1,\ldots, d$.
\end{itemize}
A polytope $\Delta \subset \g^*$ is smooth at $v$ if and only if the collection
of inward (or outward) normals to the facets of $\Delta$ that
contain $v$ can be chosen to be a basis of $\ell$. We say that a polytope $\Delta$ is
\textbf{Delzant} if it is smooth at every vertex. \\

The moment map image of a compact symplectic toric manifold is
 a Delzant polytope and, conversely, every Delzant polytope
arises as such an image (see \cite{delzant}). In general, this fails
to be true in higher complexity. However, under the additional
hypothesis of complexity preserving, the following result holds.

\begin{proposition}\label{prop:tall_delzant}
  The moment map image of a compact complexity preserving Hamiltonian
  $T$-space is a Delzant polytope in $\g^*$.

  Conversely, for every
  Delzant polytope $\Delta$ in $\g^*$ and for every
  integer $k \geq 0$, there exists a compact complexity preserving Hamiltonian
  $T$-space $\comp$ of complexity $k$ such that $\Phi(M) = \Delta$.
\end{proposition}
\begin{proof}
  Let $\comp$ be a compact complexity preserving Hamiltonian
  $T$-space of complexity $k$. Let $v \in \Phi(M)$ be a vertex. By Corollary
  \ref{cor:comp_preserving}, $\dim \Phi^{-1}(v) = 2k$. Let $\alpha_1,\ldots,\alpha_n$
  be the isotropy weights of $\Phi^{-1}(v)$ (see Remark \ref{rmk weights of F}). Since
  $\dim \Phi^{-1}(v) =2k$, precisely $k$ weights are zero. Without loss of generality, we may assume
  that $\alpha_{n-k+1},\ldots, \alpha_n = 0$. By part \ref{item:1} of Corollary \ref{Cor:ConeFace} 
  an open neighborhood of $v$ in $\Phi(M)$ looks like
  an open neighborhood of $0$ in 
  \begin{equation}\label{cone F}
    \R_{\geq 0}\text{-span} \left\lbrace \alpha_{1},\dots,\alpha_n\right\rbrace=
    \R_{\geq 0}\text{-span} \left\lbrace \alpha_{1},\dots \alpha_{n-k}\right\rbrace.
  \end{equation}
  \noindent
  Since the complexity of $\comp$ is $k$, $d = \dim T = 
  n-k$. Hence, by equation \eqref{cone F}, there are exactly
  $d$ edges that are incident to $v$. Moreover, by Remark \ref{rmk local normal form},
  the $\Z$-span of $\alpha_1,\ldots,\alpha_d$ equals $\ell^*$. Hence,
  $\Phi(M)$ is smooth at $v$ and the first statement follows.

  Conversely, fix an integer $k \geq 0$ and suppose that $\Delta$ is a
  Delzant polytope in $\g^*$. By the classification of compact symplectic toric manifolds
  (see \cite{delzant}), there exists a compact
  complexity zero $T$-space $(M',\omega',\Phi')$ such that $\Phi'(M') =
  \Delta$. Let $(M'',\omega'')$ be a closed symplectic
  manifold of dimension $2k$. Consider the $T$-action on $M:= M'\times M''$ given by
  taking the product of the above $T$-action on $M'$ with the
  trivial $T$-action on $M''$. This action is Hamiltonian for the
  symplectic form $\omega$ obtained by summing the pullbacks to $M$ of $\omega'$ and $\omega''$
  along the projections. A moment map for this $T$-action is given by the pullback to $M$ of
  $\Phi'$ along the projection $M \to M'$; we denote this moment map
  by $\Phi$. Then $(M,\omega, \Phi)$ is a compact complexity preserving
  complexity $k$ $T$-space with moment map image given by
  $\Delta$, as desired.
\end{proof}

\subsection{Compact tall complexity one $T$-spaces} \label{sec:tall-comp-compl}
In this section we introduce an important class of compact complexity one $T$-spaces.

\begin{defin}\label{def tall}
  A compact complexity one $T$-space $\comp$ is called \textbf{tall} if no reduced space is a point.
\end{defin}

To shed light on Definition \ref{def tall} we observe that, if $\comp$ is a compact 
complexity one $T$-space, then
the reduced space $M_x$ is homeomorphic to a closed, connected orientable
surface for any $x \in \Phi(M)_{\mathrm{reg}}$ (see Section \ref{sec:orbital-moment-map}). If $\comp$ is tall, then this holds {\em for all} $x \in
\Phi(M)$ (see \cite[Proposition 6.1]{kt1}). Moreover, the following
result holds.

\begin{proposition}\label{prop:tall=complexity_preserving}
  A compact complexity one $T$-space is tall if and only if it is
  complexity preserving.
\end{proposition}

\begin{proof}
  Let $\comp$ be a compact tall complexity one $T$-space $\comp$ and let $v
  \in \Phi(M)$ be a vertex. By Remark
  \ref{rmk:reg_irr_fixed}, $\Phi^{-1}(v)$ is either a fixed point or a 
  fixed surface. Since the reduced space at $v$ can be identified with
  $\Phi^{-1}(v)$ and since $\comp$ is tall,  $\Phi^{-1}(v)$ has
  dimension two. Hence, since $\comp$ has complexity one, it
  satisfies property \ref{item:13} in Corollary
  \ref{cor:comp_preserving}; therefore, it is complexity preserving. Conversely, if $\comp$ is complexity preserving, then it satisfies
  property \ref{item:14} in Corollary
  \ref{cor:comp_preserving}. Hence, by \cite[Corollary 2.6]{kt3}, no
  reduced space is a point and $\comp$ is tall. 
\end{proof}

In \cite{kt1, kt2, kt3} the authors classify tall
complexity one $T$-spaces\footnote{In {\em loc. cit.} the authors consider a more general class of tall complexity one spaces, namely those for which $M$ is 
connected but not necessarily compact
and such that there exists an open convex set $\mathcal{T}\subseteq
\g^*$ containing the image of the moment map with the property that $\Phi\colon M 
\to \mathcal{T}$ is proper. However, we state all results in {\em
  loc. cit.} only in the compact case.}. Below we recall this
classification. Henceforth, we fix a
compact complexity one $T$-space $\comp$. As a consequence of \cite[Corollary
9.7]{kt1} or \cite{li}, any two reduced spaces of
$\comp$ are homeomorphic. This motivates introducing the following notion.

\begin{definition}\label{defn:genus}
  The \textbf{genus} of a compact tall complexity one $T$-space 
  $\comp$ is the genus of the reduced space $M_x$ for any $x \in \Phi(M)$.
\end{definition}

The following result is a stepping stone for the classification of
compact tall complexity one $T$-spaces (see Proposition 2.2 in \cite{kt2}, Proposition 1.2 and
Remark 1.9 in
\cite{kt3})

\begin{proposition}\label{prop 2.2}
  If $\comp$ is a compact tall complexity one $T$-space, then there exist a closed oriented surface $\Sigma$ and a map
  $f\colon M/T \to \Sigma$ such that 
  \begin{equation}\label{map prop 2.2}
    (\overline{\Phi},f)\colon M/T \longrightarrow \Phi(M)\times \Sigma
  \end{equation}
  is a homeomorphism and the restriction $f : \Phi^{-1}(x)/T \to
  \Sigma$ is orientation-preserving for any $x \in \Phi(M)$. Given two such maps $f$ and $f'$, there exists an orientation-preserving homeomorphism 
  $\xi\colon \Sigma' \to \Sigma$ such that $f$ is homotopic to $\xi
  \circ f'$ through maps that induce homeomorphisms 
  $M/T\to \Phi(M)\times \Sigma$. 
\end{proposition}

By Proposition \ref{prop 2.2}, the genus of
$\comp$ is the genus of $\Sigma$. \\

The next invariant of tall complexity one $T$-spaces is related to the
exceptional orbits (see Remark \ref{rmk:exc_orbits}), and is introduced below. To this end, we observe
that given a closed surface $\Sigma$ and a map $f : M/T \to 
\Sigma$ as in Proposition \ref{prop 2.2}, its restriction to
$M_{\mathrm{exc}}$ makes $(\overline{\Phi},f) : M_{\mathrm{exc}} \to
\Phi(M) \times \Sigma$ injective.

\begin{definition}\label{defn:painting}
  Let $(M,\omega, \Phi), (M',\omega',\Phi')$ be compact tall complexity one
  $T$-spaces and let $\Sigma, \Sigma'$ be closed oriented surfaces.
  \begin{itemize}[leftmargin=*]
  \item A {\bf painting of $\boldsymbol{(M,\omega,\Phi)}$} is a map $f :
    M_{\mathrm{exc}} \to \Sigma$ such that
    $(\overline{\Phi},f) : M_{\mathrm{exc}} \to \Phi(M) \times \Sigma$
    is injective.
  \item An {\bf isomorphism of exceptional orbits} is a homeomorphism
    $i : M_{\mathrm{exc}} \to M'_{\mathrm{exc}}$ satisfying $\overline{\Phi} = \overline{\Phi}' \circ i$ that sends each orbit
    to an orbit with the same symplectic slice representation. 
  \item A painting $f : M_{\mathrm{exc}} \to \Sigma$ of $(M,\omega,
    \Phi)$ is {\bf equivalent} to a painting $f' : M'_{\mathrm{exc}} \to \Sigma'$ of $(M',\omega',
    \Phi')$ if there exists an isomorphism of exceptional orbits $i :
    M_{\mathrm{exc}} \to M_{\mathrm{exc}'}$ and an
    orientation-preserving homeomorphism $\xi : \Sigma \to \Sigma'$
    such that $f' \circ i$ and $\xi \circ f$ are homotopic through paintings. 
  \end{itemize}
\end{definition}

By Proposition \ref{prop 2.2}, we can associate an equivalence class
of paintings to a compact tall complexity one $T$-space (see \cite[Section 2]{kt2}). For our purposes, it is useful to introduce the following terminology.

\begin{definition}\label{defn:trivial_painting}
  Let $\comp$ be a tall, compact complexity one $T$-space. The equivalence class of paintings $[f]$ associated to $\comp$ is
  {\bf trivial} if there exists a painting $f: M_{\mathrm{exc}}
  \to \Sigma$ representing $[f]$ that is constant on each connected
  component of $M_{\mathrm{exc}}$.
\end{definition}

The classification of compact tall complexity one $T$-spaces is as
follows.

\begin{theorem}\label{sue and yael}
{\em (Karshon--Tolman, Theorem 1 in \cite{kt2}, and Theorem 1.8 and
  Remark 1.9 \cite{kt3})}
Two compact tall complexity one $T$-spaces are isomorphic if and only
if they have equal genera, equal 
Duistermaat-Heckman measures, and equivalent paintings. 
\end{theorem}

The invariants of a compact tall complexity one $T$-space
determine the moment map image, as it is the
support of the Duistermaat-Heckman measure (cf. \cite[Theorem 1.8]{kt3}). 

\section{Compact monotone Hamiltonian $T$-spaces} \label{sec:monot-comp-hamilt}
In this section we use ideas and techniques from equivariant
cohomology, referring the reader to \cite{gs-supersymmetry} for
details and background. 
\subsection{The weight sum formula}\label{sec:weight-sum-formula}
In this paper we are mostly concerned with compact Hamiltonian
$T$-spaces satisfying the following condition. 
\begin{defin}\label{def monotone}
A symplectic manifold $(M,\omega)$ is \textbf{monotone} if there exists $\lambda \in \R$ such that 
$c_1=\lambda[\omega]$, where $c_1$ is the first Chern class of $(M,\omega)$. 
It is \textbf{positive monotone} if $\lambda>0$. 
\end{defin}

\begin{remark}\label{rmk:monotone_invariant_symplecto}
  \mbox{}
  \begin{enumerate}[label=(\arabic*),ref=(\arabic*),leftmargin=*]
  \item \label{item:7} If
    $(M,\omega)$ is compact and monotone, since $[\omega] \neq 0$,
    then $\lambda$ in Definition \ref{def monotone} is unique.
  \item \label{item:17} Let $(M,\omega)$ be a monotone
    symplectic manifold and let $\Psi : (M',\omega') \to (M,\omega)$
    be a symplectomorphism. Since $\Psi$ pulls back almost complex
    structures that are compatible with $\omega$ to almost complex
    structures that are compatible with $\omega'$, $(M',\omega')$
    is monotone. Moreover, if $(M,\omega)$ is compact and if $\lambda,
    \lambda' \in \R$ are such that $c_1 =
    \lambda[\omega]$ and $c_1' = \lambda'[\omega']$, then $\lambda = \lambda'$.
  \end{enumerate}
\end{remark}

If $(M,\omega)$ is such that $H^2(M;\R)=\R$, then it is monotone (e.g.\ $\C 
P^n$). In general, (positive) monotonicity is very restrictive. In the
presence of a Hamiltonian torus action, 
the following result holds.

\begin{prop}\label{assumptions moment map c1}
  If $(M,\omega)$ is compact and monotone, and admits
  an effective Hamiltonian $T$-action, then $(M,\omega)$ is positive monotone.
\end{prop}

\begin{proof}
The proof follows \textit{mutatis mutandis} that of \cite[Lemma 5.2]{gvhs}, in which it is assumed that $M^{S^1}$ is 
discrete (see \cite[Definition 3.1]{gvhs}). Let $H \leq T$ be a one
dimensional subtorus and let $\phi : M \to \mathfrak{h}^*$ be the
induced moment map. We identify $H \simeq S^1$ and consider
$(M,\omega, \phi)$ as a Hamiltonian $S^1$-space. We observe that
$\phi\colon M \to (\text{Lie}(S^1))^*$ is a Morse-Bott function;
moreover, by 
\eqref{def moment map}, the isotropy weights in the positive normal bundle to a fixed point are 
positive (cf. \cite[Proof of Lemma 5.2]{gvhs}). Therefore, if $F_{\min}$ (respectively $F_{\max}$) denotes a fixed component on which $\phi$ attains its minimum 
(respectively maximum), all the isotropy weights in the normal bundle to $F_{\min}$ (respectively $F_{\max}$) are 
positive (respectively negative). Moreover, even if some of the isotropy weights of $p_{\min}\in F_{\min}$ (respectively at 
$p_{\max}\in F_{\max}$) are zero, by the effectiveness of the action
some of them must be different from zero. 
Hence, the sum of the isotropy weights of $p_{\min}$ (resp.\
$p_{\max}$) is strictly positive (respectively strictly 
negative). 
To complete the proof, it is enough to consider the equivariant extensions of $[\omega]$ and 
$c_1$ in the equivariant cohomology ring of $M$, which are respectively 
$[\omega-\phi]$ and $c_1^{S^1}$,
and to compare them
at $p_{\min}$ and $p_{\max}$ to deduce that $\lambda$ must be positive
(see equation (5.1) in \cite[Proof of Lemma 5.2]{gvhs}).
\end{proof}

Throughout this paper, a Hamiltonian $T$-space $\comp$ is {\bf
  monotone} if $(M,\omega)$ is. The following result is an immediate
consequence of Remark \ref{rmk:monotone_invariant_symplecto} and Proposition
\ref{assumptions moment map c1}.
\begin{corollary}\label{cor:rescaling}
  If $(M,\omega, \Phi)$ is a compact monotone Hamiltonian $T$-space, then there exists a unique $\lambda >0$ such that
  $c_1 = [\lambda \omega]$.
\end{corollary}

The next proposition extends
\cite[Remark 3.11]{gvhs}.
\begin{prop}\label{prop:weight_sum}
  If $\comp$ is a compact Hamiltonian $T$-space with $c_1 =
  [\omega]$, then there exists a unique $w \in \g^*$ such that the moment map $\widetilde{\Phi}:=\Phi + w$ 
  satisfies the \textbf{weight sum formula}, i.e., 
 \begin{equation}\label{weight sum formula}
   \widetilde{\Phi}(p) =-\sum_{j=1}^n \alpha_j,\quad \text{for all }p\in M^T\,,
 \end{equation}
 where $\alpha_1,\ldots,\alpha_n \in \ell^*$ 
 are the isotropy weights of $p$. 
\end{prop}

\begin{proof}
  Since $c_1=[\omega]$ and since the action is Hamiltonian, there
  exists a unique $w \in \g^*$ such that 
  $$
  c_1^T+w=[\omega-\Phi].
  $$
  Thus the moment map $\widetilde{\Phi} = \Phi+w$ satisfies $c_1^T=[\omega-\widetilde{\Phi}]$. 
  Since $M$ is compact and the action is Hamiltonian, there exists $p
  \in M^T$. The equality in \eqref{weight sum formula} is obtained by
  comparing these two equivariant cohomology classes at $p \in M^T$
  and observing that $c_1^T(p)=\sum_{j=1}^n\alpha_j$.
\end{proof}

\begin{remark}\label{rmk:v_invariant_isomorphism}
  Let $(M,\omega,\Phi)$ be a compact Hamiltonian $T$-space with $c_1 =
  [\omega]$. If $(M',\omega', \Phi')$ is isomorphic to
  $(M,\omega,\Phi)$, then $c_1' = [\omega']$ by Remark
  \ref{rmk:monotone_invariant_symplecto}. Let $\Psi : (M,\omega,\Phi) \to
  (M',\omega',\Phi')$ be an isomorphism. By Remark \ref{rmk:equivariant},
  $\Psi$ is equivariant. Hence, by \eqref{weight sum formula}, if $w, w' \in \mathfrak{t}^*$ are as in Proposition \ref{prop:weight_sum} for $\Phi$ and $\Phi'$
  respectively, then $w = w'$. 
\end{remark}

\begin{defin}\label{monotone ham space}
  A compact monotone Hamiltonian $T$-space $\comp$ is \textbf{normalized} if 
  \begin{enumerate}[label=(\roman*), ref=(\roman*), leftmargin=*]
  \item\label{item:5} $c_1=[\omega]$, and 
  \item\label{item:6} the moment map $\Phi$ satisfies the weight sum
    formula \eqref{weight sum formula}.
  \end{enumerate}
  In this case we call $\comp$ a \textbf{normalized monotone} Hamiltonian $T$-space. 
\end{defin}

Since the isotropy weights of a fixed point lie in $\ell^*$, Proposition \ref{prop:weight_sum} has the following
immediate consequence.

\begin{corollary}\label{cor:monotone_image_isolated}
  If $\comp$ is a normalized monotone Hamiltonian $T$-space, then
  $[\omega] \in H^2(M;\Z)$ and, for any $p
  \in M^T$, $\Phi(p) \in \ell^*$.
\end{corollary}

The following result is an immediate
consequence of Corollary
\ref{cor:rescaling} and Proposition \ref{prop:weight_sum}.

\begin{corollary}\label{cor:mono_normalized}
  If $\comp$ is a compact monotone Hamiltonian $T$-space, then there exist unique
  $\lambda >0$ and $w \in \mathfrak{t}^*$ such that $(M,\lambda
  \omega, \lambda \Phi + w)$ is normalized monotone. 
\end{corollary}

Classifying compact monotone Hamiltonian $T$-space is almost equivalent to classifying 
normalized monotone Hamiltonian $T$-spaces. More precisely, the
following holds.

\begin{lemma}\label{lemma:monotone_normalized}
  Let $(M,\omega,\Phi)$, $(M',\omega',\Phi')$ be compact monotone Hamiltonian
  $T$-spaces. Let $\lambda, \lambda' > 0$ and $v, v' \in \mathfrak{t}^*$
  be as in Corollary \ref{cor:mono_normalized}. Then $(M,\omega,\Phi)$
  and $(M',\omega',\Phi')$ are isomorphic if and only if $\lambda =
  \lambda'$, $v = v'$ and $(M,\lambda
  \omega, \lambda\Phi + v)$ is isomorphic to $(M',\lambda'
  \omega', \lambda'\Phi' + v')$. 
\end{lemma}

\begin{proof}
  Suppose that $(M,\omega,\Phi)$
  and $(M',\omega',\Phi')$ are isomorphic. Let $\Psi : (M,\omega) \to
  (M',\omega')$ be a symplectomorphism such that $\Phi' \circ \Psi =
  \Phi$. By part \ref{item:17} of Remark
  \ref{rmk:monotone_invariant_symplecto} and by Remark
  \ref{rmk:v_invariant_isomorphism}, $\lambda = \lambda'$ and $v =
  v'$. Hence, $\Psi : (M,\lambda \omega) \to (M',\lambda'\omega')$ is
  a symplectomorphism and $(\lambda'\Phi' + v') \circ \Psi = \lambda \Phi + v$, i.e.,
  $\Psi$ is an isomorphism between $(M,\lambda
  \omega, \lambda \Phi + v)$ and $(M',\lambda'
  \omega', \lambda'\Phi' + v')$. The converse is entirely analogous and is
  left to the reader.
\end{proof}

\subsection{Moment polytopes of monotone complexity preserving Hamiltonian $T$-spaces}\label{sec:moment-polyt-tall}
We recall that a polytope $\Delta$ in $\g^*$ can be described by its minimal representation (see Section \ref{sec:conventions}): 
$$\Delta=\bigcap_{i=1}^l \, \{w\in \g^* \mid \langle w,\nu_i \rangle \geq
c_i\}$$
for some inward normals $\nu_1,\ldots, \nu_l \in \g$ and
 $c_1,\ldots, c_l \in \R$. Such a polytope $\Delta$ is {\bf integral} if its
vertices belong to $\ell^*$.

\begin{remark}\label{rmk:integral}
  If $\Delta$ is integral, then it is possible to choose the inward
  normal $\nu_i$ so that it is a primitive element of $\ell$, for every $i=1,\ldots,l$. 
  The corresponding constants $c_i$'s are therefore uniquely determined
  by this choice of $\nu_i$'s. 
\end{remark}

\begin{defin}\label{defn:reflexive}
A polytope $\Delta \subset \g^*$ is 
\textbf{reflexive} if it is integral and $\nu_i \in \ell$
in its minimal representation is primitive with corresponding $c_i=-1$, for all $i=1,\ldots, l$.
\end{defin}

The following result is an immediate consequence of Definition \ref{defn:reflexive} and is stated below without proof (see \cite[Proposition 7.3]{hnp}).

\begin{lemma}\label{lemma:reflexive_interior}
  For any reflexive polytope the origin is the only interior lattice point.
\end{lemma}

Lemma \ref{lemma:reflexive_interior} and a result of Lagarias and 
Ziegler \cite{lz}
imply the following result.

\begin{corollary}\label{cor:finite}
  Up to the action of $\mathrm{GL}(\ell^*)$, there are only finitely many reflexive
  polytopes in $\g^*$.
\end{corollary}

For instance, there are sixteen two-dimensional reflexive polytopes (see \cite[Figure 2]{prv}), five of which are also 
Delzant (see Figure \ref{smooth}).
\begin{figure}[htbp]
\begin{center}
\includegraphics[width=12cm]{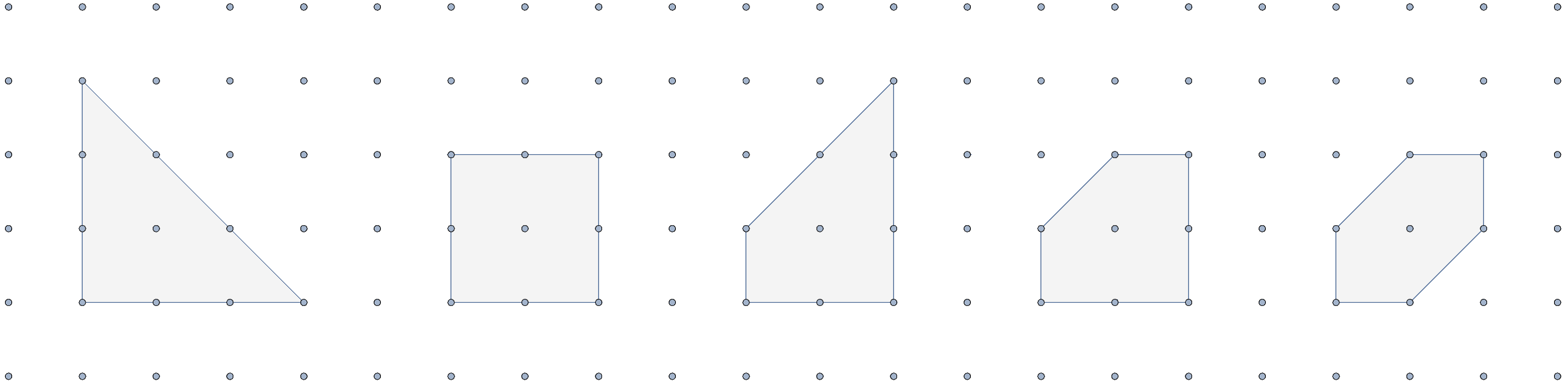}
\caption{The five reflexive Delzant polygons.}
\label{smooth}
\end{center}
\end{figure}

For a rational polytope $\Delta \subset \g^*$, given a vertex $v$ of $\Delta$  
one can choose the vectors $\alpha_i$'s in \eqref{eq:17}, which support the edges coming out of $v$, to be
primitive elements of $\ell^*$; these vectors are uniquely determined and referred to as the \textbf{weights of the vertex} $v$.

Reflexive Delzant polytopes are in particular rational and they can be
characterized in terms of the weights of their vertices. More
precisely, the following result, proved in various contexts by various
authors, holds (see, in particular, \cite[Prop 1.8]{ep}, \cite[Sect.\ 3]{mcduff 
displacing} and 
\cite[Prop.\ 2.12]{gvhs}).

\begin{proposition}\label{prop reflexive eq}
Let $\Delta\subset \g^*$ be a $d$-dimensional Delzant polytope. The following conditions are equivalent:
\begin{enumerate}[label=(\roman*),ref=(\roman*),leftmargin=*]
\item \label{item:8} $\Delta$ is a reflexive polytope.
\item \label{item:9} $\Delta$ satisfies the \textbf{weight sum formula}, i.e., for
  each vertex $v \in \Delta$,
  \begin{equation}\label{wsf polytope}
    v = - \sum_{j=1}^d \alpha_j\,,
  \end{equation}
  where $\alpha_1,\ldots,\alpha_d$ are the weights of $v$.
\end{enumerate}
\end{proposition}

\begin{rmk}
In \cite[Prop.\ 2.12]{gvhs} it is assumed that the origin is an
interior point of $\Delta$ to prove that \ref{item:9} implies \ref{item:8}. 
However this follows by \eqref{wsf polytope}. Indeed, consider the multiset $\mathcal{W}$ of all the primitive 
vectors appearing as weights of vertices of $\Delta$.
Note that, if $\alpha \in \mathcal{W}$ has multiplicity $r$, then so
does $-\alpha$. 

Hence the sum of all the weights in $\mathcal{W}$ is $
0 \in \g^*$. 
Therefore, if $\Delta$ satisfies \eqref{wsf polytope}, then
\begin{equation}\label{sum vertices zero}
\sum_{v\in \mathcal{V}} v = 0\,,
\end{equation}
where $\mathcal{V}$ is the set of vertices of $\Delta$. Since $\Delta$ is the
convex hull of its vertices, the interior points of $\Delta$ are
precisely those that can be written as follows:
\begin{equation}
  \label{eq:4}
  \sum_{v\in \mathcal{V}}\lambda_v v \quad \text{with } \lambda_v>0\;\;\text{for all  }v\in \mathcal{V}\;\;\text{and}\;\; \sum_{v\in 
    \mathcal{V}}\lambda_v=1 \,.
\end{equation}
Let $\lambda_v=\frac{1}{|\mathcal{V}|}$ for all
$v\in \mathcal{V}$. Then by \eqref{sum vertices zero}, we have  $0 =\sum_{v\in \mathcal{V}} \lambda_v
v$. Hence, \eqref{eq:4} yields that $0$ belongs to the interior of $\Delta$.
\end{rmk}

The following technical lemma regarding reflexive Delzant polytopes is
used extensively in Sections \ref{sec:stuff-that-we} and
\ref{sec:an-expl-real} below.

\begin{lemma}\label{lemma:weight_edge_out_facet}
  Let $\Delta$ be a reflexive Delzant polytope in $\g^*$, let
  $\mathcal{F}$ be a facet of $\Delta$ supported on the affine
  hyperplane $\{w \in \g^* \mid \langle w, \nu \rangle = -1\}$, let $v$ be a vertex of
  $\Delta$ in $\mathcal{F}$, and let $\alpha_1,\ldots, \alpha_d \in \ell^*$ be the weights of $v$ ordered so that
  \begin{equation}
    \label{eq:70}
    \mathcal{F} \subset v + \R_{\geq 0} \langle \alpha_1,\ldots,\alpha_{d-1}\rangle.
  \end{equation}
  Then $\langle \alpha_d, \nu \rangle = 1$. Moreover, if $e$ is the
  edge that is incident to $v$ and comes out of $\mathcal{F}$, then
  there exists $t_{\max} \in \Z_{> 0}$ such that $e = \{v + t \alpha_d
  \mid 0 \leq t \leq t_{\max}\}$.
\end{lemma}
\begin{proof}
  By \eqref{eq:70}, $\langle \alpha_i, \nu \rangle = 0$ for all
  $i=1,\ldots, d-1$. By Proposition \ref{prop reflexive eq}, $\Delta$
  satisfies the weight sum formula at $v$. Since $v \in \mathcal{F}$,
  $\langle \alpha_d, \nu \rangle = 1$.

  If $e$ is an edge of $\Delta$ as in the statement, then there exists
  $t_{\max} > 0$ such that $e = \{v + t \alpha_d
  \mid 0 \leq t \leq t_{\max}\}$. It remains to show that $t_{\max}$ is a
  positive integer. To this end, we observe that $v':= v + t_{\max}
  \alpha_d$ is a vertex of $\Delta$. Since $\Delta$ is reflexive, $v'
  \in \ell^*$. By definition of weight, $\alpha_d$ is
  primitive in $\ell^*$. Hence, $t_{\max} \in \mathbb{Z}_{>0}$. 
\end{proof}

To conclude this section, we look at the relation between normalized
monotone complexity preserving $T$-spaces and reflexive Delzant
polytopes in $\g^*$. We start with the following strengthening of Proposition \ref{prop:tall_delzant} under the
additional assumption of monotonicity. 

\begin{proposition}\label{prop:mom_map_image}
The moment map image of a normalized monotone complexity
preserving $T$-space is a reflexive Delzant polytope in $\g^*$.
Conversely, for every reflexive Delzant polytope $\Delta$ in $\g^*$ and for
every integer $k \geq 0$, there exists a normalized monotone complexity
preserving Hamiltonian $T$-space $\comp$ of complexity $k$ such that
$\Phi(M) = \Delta$.
\end{proposition}

Before proving Proposition \ref{prop:mom_map_image}, we recall the following result,
stated below without proof (see \cite[Proposition 3.10]{gvhs}).

\begin{proposition}\label{prop:monotone_toric}
  Let $\comp$ be a compact symplectic toric manifold. If $\Phi(M)$ is
  a reflexive Delzant polytope, then $\comp$ is normalized monotone.
\end{proposition}

\begin{proof}[Proof of Proposition \ref{prop:mom_map_image}]
Let $\comp$ be a normalized monotone complexity
preserving $T$-space. By Proposition \ref{prop:tall_delzant},
$\Phi(M)$ is a Delzant polytope in $\g^*$. It remains to show
that $\Phi(M)$ is reflexive. Since $\comp$ is complexity preserving, by part
\ref{item:1} of Corollary \ref{Cor:ConeFace}, the weights of $\Phi(M)$
at a vertex $v$ are equal to the non-zero weights of any $p \in
\Phi^{-1}(v)$. Since $\comp$ is normalized monotone, $\Phi$ satisfies the weight sum 
formula \eqref{weight sum formula}. Hence, $\Phi(M)$ also satisfies the
weight sum formula \eqref{wsf polytope} and so, by Proposition \ref{prop
  reflexive eq}, it is reflexive.

Conversely, let $\Delta$ be a reflexive Delzant polytope in $\g^*$ and
let $k \geq 0$ be an integer. We adapt the second half of the proof of Proposition \ref{prop:tall_delzant} (and fix the notation therein), to show that we can make appropriate choices so that the
resulting complexity preserving $T$-space of complexity $k$ is normalized monotone. Since
$\Delta$ is reflexive, by Proposition \ref{prop:monotone_toric}, $(M', 
\omega')$ is normalized monotone. Choose $M'' = \C
P^k$ and $\omega''$ to be a monotone symplectic form on $\C
P^k$ such that $c_1(\C P^k) = [\omega'']$. The complexity
preserving $T$-space $(M,\omega,\Phi)$ constructed in the second
half of the proof of Proposition \ref{prop:tall_delzant} has
complexity $k$ and is normalized monotone, as desired.
\end{proof}

We finish with the following simple, useful result.

\begin{lemma}\label{lemma:normalized_complexity_preserving}
  Let $\comp$ be a compact monotone complexity
  preserving $T$-space. If $\Phi(M)$ is reflexive Delzant,
  then $\comp$ is normalized monotone. 
\end{lemma}

\begin{proof}
  If $\dim M =0$, there is nothing to prove, so we may assume $\dim M
  > 0$. Let $k$ be the complexity of $\comp$. Since $(M,\omega)$ is
  monotone, by the proof of Proposition \ref{prop:weight_sum}, there exists a
  unique constant $w \in \g^*$
  such that
  $$c^T_1 + w= \lambda [\omega - \Phi].$$
  We fix a
  vertex $v \in \Phi(M)$ and a fixed point $p \in \Phi^{-1}(v)$. Evaluating both sides of the above
  displayed equality at $p$,
  \begin{equation}
    \label{eq:81}
    \sum\limits_{i=1}^{n-k} \alpha_i + c= - \lambda v,
  \end{equation}
  \noindent
  where $\alpha_1,\ldots, \alpha_{n-k}$ are the non-zero
  weights of $p$. Since $\comp$ is complexity preserving, the non-zero
  isotropy weights of $p$ in $M$ are precisely the
  weights of $v$ in $\Phi(M)$. Hence, by Proposition \ref{prop reflexive eq},
  \eqref{eq:81} gives that $-v +w = - \lambda v $. Since this equality
  holds for any vertex $v \in \Phi(M)$, we have $w = 0$ and
  $\lambda = 1$, as desired.
\end{proof}

\section{Complete invariants of compact monotone tall
  complexity one $T$-spaces}\label{sec:stuff-that-we}

\subsection{The genus and a minimal facet}\label{sec:duist-heckm-funct}
In this section we explore the first consequences of the combination
of tallness and monotonicity of compact complexity one $T$-spaces,
recovering and extending some of the results in \cite{ss}. To this
end, let $\comp$ be a monotone tall complexity one $T$-space of
dimension $2n$ and let $v \in \Phi(M)$ be a vertex. Let $N = L_1\oplus \ldots \oplus
L_{n-1}$ be the normal bundle to $\Sigma:=\Phi^{-1}(v)$ in $M$ together with its
$T$-equivariant splitting into $T$-invariant complex line bundles as in
\eqref{eq:25}.

\begin{lemma}\label{lemma:chern}
  Let $\comp$ be a compact monotone tall complexity one $T$-space of
  dimension $2n$ and
  let $v \in \Phi(M)$ be a vertex that attains the minimum of
  $DH\comp$. If $c_1(\Sigma), c_1(L_i)$ denote the first Chern
  class of $\Sigma$ and of the complex line bundle $L_i$ for any
  $i=1,\ldots, n-1$ respectively, then
  \begin{equation}
    \label{eq:65}
    c_1(\Sigma)[\Sigma] > - \sum\limits_{i=1}^{n-1}c_1(L_i)[\Sigma].
  \end{equation}
  Moreover, the genus of $\comp$ in the sense of Definition
  \ref{defn:genus} equals zero.
\end{lemma}

\begin{proof}
  Since $(M,\omega)$ is monotone, by Proposition \ref{assumptions
    moment map c1}, it is positive monotone. Hence, since $\Sigma$ is
  a symplectic submanifold of $(M,\omega)$, 
  \begin{equation}
    \label{eq:30}
    0 < c_1[\Sigma] = c_1(\Sigma)[\Sigma] + c_1(N)[\Sigma] =
    c_1(\Sigma)[\Sigma] + \sum\limits_{i=1}^{n-1}c_1(L_i)[\Sigma],
  \end{equation}
  \noindent
  whence \eqref{eq:65} holds. By Lemma \ref{Lemma:DHnearfixedSurface}, the right hand side of \eqref{eq:65} is
  non-negative, which implies that $\Sigma$ is diffeomorphic to a
  sphere, as desired.
\end{proof}

Lemma \ref{lemma:chern} is a stepping stone towards the following
important result.

\begin{proposition}\label{prop:minimum_facet}
  Let $\comp$ be a compact monotone tall complexity one
  $T$-space of dimension $2n$. There exists a facet $\mathcal{F}$ of $\Phi(M)$ such
  that $DH(M_{\mathcal{F}},\omega_{\mathcal{F}},\Phi_{\mathcal{F}})$ is constant and equal to the minimum
  of $DH\comp$, where
  $(M_{\mathcal{F}},\omega_{\mathcal{F}},\Phi_{\mathcal{F}})$ is
  defined as in \eqref{eq:11}. Moreover, for any vertex $v \in \mathcal{F}$, there exist $n-2$
  non-zero isotropy weights $\alpha_1,\ldots,
  \alpha_{n-2}$ of $\Phi^{-1}(v)$ such that
  \begin{itemize}[leftmargin=*]
  \item $\mathcal{F} \subset v + \R_{\geq 0} \langle \alpha_1, \ldots,
    \alpha_{n-2} \rangle$, and
  \item the self-intersection of $\Phi^{-1}(v)$ in $\Phi^{-1}(e_j)$
    equals zero for all $j=1,\ldots, n-2$, where $e_j \subset
    \mathcal{F}$ is the edge incident to $v$ with tangent vector given by $\alpha_j$.
  \end{itemize}
\end{proposition}

\begin{proof}
  By Proposition
  \ref{prop:tall=complexity_preserving}, $\comp$ is complexity
    preserving. Hence, by Corollary \ref{cor:DH_boundary_comp_pres},
    for any face $\tilde{\mathcal{F}}$ of $\Phi(M)$
    $$
    DH(M_{\tilde{\mathcal{F}}},\omega_{\tilde{\mathcal{F}}},\Phi_{\tilde{\mathcal{F}}})
    = DH \comp|_{\tilde{\mathcal{F}}}.$$
    \noindent
  Therefore, to prove the first statement, it suffices to prove that there exists a facet
  $\mathcal{F}$ of $\Phi(M)$ such that $DH \comp|_{\mathcal{F}}$ is
  constant and equal to the minimum of $DH\comp$ -- see Figure \ref{DHGraphNearVertex}.

  By Corollary \ref{cor:min_DH_function}, the minimum of $DH \comp$ is
  attained at a vertex of $\Phi(M)$, say $v_0$. Let $\alpha_1,\ldots,
  \alpha_{n-1}$ be the non-zero isotropy weights of the fixed surface $\Sigma:=\Phi^{-1}(v_0)$. By
  Lemma \ref{Lemma:DHnearfixedSurface},
  $$ DH\comp \left( v_0 + \sum\limits_{i=1}^{n-1}t_i \alpha_i\right) =
  \int\limits_{\Sigma} \omega - \sum\limits_{i=1}^{n-1}t_i
  c_1(L_i)[\Sigma] $$
  \noindent
  for all $t_1,\ldots, t_{n-1}
  \geq 0$ sufficiently small. By Lemma
  \ref{lemma:chern}, $ 2 > - \sum\limits_{i=1}^{n-1}c_1(L_i)[\Sigma]$;
  moreover, by Lemma \ref{Lemma:DHnearfixedSurface}, $c_1(L_i)[\Sigma]\leq
  0$ for all $i=1,\ldots, n-1$. Thus
  at least $n-2$ of $c_1(L_1)[\Sigma], \ldots,
  c_1(L_{n-1})[\Sigma]$ vanish; there is no loss of generality in
  assuming that $c_1(L_i)[\Sigma]=0$ for all $i=1,\ldots, n-2$. Hence,
  $$ DH\comp \left( v_0 + \sum\limits_{i=1}^{n-1}t_i \alpha_i\right) =
  \int\limits_{\Sigma} \omega - t_{n-1}
  c_1(L_{n-1})[\Sigma] $$
  \noindent
  for all $t_1,\ldots, t_{n-1}
  \geq 0$ sufficiently small. In particular, the restriction of $DH \comp$ to a sufficiently small
  neighborhood of $v_0$ in
  $$ (v_0 + \R_{\geq 0} \langle \alpha_1, \ldots, \alpha_{n-2} \rangle)
  \cap \Phi(M) $$
  \noindent
  is constant and equal to $DH\comp(v_0)$, which, by assumption, is the
  minimum of $DH \comp$. Let $\mathcal{F}$ be the facet of $\Phi(M)$
  that is contained in $v_0 + \R_{\geq 0} \langle \alpha_1, \ldots,
  \alpha_{n-2} \rangle$. Since $DH \comp$ is concave (see Proposition
  \ref{prop:concave}), since $DH \comp$ attains its minimum at $v_0$,
  and since $v_0 \in \mathcal{F}$, $DH\comp|_{\mathcal{F}}$
  is constant and equal to the minimum of $DH \comp$. Moreover, since
  $c_1(L_i)[\Sigma] = 0$, the
  self-intersection of $\Sigma$ in $\Phi^{-1}(e_i)$ is zero for each
  $i= 1,\ldots, n-2$.

  It remains to show that the bullet points in the statement hold for
  all vertices of $\mathcal{F}$. However, since $DH\comp|_{\mathcal{F}}$
  is constant and equal to the minimum of $DH \comp$, it follows that
  any vertex of $\mathcal{F}$ is a minimum of $DH \comp$. Hence, the
  above argument gives the desired result.
\end{proof}

\begin{figure}[h]
	\begin{center}
		\includegraphics[width=4cm]{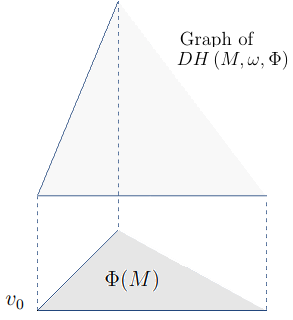}
		\caption{}
		\label{DHGraphNearVertex}
	\end{center}
\end{figure}

A facet as in Proposition
\ref{prop:minimum_facet} plays an important role throughout Section \ref{sec:stuff-that-we}. 

\begin{definition}\label{defn:minimal_facet}
  Let $\comp$ be a compact monotone tall complexity one
  $T$-space. A facet of $\Phi(M)$ satisfying the conclusions of
  Proposition \ref{prop:minimum_facet} is called a {\bf minimal
    facet} of $\Phi(M)$ and denoted by $\mathcal{F}_{\mathrm{min}}$. Given a
  minimal facet $\mathcal{F}_{\min}$, the sheet 
  corresponding to $\mathcal{F}_{\min}$ is denoted by
  $(M_{\min},\omega_{\min},\Phi_{\min})$.   
\end{definition}

We observe that, in spite of the notation,
$(M_{\min},\omega_{\min},\Phi_{\min})$ clearly depends on
$\mathcal{F}_{\min}$. However, we trust that the notation does not
cause confusion.


\begin{corollary}\label{cor:min_facet_no_iso}
  Let $\comp$ be a compact monotone tall complexity one
  $T$-space. If $\mathcal{F}_{\min}$ is a minimal facet of
  $\Phi(M)$, then $M_{\min}$ contains no isolated fixed point of
  $\comp$. 
\end{corollary}
\begin{proof}
  By Remark \ref{rmk:complexity_preserving_restriction} and
  Proposition \ref{prop:tall=complexity_preserving},
  $(M_{\mathrm{min}},\omega_{\mathrm{min}},\Phi_{\mathrm{min}})$ is a
  compact tall complexity one
  $T/H_{\mathcal{F}_{\min}}$-space. By Corollary
  \ref{cor:DH_boundary_comp_pres} and Proposition
  \ref{prop:minimum_facet},
  $DH(M_{\mathrm{min}},\omega_{\mathrm{min}},\Phi_{\mathrm{min}})$ is
  constant. Hence, by Lemma \ref{lemma:constant}, there are no
  singular values in the (relative) interior of $\Phi_{\mathrm{min}}(M_{\min})$. Thus, by
  Lemma \ref{lemma:comp_pres_no_iso},
  $(M_{\mathrm{min}},\omega_{\mathrm{min}},\Phi_{\mathrm{min}})$ has
  no isolated fixed points for the
  $T/H_{\mathcal{F}_{\min}}$-action and, hence, for the $T$-action.
\end{proof}

To conclude this section, we prove that certain self-intersections of
the pre-image of vertices in a minimal facet are independent of the
vertices. To this end, we say that an edge $e$ of a polytope 
$\Delta$ {\bf comes out of a facet} $\mathcal{F}$ if it is not
contained in $\mathcal{F}$ but it is incident to
a vertex of $\Delta$ contained in $\mathcal{F}$.

\begin{lemma}\label{lemma:s_invariant}
  Let $\comp$ be a compact monotone tall complexity one
  $T$-space and let $\mathcal{F}_{\min}$ be a minimal facet of
  $\Phi(M)$. There exists $s \in \Z$ such that, for any vertex $v \in
  \mathcal{F}_{\min}$, the self-intersection of $\Phi^{-1}(v)$ in
  $\Phi^{-1}(e)$ equals $s$, where $e$ is the edge of $\Phi(M)$ that
  comes out of $\mathcal{F}_{\min}$ and is incident to $v$.
\end{lemma}

\begin{proof}
  Let $v_1,v_2 \in \mathcal{F}_{\min}$ be vertices and let $e_1,e_2$
  be the edges that come out of $\mathcal{F}_{\min}$ that are incident
  to $v_1$ and to $v_2$ respectively. By Proposition
  \ref{prop:minimum_facet}, $DH\comp(v_1) = DH \comp(v_2)$. Let
  $\Sigma_i:= \Phi^{-1}(v_i)$ for $i=1,2$. Thus, by Lemma
  \ref{Lemma:DHnearfixedSurface}, $[\omega](\Sigma_1) =
  [\omega](\Sigma_2)$. Since $(M,\omega)$ is monotone, it is positive
  monotone by Proposition \ref{assumptions moment map c1}. Hence, $c_1(\Sigma_1)
  = c_1(\Sigma_2)$. Moreover, by Proposition \ref{prop:minimum_facet},
  the only self-intersection of $\Sigma_i$ that can possibly be
  different from zero is that in $\Phi^{-1}(e_i)$, for $i=1,2$. Hence,
  since $\Sigma_1 \simeq \Sigma_2$, the result follows.
\end{proof}

\subsection{A characterization of isolated fixed points}\label{sec:the-image-of}

Henceforth, we assume that $\comp$ is a normalized monotone tall complexity one $T$-space unless otherwise stated. Moreover, we fix a
minimal facet $\mathcal{F}_{\min}$ of $\Phi(M)$ (which exists by
Proposition \ref{prop:minimum_facet}). By Proposition
\ref{prop:mom_map_image}, $\Phi(M)$ is a reflexive Delzant
polytope. Hence, by Definition \ref{defn:reflexive}, there exists a unique
primitive $\nu_{\min} \in \ell$ such that
\begin{equation}
  \label{eq:10}
  \begin{split}
    \mathcal{F}_{\min} & \subset \{w \in \mathfrak{t}^* \mid \langle w,\nu_{\min}
    \rangle = -1\}, \\
    \Phi(M) & \subset \{w \in \mathfrak{t}^* \mid \langle w,\nu_{\min}
    \rangle \geq -1 \}.
  \end{split}
\end{equation}
Finally, if $v \in \mathcal{F}_{\min}$ is any vertex and $\alpha_1,\ldots,\alpha_{n-1}$
are the non-zero isotropy weights of $\Phi^{-1}(v)$, then, by
Proposition \ref{prop:minimum_facet}, we may assume that the weights are
ordered so that
\begin{equation}
  \label{eq:59}
  \mathcal{F}_{\min} \subset v + \R_{\geq 0} \langle
  \alpha_1,\ldots, \alpha_{n-2} \rangle. 
\end{equation}
In particular, by Lemma
\ref{lemma:weight_edge_out_facet},
\begin{equation}
  \label{eq:16}
  \langle \alpha_{n-1}, \nu_{\min} \rangle = 1.
\end{equation}

\begin{proposition}\label{prop:iso_fixed_pts_monotone}
  Let $\comp$ be a normalized monotone tall complexity one $T$-space of dimension $2n$. If $p \in M^T$ is
  isolated, then there exists an edge $e$ of $\Phi(M)$
  that comes out of $\mathcal{F}_{\min}$ such that 
  $\Phi(p)$ is (the only element) in the
  intersection of $e$ with the linear hyperplane $ \{w \in \mathfrak{t}^* \mid \langle w,\nu_{\min}
  \rangle = 0 \}$.

  Moreover, if $v \in \mathcal{F}_{\min}$ is the vertex
  that is incident to $e$ and if $\alpha_1,\ldots,\alpha_{n-1}$
  are the non-zero isotropy weights of $\Phi^{-1}(v)$
  ordered so that \eqref{eq:59} holds, 
  then the isotropy weights of $p$ are
  \begin{equation}
    \label{eq:32}
    \alpha_1,\ldots,\alpha_{n-2}, \alpha_{n-1},-\alpha_{n-1}.
  \end{equation}
\end{proposition}

\begin{figure}[htbp]
	\begin{center}
		\includegraphics[width=5.5cm]{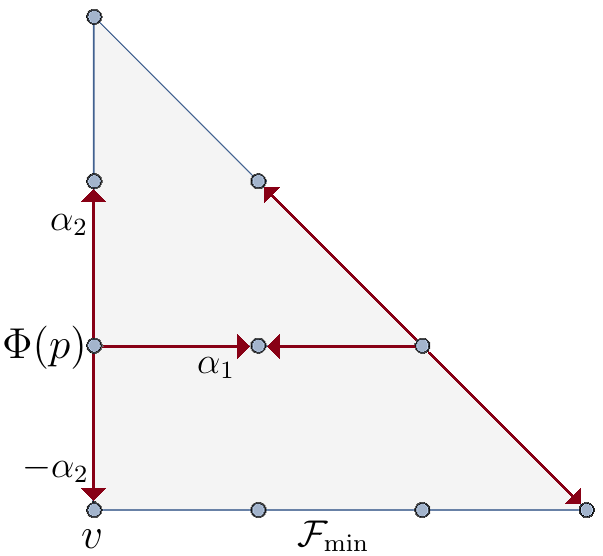}
		\caption{A schematic illustration of Proposition
                  \ref{prop:iso_fixed_pts_monotone}. The non-zero isotropy weights of
                  $\Phi^{-1}(v)$ are precisely $\alpha_1$ and
                  $\alpha_2$.}
		\label{isolatedfixedpoints}
	\end{center}
\end{figure}

\begin{proof}
  Since $\comp$ is normalized monotone, by Corollary
  \ref{cor:monotone_image_isolated}, $\Phi(p) \in
  \ell^*$. Hence, since $\nu_{\min} \in \ell$, $\langle \Phi(p), \nu_{\min}
  \rangle \in \Z$. Moreover, by Corollary \ref{cor:min_facet_no_iso},
  $\Phi(p) \notin \mathcal{F}_{\min}$; hence, by \eqref{eq:10}, $\langle \Phi(p), \nu_{\min}
  \rangle$ is a non-negative integer. Let $\beta_1,\ldots, \beta_n \in \ell^*$ be the isotropy weights of
  $p$. By Lemma
  \ref{lemma:special_weight}, there exists an isotropy weight $\beta$
  of $p$ such that $\langle  \beta, \nu_{\min} \rangle < 0$. Without
  loss of generality, we may assume that $\beta_{n} = \beta$. Let
  $(N,\omega_N,\Phi_N)$ be the sheet along $\beta_n$ given by Lemma
  \ref{lemma:sheet_weight}, let 
  $$H = \exp\left(\{ \xi \in \g \mid \langle \beta_n,\xi \rangle \in
    \Z\}\right)$$
  be its stabilizer, and let $q \in M^T \cap N$ be a fixed
  point satisfying the conclusions of Corollary
  \ref{cor:downward_pointing_sheet}, i.e.,
  \begin{itemize}[leftmargin=*]
  \item $\Phi(q) = \Phi_N(q)$ is a global extremum of
    $\Phi_N$,
  \item $-\beta_n$ is an isotropy weight of $q$, and $\alpha_1$ $\alpha_2$  $-\alpha_2$  $\mathcal{F}_{\min}$
  \end{itemize}
  \begin{equation}
    \label{eq:60}
    \langle \Phi(q), \nu_{\min}\rangle <
    \langle \Phi(p),\nu_{\min} \rangle.
  \end{equation}
  We split the proof in two cases: first, we assume that $\langle \Phi(p), \nu_{\min}
  \rangle$ is minimal among isolated fixed points and, second, we
  deduce the general case from this special one. \\

  {\bf Case 1:} We suppose that $\langle \Phi(p), \nu_{\min}
  \rangle$ is minimal among isolated fixed points. By
  \eqref{eq:60}, the fixed point $q$ is not isolated. Hence, it lies
  on a fixed surface. By Proposition \ref{prop:tall_vertex}, $v$ is a vertex of $\Phi(M)$. Let
  $\alpha_1,\ldots, \alpha_{n-1}$ be the non-zero isotropy weights of
  $\Phi^{-1}(\Phi(q))$ ordered so that $\alpha_{n-1} = - \beta_n$. By
  Proposition \ref{prop:tall_delzant}, $\alpha_1,\ldots,
  \alpha_{n-1}$ are a basis of $\ell^*$. Hence, $\beta_n$ is a
  primitive element in $\ell^*$. Moreover, since $H$ must be one of the stabilizers of dimension $n-2$ for points sufficiently
  close to $q$ in $M$, it follows that $\dim N = 4$ and that $\Phi(p)
  + \R \langle \beta_n\rangle$ contains an edge of
  $\Phi(M)$. Hence, since $\comp$ is tall, $\Phi(p)$
  lies in the (relative) interior of this edge. 

  Hence, by Corollary \ref{Cor:ConeFace}, there exists precisely one
  $i=1,\ldots, n-1$ such that $\beta_i$ is a multiple of $\beta_n$;
  without loss of generality, we assume that $i=n-1$. By  Remark \ref{rmk local normal form}, the
  $\Z$-span of $\beta_1,\ldots, \beta_{n-2},\beta_n$ equals
  $\ell^*$. Hence, by a dimension count, $\beta_1,\ldots,
  \beta_{n-2},\beta_n$ are linearly independent. We
  claim that $\langle \beta_j, \nu_{\min} \rangle \geq 0$
  for all $j =
  1,\ldots, n-2$. Suppose, on the contrary, that there exists an index
  $j=1,\ldots, n-2$ such that $\langle \beta_j, \nu_{\min}
  \rangle <
  0$. By Corollary \ref{cor:downward_pointing_sheet} applied to $\beta_j$ and
  by the above argument, $\Phi(p)$
  must lie in the (relative) interior of an edge of $\Phi(M)$ that is
  contained in $\Phi(p)
  + \R \langle \beta_j\rangle$. (Observe that this uses the fact that $\langle \Phi(p), \nu_{\min}
  \rangle$ is minimal among all isolated fixed points.) Since any point in a convex polytope
  is contained in the (relative) interior of at most one edge, it
  follows that $\R \langle \beta_j\rangle = \R \langle
  \beta_n\rangle$, which contradicts the linear independence of $\beta_1,\ldots,
  \beta_{n-2},\beta_n$.

  Since $\beta_{n-1}$ is a multiple of $\beta_n$, since $\beta_n$
  is primitive, and since $\Phi(p)$ lies in the (relative) interior of
  an edge of $\Phi(M)$, there exists a  positive integer $\lambda$
  such that $\beta_{n-1} = -\lambda \beta_n$. Since $\comp$ is normalized, the weight
  sum formula at $p$ 
  \begin{equation}
    \label{eq:34}
    \Phi(p) = -\sum\limits_{j=1}^{n} \beta_j
  \end{equation}
  \noindent
  implies that
  \begin{equation}
    \label{eq:35}
    0 \leq \langle \Phi(p), \nu_{\min} \rangle = \underbrace{- \sum\limits_{j=1}^{n-2} \langle \beta_j ,
      \nu_{\min} \rangle}_{\leq 0} + \underbrace{(\lambda - 1)}_{\geq 0}
    \underbrace{\langle \beta_n, \nu_{\min} \rangle}_{< 0} \leq 0.
  \end{equation}
  \noindent
  Therefore $\langle \Phi(p), \nu_{\min} \rangle= 0$, $\lambda = 1$ and $\langle \beta_j,
  \nu_{\min} \rangle = 0$ for all $j=1,\ldots, n-2$. Since $\langle \Phi(q),
  \nu_{\min} \rangle < \langle \Phi(p), \nu_{\min} \rangle$ by \eqref{eq:60}, since
  $\langle \Phi(q), \nu_{\min} \rangle \in \Z$ by Corollary \ref{cor:monotone_image_isolated}, and
  by \eqref{eq:10}, we have that $\langle \Phi(q),
  \nu_{\min} \rangle = -1$, i.e., $v:=\Phi(q)$ is a vertex of
  $\mathcal{F}_{\min}$. Moreover, since $\alpha_{n-1} = -\beta_n$ and
  since $\langle \beta_n, \nu_{\min} \rangle < 0$, the
  edge incident to $v$ contained in $\Phi(p) + \R \langle \beta_n
  \rangle$ comes out of $\mathcal{F}_{\min}$. Since $\beta_{n-1} = -\lambda \beta_n$, then
  $\beta_{n-1} = \alpha_{n-1}$. 
  We observe that the set (!)
  $\{\beta_1,\ldots,\beta_{n-2}\}$ is precisely the multiset of
  isotropy weights for the $T$-action on the normal bundle to the
  $T$-invariant submanifold $N$ at the point $p$. Hence, this set
  equals $\{\alpha_1,\ldots,\alpha_{n-2}\}$ modulo $\Z \langle \beta_n
  \rangle = \Z \langle \alpha_{n-1} \rangle$, since the latter is the
  multiset of isotropy weights for the $T$-action on the normal bundle
  to the
  $T$-invariant submanifold $N$ at another point. 
  The affine hyperplane $v + \R \langle \alpha_1,\ldots,\alpha_{n-2}\rangle$ contains
  $\mathcal{F}_{\min}$ by \eqref{eq:59}. Hence, $\langle \alpha_j,
  \nu_{\min} \rangle = 0$ for all $j =1,\ldots, n-2$. Since $\langle \alpha_{n-1},
  \nu_{\min}\rangle > 0$ and $\langle \beta_j,
  \nu_{\min} \rangle = 0$ for all $j=1,\ldots, n-2$,
  then $\{\beta_1,\ldots,\beta_{n-2}\} =
  \{\alpha_1,\ldots,\alpha_{n-2}\}$. This completes the proof of the
  result under the hypothesis that $\langle \Phi(p), \nu_{\min}
  \rangle$ is minimal among all isolated fixed points. \\

  {\bf Case 2:} To conclude the proof, it suffices to show that, if $p \in M^T$ is isolated,
  then $\langle \Phi(p), \nu_{\min}
  \rangle$ is minimal. Suppose not; then, by the above argument and
  since $\langle \Phi(p), \nu_{\min}
  \rangle \in \Z_{\geq 0}$, there exists $p' \in M^T$ such
  that $\langle \Phi(p'), \nu_{\min}
  \rangle >0$ and minimal among fixed points with positive pairing. Let $\beta_1',\ldots, \beta_n' \in \ell^*$ be the
  isotropy weights of $p'$. By Lemma
  \ref{lemma:special_weight}, we may assume that $\langle \beta_n',
  \nu_{\min} \rangle < 0$. We consider the sheet
  $(N',\omega_{N'},\Phi_{N'})$ along $\beta_n'$ given by Lemma
  \ref{lemma:sheet_weight} and the point $q' \in M^T \cap N'$
  satisfying the conclusions of Corollary
  \ref{cor:downward_pointing_sheet}. In analogy with \eqref{eq:60}, $\langle \Phi(q'), \nu_{\min} \rangle <
  \langle \Phi(p'), \nu_{\min} \rangle$. Hence, by minimality, $q'$ is either
  isolated and satisfies $\langle \Phi(q'), \nu_{\min} \rangle = 0$, or
   is not isolated. In either
  case, by the arguments used in the special case above, $\beta'_n \in
  \ell^*$ is primitive, $\dim N' = 4$ and $\Phi(p')$ lies
  in the (relative) interior of an edge of $\Phi(M)$. In particular, there is precisely one other
  isotropy weight of $p'$ that is collinear with $\beta'_n$, say
  $\beta'_{n-1}$. Moreover, as above, $\langle \beta'_j, \nu_{\min} \rangle \geq 0$ for all $j =
  1,\ldots, n-2$. Set $\beta'_{n-1} = - \lambda'
  \beta_n$ for some positive integer $\lambda'$\color{black}; since
  $\beta'_n$ is primitive, $\lambda'\geq 1$. We use again the weight sum formula \eqref{eq:34}
  and, in analogy with \eqref{eq:35}, we obtain the following absurd string
  of inequalities
  $$ 0 < \langle \Phi(p'), \nu_{\min} \rangle = \underbrace{- \sum\limits_{j=1}^{n-2} \langle \beta'_j ,
      \nu_{\min} \rangle}_{\leq 0} + \underbrace{(\lambda' - 1)}_{\geq 0}
    \underbrace{\langle \beta'_n, \nu_{\min} \rangle}_{< 0} \leq 0.$$
\end{proof}

\subsection{The space of exceptional orbits and the
  painting}\label{sec:intern-walls-except}
Motivated by Proposition \ref{prop:iso_fixed_pts_monotone}, our first
aim is to prove properties of an isolated fixed point 
with isotropy weights $\alpha_1,\ldots,
\alpha_{n-2},\alpha_{n-1},-\alpha_{n-1}$ in a complexity one
$T$-space (see Lemma \ref{lemma:local_description_exceptional} below). By Remark \ref{rmk local
  normal form}, $\alpha_1,\ldots, \alpha_{n-1}$ form a basis of
$\ell^*$; therefore, by a dimension count, $\alpha_j$ is primitive for all
$j=1,\ldots, n-1$. We define the following subgroups of $T$:
$$
H:=\exp\left( \{\xi \in \g \mid \langle \alpha_i , \xi \rangle = 0, \text{  for all  } i=1,\ldots, n-2\}\right),
$$
which is of dimension 1, and 
$$
T':=\exp\left( \{\xi \in \g \mid \langle \alpha_{n-1} , \xi \rangle = 0\}\right),
$$
which is of codimension 1. 
Observe that $T \simeq T' \times H$; moreover, we use the given inner
product to identify the duals $\mathfrak{h}^*, (\g')^*$ of the Lie
algebras of $H$ and $T$ with $\R \langle \alpha_{n-1}\rangle$ and $\R \langle \alpha_1,\ldots, \alpha_{n-2}
\rangle$ respectively. 

We start by looking at the local model determined by the above
isotropy weights (see Section \ref{sec:exceptional-orbits}). We consider the
following $T$-action on $\C^n$
\begin{equation}\label{eq:55bis}
  \begin{split}
    \exp(\xi)\cdot & (z_1,\ldots,z_{n-2},z_{n-1},z_n)= \\
    & (e^{ 2\pi i \langle \alpha_1 , \xi
      \rangle}z_1,\ldots,
    e^{ 2\pi i \langle \alpha_{n-2} , \xi \rangle}z_{n-2},e^{ 2\pi i \langle \alpha_{n-1} , \xi \rangle} z_{n-1},
    e^{ 2\pi i \langle -\alpha_{n-1} , \xi \rangle} z_{n} ) \text{ for
      } \xi \in \g,
  \end{split}
\end{equation}
with moment map $\Phi_0: \C^n \to \g^*$ given by
\begin{equation}
  \label{eq:41}
  \Phi_0(z_1,\ldots, z_n) = \pi \left(\sum\limits_{j=1}^{n-2}
    \alpha_j |z_j|^2 + \alpha_{n-1}(|z_{n-1}|^2-|z_n|^2) \right).
\end{equation} 
From \eqref{eq:55bis} it is clear that $0 \in \C^n$ is a fixed
point, that the circle $H$ acts trivially on 
$\C^{n-2} =
\C \langle z_1,\ldots, z_{n-2}\rangle $, and that the $(n-2)$-dimensional
torus $T'$ acts trivially on $\C^2=\C \langle z_{n-1}, z_n\rangle$. Therefore
the linear $T$-action on
$\C^n$ of \eqref{eq:55bis} splits as the product of a toric $T'$-action on $\C^{n-2} =
\C \langle z_1,\ldots, z_{n-2}\rangle $, and a complexity one $H$-action on
$\C^2 = \C \langle z_{n-1}, z_n\rangle$. Moreover,
\begin{enumerate}[label=(\arabic*),ref=(\arabic*), leftmargin=*]
\item \label{item:19}
  the stabilizer in $T$ of a point $q := (z_1,\ldots,z_n) \in \C^n$ is the
  product of the stabilizer  in $T'$ of $q_1:= (z_1,\ldots, z_{n-2}) \in \C^{n-2}$
  and of the stabilizer  in $H$ of $q_2:=(z_{n-1},z_n) \in \C^2$,
\item \label{item:25} the symplectic slice representation of
  $q \in \C^n$ for the action of $T$ splits as the product of the
  symplectic slice representations of $q_1\in \C^{n-2}$ for the action of $T'$
  and of $q_2 \in \C^2$ for the action of $H$, and 
\item \label{item:18} a point $q\in \C^n$ is exceptional with
  respect to the action of $T$ if and only if at least
  one of $q_1 \in \C^{n-2}$ and $q_2\in
  \C^2$ is exceptional with respect to the corresponding actions of $T'$ and $H$.
\end{enumerate}
We observe that property \ref{item:18} follows from properties
\ref{item:19} and \ref{item:25}. Hence, in order to understand properties of the product, we consider
each factor separately. This is the content of the following two
results.

\begin{lemma}\label{claim:toric}
Consider $\C^{n-2}$ with the above linear toric $T'$ action. Then
  every
  point in $\C^{n-2}$ is regular for the action of $T'$. Moreover, for each $q_1= (z_1,\ldots, z_{n-2}) \in \C^{n-2}$, 
  the subset $ J:=\{ j \in \{1,\ldots,n-2\} \mid z_j\neq 0\}$ is the unique subset 
  such that
  \begin{itemize}[leftmargin=*]
  \item the moment map image $\pi \sum\limits_{j=1}^{n-2}
    \alpha_j |z_j|^2 \in (\g')^*$ lies in $\R_{> 0} \langle \{\alpha_j
    \mid j \in J \} \rangle$ (if $J = \emptyset$, then $q_1 = 0$ and
    the moment map image equals zero),
  \item the stabilizer of $q_1$ is 
  \begin{equation}\label{def KJ}
    K_J:= \exp \left( \{\xi \in \g'\mid \langle \alpha_j , \xi
        \rangle = 0 \text{ for all }j\in J\} \right),
  \end{equation}
    and
  \item the isotropy weights of $q_1$ are $\{\alpha_j
    \mid j \notin J\}\subset (\g')^*$, where we identify $\mathrm{Lie}(K_J)^*$ with $\R\langle \{\alpha_j\mid j\notin J \cup \{n-1\}\} \rangle\subseteq (\g')^*\subset \g^*$.
  \end{itemize}   
  Conversely, given any subset $J \subseteq
  \{1,\ldots,n-2\}$ and any $w \in \R_{> 0} \langle \{\alpha_j
  \mid j \in J \} \rangle$, there exists $q_1 = (z_1,\ldots, z_{n-2}) \in \C^{n-2}$ such
  that $\pi \sum\limits_{j=1}^{n-2}
  \alpha_j |z_j|^2 =w$, the stabilizer of $q_1$ is $K_J$, and  the isotropy weights of $q_1$ are $\{\alpha_j
  \mid j \notin J\}$.

  Finally, the subset of points with trivial stabilizer
  is path-connected and dense.  
\end{lemma}

\begin{proof}
  By Lemma \ref{lemma:regular_connected_stab}, every point in a complexity zero Hamiltonian space is regular.
  The linear toric $T'$-action on $\C^{n-2}$ is given explicitly
  by
  \begin{equation}
    \label{eq:55}
    \exp(\xi')\cdot (z_1,\ldots,z_{n-2})=(e^{ 2\pi i \langle \alpha_1 , \xi'
      \rangle}z_1,\ldots,
    e^{ 2\pi i \langle \alpha_{n-2} , \xi' \rangle}z_{n-2})\quad\text{for}\;\;\xi'\in \g'\, , 
  \end{equation}
  (cf. \eqref{action Cn}). By definition of $J$, the moment map
  image $\pi \sum\limits_{j=1}^{n-2}
  \alpha_j |z_j|^2 $ lies in $\R_{> 0} \langle \{\alpha_j
  \mid j \in J \} \rangle$.  Since $\alpha_1,\ldots,\alpha_{n-2}$ are linearly independent, $J$ is the only such subset of $\{1,\ldots,n-2\}$.
  This proves the first bullet point.
  
  Next we prove the second bullet point. Let $K$ be the stabilizer
  of $q_1$. By
  \eqref{eq:55}, if $\xi' \in \g'$ then $\exp(\xi') \in K$ if and only if $\langle
  \alpha_j,\xi'\rangle \in \Z$ for all $j \in J$. 
  However, since each $\alpha_j$ is primitive,
  $$K=\exp\left( \{ \xi'\in \g'\mid \langle \alpha_j, \xi'\rangle \in \Z \text{ for all }j\in J\} \right)=\exp\left( \{ \xi'\in \g'\mid \langle \alpha_j, \xi'\rangle =0\text{ for all }j\in J\} \right),$$
  which, by definition, is exactly $K_J$.
  
  We turn to the proof of the third bullet point. The symplectic slice
  representation of $q_1$ is the
  following representation of $K_J$: We set
  $$ \C^{J}: =\{(w_1,\ldots, w_{n-2}) \in \C^{n-2} \mid w_j = 0 \text{
    for all } j \in J\}. $$
  This is a $T'$-invariant complex subspace of $\C^{n-2}$ that can be
  identified symplectically with the symplectic normal to the
  $T'$-orbit of $q_1$, once the tangent space at $q_1$ is
  identified with $\C^{n-2}$. Under this identification, since the
  $T'$-action on $\C^{n-2}$ is linear, the $K_J$-action on $\C^{J}$ is
  given by the restriction of the $T'$-action to $K_J$. The
  isotropy weights of this $K_J$-action are given by the set
  $\{\alpha_j \mid j \notin J\}\subset (\g')^*$.
  
  Conversely, given a subset $J \subseteq \{1,\ldots,n-2\}$ and $w \in \R_{> 0} \langle \{\alpha_j
  \mid j \in J \} \rangle$,  there exist positive constants
  $\lambda_j$ for $j \in J$ such that $w = \sum\limits_{j \in J}
  \lambda_j \alpha_j$. The point $q_1 = (z_1,\ldots, z_{n-2}) \in
  \C^{n-2}$ with coordinates given by 
  \begin{equation*}
    z_j =
    \begin{cases}
      \pi^{-1}\sqrt{\lambda_j} & \text{if } j \in J \\
      0 & \text{if } j \notin J
    \end{cases}
  \end{equation*}
  is such that $\pi \sum\limits_{j=1}^{n-2}
  \alpha_j |z_j|^2 =w$, its stabilizer is $K_J$, and its isotropy weights are $\{\alpha_j 
  \mid j \notin J\}$.
  
  Finally, $q_1 = (z_1,\ldots, z_{n-2}) \in \C^{n-2}$ has trivial stabilizer
  if and only if $z_j \neq 0$ for all $j=1,\ldots, n-2$. The subset
  $$ \{(z_1,\ldots, z_{n-2}) \in \C^{n-2} \mid z_j \neq 0 \text{ for
    all } j =1,\ldots, n-2 \} $$
  is clearly path-connected and dense. 
\end{proof}

\begin{lemma}\label{claim:pm_one}
  Consider $\C^2$ with the above linear complexity one $H$-action. A point $ q_2 \in \C^2 $ is exceptional
  if and only if $q_2 = (0,0)$. Moreover, $q_2 = (0,0)$ is the only point
  stabilized by $H$. In this case, the isotropy weights of $q_2$ are
  $\{\pm \, \alpha_{n-1}\}$. 
\end{lemma}
\begin{proof}
  We fix an isomorphism  between $H$ and $S^1$ so that
  $\alpha_{n-1}$ corresponds to $+1$. Under this isomorphism, the above linear $H$-action on
  $\C^2$ can be identified with the linear $S^1$-action on $\C^2$
  with weights equal to $+1$ and $-1$. The result then follows from a
  simple computation 
  and from Lemma \ref{lemma:regular_connected_stab}.
\end{proof}

Theorem \ref{local normal form} and Lemmas \ref{claim:toric},
\ref{claim:pm_one} imply the following result that is central to this section.
    
\begin{lemma}\label{lemma:local_description_exceptional}
  Let $\comp$ be a complexity one $T$-space of dimension $2n$ and let $p \in M^T$
  be isolated with isotropy weights
  $\alpha_1,\ldots,\alpha_{n-2},\alpha_{n-1},-\alpha_{n-1}$. There
  exists an open neighborhood $U$ of $p$ such that the following are
  equivalent:
  \begin{itemize}[leftmargin=*]
  \item $q \in U$ is exceptional, and
  \item $\Phi(q) \in \Phi(p) + \R_{\geq 0}\langle \{ \alpha_j \mid j
    = 1,\ldots, n-2\}\rangle$ and the stabilizer of $q$ contains $H$.
  \end{itemize}
  Moreover, given $q \in U$ exceptional, if $J \subseteq
  \{1,\ldots, n-2\}$ is defined by
  \begin{equation}
    \label{eq:22}
    \Phi(q) \in \Phi(p) + \R_{> 0} \langle \{\alpha_j \mid j \in J \}\rangle,
  \end{equation}
  then
  \begin{enumerate}[label=(\roman*),ref=(\roman*),leftmargin=*]
  \item \label{item:35}  the stabilizer of $q$ is $K_J \times H$, where
    $K_J$ is defined in \eqref{def KJ}, and
  \item \label{item:36}  the isotropy weights of $q$ are
    $$ \{\alpha_j \mid j \notin J\} \cup \{\pm \, 
    \alpha_{n-1}\}, $$
    \noindent
     where we identify
    $\text{Lie}(K_J)^*\subseteq (\g')^* \subset \g^*$ with $\R \langle \{ \alpha_j
    \mid j\notin J \cup \{n-1\} \}\rangle$. 
  \end{enumerate}
  Conversely, given any subset $J \subseteq \{1,\ldots, n-2\}$ and any
  $w \in \Phi(p) + \R_{>0} \langle \alpha_j \mid j \in J \}\rangle$,
  there exists an exceptional point $q \in U$ such that $\Phi(q) = w$,
  the stabilizer of $q$ is $K_J \times H$, and the isotropy weights of
  $q$ are $ \{\alpha_j \mid j \notin J\} \cup \{\pm \, 
  \alpha_{n-1}\} $.
  
  Finally, the subset
  $$ \{q \in U \mid q \text{ is exceptional and has
    stabilizer } H \} $$
  is path-connected and dense in $\{q \in U \mid q \text{ is exceptional}\}$.
\end{lemma}

\begin{proof}
  By Theorem \ref{local normal form}, it suffices to consider the
  local model determined by the isotropy weights, $p = 0 \in
  \C^n$ and $\Phi(p) = 0$. By property \ref{item:18}, and Lemmas \ref{claim:toric} and
  \ref{claim:pm_one}, a point $q = (q_1,q_2) \in \C^n$ is exceptional if and only if $q_2 = (0,0)$, which is also equivalent to the
  stabilizer of $q_2$ being $H$. Suppose that $q = (q_1,q_2)$ is exceptional and let $J \subseteq \{1,\ldots,
  n-2\}$ be the subset given by Lemma \ref{claim:toric}. Since $q_2 =
  (0,0)$, by \eqref{eq:41}, $\Phi(q)$ lies in $ \Phi(p) +
  \R_{> 0} \langle \{\alpha_j \mid j \in J \} \rangle$ if and only if $\pi \sum\limits_{j=1}^{n-2}
  \alpha_j |z_j|^2 \in (\g')^*$ lies in $\R_{> 0} \langle \{\alpha_j
  \mid j \in J \} \rangle$. 
  Hence, $J$ is the unique subset of $\{1,\ldots,
  n-2\}$ such that \eqref{eq:22}
  holds. Properties \ref{item:35} and \ref{item:36} in the statement
  follow immediately from \ref{item:19} and \ref{item:25} in the discussion preceding Lemma \ref{claim:toric}, and from
  Lemmas \ref{claim:toric} and \ref{claim:pm_one}. 
  
  Conversely, let $J \subseteq \{1,\ldots,
  n-2\}$ be a subset and $w \in \R_{> 0} \langle \{\alpha_j \mid j \in J \}
  \rangle$. Let $q_1 \in \C^{n-2}$ be the point given by Lemma
  \ref{claim:toric}. By property \ref{item:18} in the discussion preceding Lemma \ref{claim:toric} and Lemma
  \ref{claim:pm_one}, the point $q = (q_1,0,0) \in V$ is exceptional. Moreover, by
  \eqref{eq:41},
  $\Phi(q) = w$. By Lemmas \ref{claim:toric} and \ref{claim:pm_one}, and by
  properties \ref{item:19} and \ref{item:25}, the stabilizer of $q$ is
  $K_J \times H$ and the isotropy weights of $q$
  are $ \{\alpha_j \mid j \notin J\} \cup \{\pm \, 
  \alpha_{n-1}\}$, as desired.
  
  Finally, by Lemmas \ref{claim:toric} and \ref{claim:pm_one},
  \begin{equation}
    \label{eq:63}
    \{q =(q_1,q_2) \in \C^{n-2} \times \C^2 \mid q \text{ is
      exceptional and has
      stabilizer } H \} 
  \end{equation}
  equals
  \begin{equation}
    \label{eq:62}
    \{q = (q_1,0,0) \in \C^{n-2} \times \C^2 \mid q_1 \text{ has
      trivial stabilizer}\}. 
  \end{equation}
  By Lemma \ref{claim:toric}, $\{q_1 \in \C^{n-2} \mid q_1 \text{ has
    trivial stabilizer}\}$ is path-connected and dense in
  $\C^{n-2}$, thus completing the proof.
\end{proof}

The subset $J$ associated to an exceptional point near the isolated
fixed point of Lemma \ref{lemma:local_description_exceptional} has the
following useful property.

\begin{corollary}\label{cor:equal_symp_slice_local}
    Let $\comp$ be a complexity one $T$-space of dimension $2n$ and let $p \in M^T$
    be isolated with isotropy weights
    $\alpha_1,\ldots,\alpha_{n-2},\alpha_{n-1},-\alpha_{n-1}$. Let $U$ be the open neighborhood of $p$ given by
    Lemma \ref{lemma:local_description_exceptional}. Given exceptional points $q, q'
    \in U$, let $J, J' \subseteq
    \{1,\ldots, n-2\}$ be the subsets corresponding to $q, q'$ as in
    Lemma
    \ref{lemma:local_description_exceptional}. The symplectic slice
    representations of $q$ and $q'$ are isomorphic if and only if $J = J'$.
\end{corollary}

\begin{proof}
  If $J = J'$, then by parts \ref{item:35} and \ref{item:36}
  of Lemma \ref{lemma:local_description_exceptional}, the points $q$ and
  $q'$ have equal stabilizers and the same isotropy weights. Since
  their common stabilizer is connected, it follows that they have
  isomorphic symplectic slice representations. Conversely, suppose that $q$
  and $q'$ have isomorphic symplectic slice representations. Hence, by
  Lemma \ref{lemma:local_description_exceptional}, they have connected
  stabilizers, so that $K_J = K_{J'}$. 
  Since the dual of the Lie algebra of $K_J$ can be identified with
  $\R \langle \{\alpha_j \mid j\notin J \cup \{n-1\}\}\rangle$, and since $\alpha_1,\ldots, \alpha_{n-2}$ are linearly
  independent, it follows that $J = J'$.  
\end{proof}

Throughout this section, we apply Lemma
\ref{lemma:local_description_exceptional} and Corollary
\ref{cor:equal_symp_slice_local} to an isolated fixed point in a 
normalized monotone tall complexity one $T$-space. In this case, $H =
H_{\mathcal{F}_{\min}}$, the stabilizer of
$(M_{\mathcal{F}_{\min}},\omega_{\mathcal{F}_{\min}},\Phi_{\mathcal{F}_{\min}})$,
see Definition \ref{defn:minimal_facet}. Intuitively speaking, the next result is the `global version' of Lemma
\ref{lemma:local_description_exceptional} for normalized monotone tall
complexity one $T$-spaces.

\begin{lemma}\label{lemma:global_desc_exceptional}
  Let $\comp$ be a normalized monotone
  tall complexity one $T$-space of dimension $2n$. Let $q \in M$
  be exceptional. There exist $p \in M^T$ isolated
  and a unique subset $J \subseteq
  \{1,\ldots, n-2\}$ such that, if
  $\alpha_1,\ldots,\alpha_{n-2},\alpha_{n-1},-\alpha_{n-1}$ are the
  isotropy weights of $p$ as in Proposition
  \ref{prop:iso_fixed_pts_monotone}, then
  \begin{enumerate}[label=(\roman*), ref=(\roman*),leftmargin=*]
  \item \label{item:37} the moment map image $\Phi(q)$ lies in $\Phi(p)
    + \R_{>0} \langle \{ \alpha_j \mid j \in J\} \rangle$,
  \item \label{item:38} the stabilizer of $q$ is $K_J \times
    H_{\mathcal{F}_{\min}}$, where $K_J \leq T'$ is as in
    \eqref{def KJ}, and
  \item \label{item:39} 
    the isotropy weights of $q$ are $\{\alpha_j \mid j \notin J\} \cup \{\pm \, 
    \alpha_{n-1}\}$ (see part \ref{item:36} of Lemma \ref{lemma:local_description_exceptional}). 
  \end{enumerate}
\end{lemma}

\begin{proof}
  By Lemma
  \ref{lemma:exceptional}, the sheet $(N,\omega_N,\Phi_N)$ through $q$ is exceptional. Since
  $N$ is compact, it contains a fixed point $p \in M^T$ that is exceptional and therefore isolated by
  Lemma \ref{lemma:fixed_point_exceptional}. Since $N$ is connected,
  by the principal orbit
  theorem (see \cite[Theorem
  2.8.5]{dk}), 
  there exists a relatively open, dense and connected
  subset $N'$ of $N$ such that, if $q' \in
  N'$, then $q$ and $q'$ have isomorphic symplectic
  slice representations. In particular, if $U$ is the open neighborhood of
  $p$ given by Lemma \ref{lemma:local_description_exceptional}, then $U \cap N'$ is not empty; moreover, for all $q' \in U \cap N'$, the
  symplectic slice representation of $q'$ is isomorphic to that of $q$. By
  Lemma \ref{lemma:local_description_exceptional} and 
  Corollary \ref{cor:equal_symp_slice_local}, there exists a unique
  subset $J \subseteq \{1,\ldots, n-2\}$ such that, for all $q' \in U
  \cap N'$,
  \begin{itemize}[leftmargin=*]
  \item the moment map image $\Phi(q')$ lies in $\Phi(p)
    + \R_{>0} \langle \{ \alpha_j \mid j \in J\} \rangle$,
  \item the stabilizer of $q'$ is $K_J \times
    H_{\mathcal{F}_{\min}}$, where $K_J \leq T'$ is as in
    \eqref{def KJ}, and
  \item the isotropy weights of $q'$ are $\{\alpha_j \mid j \notin J\} \cup \{\pm \, 
    \alpha_{n-1}\}$.
  \end{itemize}
  Since $U \cap N \neq \emptyset$, the second and third
  bullet points imply properties
  \ref{item:38} and \ref{item:39}. To see that
  property \ref{item:37} holds, we observe that, by the first bullet point, $\Phi(U \cap N')$ is contained in $\Phi(p)
  + \R_{>0} \langle \{ \alpha_j \mid j \in J\} \rangle$. Since
  $N_{\mathrm{reg}}$ is dense in $N$, we have that $\Phi(U \cap N)$
  is contained in  $\Phi(p)
  + \R_{\geq 0} \langle \{ \alpha_j \mid j \in J\} \rangle$. On the
  other hand, since the stabilizer of $q$ is $K_J \times
  H_{\mathcal{F}_{\min}}$, the sheet
  $(N,\omega_N,\Phi_N)$ is a compact Hamiltonian $T''$-space, where
  $T'' = T/(K_J \times
  H_{\mathcal{F}_{\min}}) \simeq T' /K_J$. By construction, we may identify the dual of the Lie
  algebra of $T''$ with $\Phi(p) + \R\langle \{\alpha_j \mid j \in J\}
  \rangle$. Hence, by the Convexity Package (Theorem
  \ref{thm:con_pac}), the moment map image $\Phi_N(N) = \Phi(N)$ is a
  convex polytope in $\Phi(p)
  + \R \langle \{ \alpha_j \mid j \in J\} \rangle$. Since $\Phi(p)
  + \R_{\geq 0} \langle \{ \alpha_j \mid j \in J\} \rangle$ is convex
  in $\Phi(p)
  + \R \langle \{ \alpha_j \mid j \in J\} \rangle$ and since $\Phi(U \cap N)$
  is contained in  $\Phi(p)
  + \R_{\geq 0} \langle \{ \alpha_j \mid j \in J\} \rangle$, $\Phi(N)$ is contained in $\Phi(p)
  + \R_{\geq 0} \langle \{ \alpha_j \mid j \in J\} \rangle$. In
  particular, the interior of $\Phi(N)$ is contained in $\Phi(p)
  + \R_{>0} \langle \{ \alpha_j \mid j \in J\} \rangle$. Since $q$ is
  a regular point for the moment map $\Phi_N$, the moment map image
  $\Phi_N(q) = \Phi(q)$ lies in the (relative) interior of $\Phi(N)$,
  as desired.
\end{proof}

\begin{remark}\label{rmk:image_symplectic_slice}
  By Lemma \ref{lemma:global_desc_exceptional}, if $q$ is exceptional,
  then the moment map $\Phi(q)$ can be used to reconstruct the symplectic slice
  representation of $q$. To see this, we observe that, by property
  \ref{item:37} and Proposition \ref{prop:iso_fixed_pts_monotone}, $\Phi(q)$ lies in the affine hyperplane $\Phi(p)+\{w \in \g^*
  \mid \langle w, \nu_{\min} \rangle = 0\}$. 
  We recall that a basis
  for the linear subspace $\{w \in \g^*
  \mid \langle w, \nu_{\min} \rangle = 0\}$ is given by $\alpha_1,\ldots, \alpha_{n-2}$
  (cf. \eqref{eq:59}). Hence, by Lemma
  \ref{lemma:global_desc_exceptional}, $\Phi(q)$ determines the subset $J$
  uniquely, and $J$ determines the stabilizer and the isotropy weights of
  $q$. Since the stabilizer of $q$ is connected, the claim follows.
\end{remark}

Our next aim is to prove Proposition
\ref{lemma:exceptional_sheet_dimension_1}, which plays an
important role in several key results below (e.g., Theorems
\ref{cor:exceptional_orbits} and \ref{thm:possibilities_DH}. We start
with the following result.

\begin{lemma}\label{lemma:iso_fixed_all_edges}
  Let $\comp$ be a normalized monotone tall complexity one
  $T$-space. The following are equivalent:
  \begin{enumerate}[label=(\arabic*),ref=(\arabic*),leftmargin=*]
  \item \label{item:30} there exists an isolated fixed point, and
  \item \label{item:31} for each edge $e$ that comes out of $\mathcal{F}_{\min}$,
    there exists an isolated fixed point $p \in M^T$ such that
    $\Phi(p) \in e$.  
  \end{enumerate}
\end{lemma}
\begin{proof}
  Clearly \ref{item:31} implies \ref{item:30}. Conversely, suppose
  that there exists an isolated fixed point $p \in M^T$. If $\dim M =
  4$, there is nothing to prove, so we may assume that $\dim M \geq
  6$. By
  Proposition \ref{prop:iso_fixed_pts_monotone}, there exists an edge
  $e$ that comes out of $\mathcal{F}_{\min}$ such that $\Phi(p) \in
  e$. Let $v \in \mathcal{F}_{\min}$ be the vertex to which $e$ is
  incident and let $\alpha_1,\ldots, \alpha_{n-2},\alpha_{n-1}$ be the
  non-zero isotropy weights of $\Phi^{-1}(v)$ ordered so that
  \eqref{eq:59} holds.

  For $j=1,\ldots, n-2$, let $v_j \in
  \mathcal{F}_{\min}$ be the vertex that lies on the edge supported on
  $v + \R_{\geq 0}\langle \alpha_j \rangle$ and is not $v$. Let $e_j$
  be the edge that comes out of $\mathcal{F}_{\min}$ that is incident
  to $v_j$. We claim that there exists an isolated fixed point $p_j$
  such that $\Phi(p_j) \in e_j$ (see Figure \ref{Pic1ImageMexc}). To this end, by Lemma
  \ref{lemma:local_description_exceptional}, there exists an
  exceptional point $q \in M$
  arbitrarily close to $p$ such that $\Phi(q) \in \Phi(p) + \R_{>0} \langle
  \alpha_j \rangle$; moreover, the stabilizer of $q$ has codimension
  1. Let $(N,\omega_N,\Phi_N)$ be the sheet through $q$. Since $q$ is
  exceptional, so is $(N,\omega_N,\Phi_N)$; furthermore, $p \in N$ by construction. Hence,
  $(N,\omega_N,\Phi_N)$ is a compact symplectic toric manifold with
  moment map image contained in $\Phi(p) + \R_{\geq 0} \langle
  \alpha_j \rangle$. Let $p_j \in M^T \cap N$ be the unique fixed
  point such that $\Phi(p_j) \in \Phi(p) + \R_{>0} \langle \alpha_j
  \rangle$. Since $(N,\omega_N,\Phi_N)$ is exceptional, so is
  $p_j$; moreover, by Lemma \ref{lemma:fixed_point_exceptional}, $p_j$
  is isolated. Hence, by Proposition
  \ref{prop:iso_fixed_pts_monotone}, the image $\Phi(p_j)$ lies on an
  edge that comes out of $\mathcal{F}_{\min}$. This edge is
  necessarily $e_j$: To see this, we observe that the moment map image
  $\Phi(N)$ is contained in the affine two-dimensional plane $v + \R
  \langle \alpha_j,\alpha_{n-1}\rangle$. This plane supports a
  two-dimensional face of $\Phi(M)$ that contains $e$ and the edge
  of $\mathcal{F}_{\min}$ that is incident to both $v$ and
  $v_j$. Hence, there exists only one other edge that is incident to
  $v_j$ that is contained in this affine plane. Since this plane
  intersects $\mathcal{F}_{\min}$ precisely in the edge that is
  incident to both $v$ and $v_j$, by \eqref{eq:59} applied to the
  non-zero isotropy weights of $\Phi^{-1}(v_j)$, the other edge that
  is incident to $v_j$ and contained in the above affine plane must
  come out of $\mathcal{F}_{\min}$, i.e., it must be $e_j$. 
 
  \begin{figure}[htbp]
  	\begin{center}
  		\includegraphics[width=7cm]{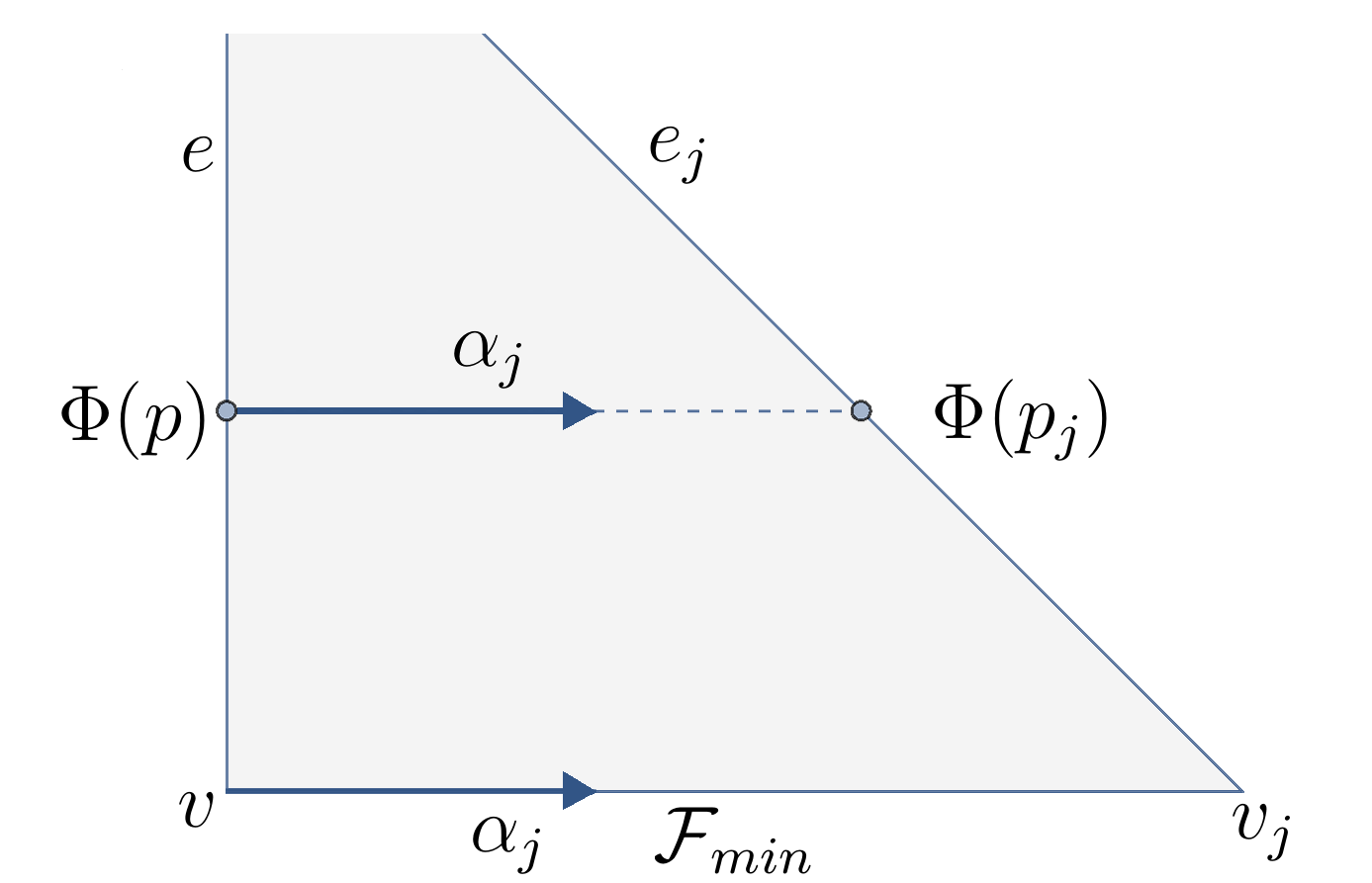}
  		\caption{A schematic illustration of the argument in
                  the proof of Lemma \ref{lemma:iso_fixed_all_edges}.}
  		\label{Pic1ImageMexc}
  	\end{center}
  \end{figure}

  By the last paragraph, \ref{item:31} holds for each edge that comes out of
  $\mathcal{F}_{\min}$ that is incident to a vertex of
  $\mathcal{F}_{\min}$ that is adjacent to $v$ in $\mathcal{F}_{\min}$
  (i.e., there exists an edge of $\mathcal{F}_{\min}$ that is incident
  to both vertices). We define the following relation on the set of
  vertices of $\mathcal{F}_{\min}$:
  $$ v_1 \sim v_2 \Leftrightarrow \text{ either } v_1=v_2 \text{ or
  } v_1 \text{ is adjacent to } v_2. $$
  Since the transitive closure of the above
  relation has one equivalence class and since there is a one-to-one
  correspondence between edges that come out of $\mathcal{F}_{\min}$
  and vertices of $\mathcal{F}_{\min}$, \ref{item:31} holds.
\end{proof}

As a consequence of Lemma \ref{lemma:iso_fixed_all_edges}, 
we obtain the following sufficient
condition for a normalized monotone tall complexity one
$T$-space to be without isolated fixed points.

\begin{corollary}\label{cor:no_iso_fixed_points}
  Let $\comp$ be a normalized monotone tall complexity one $T$-space. If
  there is a vertex of $\Phi(M)$ on the linear hyperplane $\{w
  \in \g^* \mid \langle w, \nu_{\min} \rangle =
  0\}$, then there are no isolated fixed points. Moreover,
  $M_{\mathrm{exc}} = \emptyset$.
\end{corollary}

\begin{proof}
  Let $v \in \Phi(M)$ be a vertex of $\Phi(M)$ such that $\langle v, \nu_{\min} \rangle =
  0$. First we show that there is an edge $e$ that comes out of $\mathcal{F}_{\min}$ that is
  incident to $v$. Let $\alpha_1,\ldots, \alpha_{n-1}$ be
  the non-zero isotropy weights of $\Phi^{-1}(v)$. By
  Lemma \ref{lemma:special_weight}, we may assume that $\langle \alpha_{n-1}, \nu_{\min} \rangle <
  0$. Let $e$ be the edge that is contained in $v + \R_{\geq
    0} \langle \alpha_{n-1} \rangle$ and let $v' \in \Phi(M)$ be the
  other vertex to which $e$ is incident. By construction, $\langle v', \nu_{\min} \rangle <
  0$. Moreover, since $\comp$ is normalized monotone, $\Phi(M)$ is
  integral. Therefore, by \eqref{eq:10}, $\langle v',
  \nu_{\min} \rangle = -1$, i.e., $v'$ is a vertex of
  $\mathcal{F}_{\min}$. Since $\langle \alpha_{n-1}, \nu_{\min} \rangle <
  0$, $e$ is an edge of $\Phi(M)$ that comes out of
  $\mathcal{F}_{\min}$. Hence, $v$ is the only element in the
  intersection of $e$ and $\{w
  \in \g^* \mid \langle w, \nu_{\min} \rangle =
  0\}$. By Theorem \ref{thm:con_pac} and Proposition
  \ref{prop:iso_fixed_pts_monotone}, there is no isolated fixed point
  that is mapped to $e$ under $\Phi$. Hence, by Lemmas
  \ref{lemma:fixed_point_exceptional}, 
  \ref{lemma:necessary_sufficient_existence_exceptional} and
  \ref{lemma:iso_fixed_all_edges}, the result follows.
\end{proof}

The next result plays a key role throughout the paper.

\begin{proposition}\label{lemma:exceptional_sheet_dimension_1}
  Let $\comp$ be a normalized monotone tall complexity one $T$-space of dimension $2n$. If
  $(N,\omega_N,\Phi_N)$ is an exceptional sheet that is stabilized by a
  one-dimensional subgroup $H$, then $H = H_{\mathcal{F}_{\min}}$ and
  \begin{equation}
    \label{eq:53}
    \Phi(N) = \Phi(M) \cap \{w \in \g^* \mid \langle w, \nu_{\min}
    \rangle = 0\}. 
  \end{equation}
\end{proposition}

\begin{proof}
  Since $(N,\omega_N,\Phi_N)$ is exceptional, every point in $N$ is
  exceptional. Moreover, since $(N,\omega_N,\Phi_N)$ is stabilized by
  $H$, there exists a point $q \in N$ with stabilizer equal to
  $H$. Hence, by part \ref{item:38} of Lemma \ref{lemma:global_desc_exceptional}, there exist $p \in M^T$
  isolated and a unique subset $J \subseteq \{1,\ldots, n-2\}$ such 
  that, if $\alpha_1,\ldots,\alpha_{n-2},\alpha_{n-1},-\alpha_{n-1}$ are the
  isotropy weights of $p$ as in Proposition
  \ref{prop:iso_fixed_pts_monotone}, then the stabilizer of $q'$ is $K_J \times
  H_{\mathcal{F}_{\min}}$, where $K_J \leq T'$ is as in
  \eqref{def KJ}. By definition, $K_J$ is
  connected. Hence, if the dimension of the stabilizer of $q$ is one,
  then it must be $H_{\mathcal{F}_{\min}}$, thus proving the first
  statement.
  
  By Proposition \ref{prop:mom_map_image}, $\Phi(M)$ is a reflexive
  Delzant polytope and therefore the origin lies in the interior of
  $\Phi(M)$ (see Lemma
  \ref{lemma:reflexive_interior}). 
  Hence, the interior of $\Phi(M) \cap \{w \in \g^* \mid  \langle w, \nu_{\min}
  \rangle = 0\}$ in $\{w \in \g^* \mid  \langle w, \nu_{\min}
  \rangle = 0\}$ is non-empty. Since both
  $M$ and $N$ are compact, and since $\{w \in \g^* \mid \langle w, \nu_{\min}
  \rangle = 0\}$ is a linear hyperplane in $\g^*$, by the
  Convexity Package (Theorem \ref{thm:con_pac}), both $\Phi(M) \cap \{w \in \g^* \mid \langle w, \nu_{\min}
  \rangle = 0\}$ and $\Phi_N(N) = \Phi(N)$ are convex
  polytopes. Therefore, in order to prove that \eqref{eq:53} holds, it
  suffices to show that $\Phi(M) \cap \{w \in \g^* \mid \langle w, \nu_{\min}
  \rangle = 0\}$ and $\Phi(N)$ have the same vertices. 

  Since $M_{\mathrm{exc}} \neq \emptyset$, by Corollary
  \ref{cor:no_iso_fixed_points}, there is no vertex of $\Phi(M)$ lying
  on $\{w \in \g^* \mid \langle w, \nu_{\min}
  \rangle = 0\}$. Hence, a point $\hat{v} \in \Phi(M) \cap \{w \in \g^* \mid \langle w, \nu_{\min}
  \rangle = 0\}$ is a vertex if and only if there exists an edge $e$
  of $\Phi(M)$
  that comes out of $\mathcal{F}_{\min}$ such that $\hat{v}$ is the
  intersection of $e$ with $ \{w \in \g^* \mid \langle w, \nu_{\min}
  \rangle = 0\}$.  On the other hand, since $(N,\omega_N,\Phi_N)$ is exceptional and since the
  complexity of $\comp$ is one, by Proposition
  \ref{prop:regular_exceptional}, the complexity of
  $(N,\omega_N,\Phi_N)$ is zero, i.e., it is a compact symplectic
  toric manifold. Therefore, $\hat{v} \in \Phi(N)$ is a vertex if and only
  if there exists an isolated fixed point $p \in N$ such that
  $\Phi(p) = \hat{v}$.

  Let $\hat{v} \in \Phi(N)$ be a vertex and let $p \in N$ be as
  above. By Proposition \ref{prop:iso_fixed_pts_monotone}, there
  exists $\hat{v}$ an edge $e$ of $\Phi(M)$
  that comes out of $\mathcal{F}_{\min}$ such that $\hat{v}$ is the
  intersection of $e$ with $ \{w \in \g^* \mid \langle w, \nu_{\min}
  \rangle = 0\}$. Hence, each vertex of $\Phi(N)$ is a vertex of $\Phi(M) \cap \{w \in \g^* \mid \langle w, \nu_{\min}
  \rangle = 0\}$. Moreover, by Lemma
  \ref{lemma:local_description_exceptional}, there exists an open
  neighborhood $U$ of $p$ such that $U \cap N$ is precisely the subset
  of exceptional points in $U$ and
  \begin{equation}
    \label{eq:51}
    \Phi(U \cap N) = \Phi(U) \cap \{w \in \g^* \mid \langle w, \nu_{\min}
    \rangle = 0\}.
  \end{equation}  
  Set $V:= \Phi(U)$. By the Convexity Package (Theorem
  \ref{thm:con_pac}), $V$ is an open neighborhood of
  $\hat{v}$. Moreover, by \eqref{eq:51},
  \begin{equation}
    \label{eq:61}
    V \cap \Phi(N) = V \cap \{w \in \g^* \mid \langle w, \nu_{\min}
    \rangle = 0\},
  \end{equation}
  (see Figure \ref{Pic2ImageMexc}).
  
  \begin{figure}[htbp]
  	\begin{center}
  		\includegraphics[width=12cm]{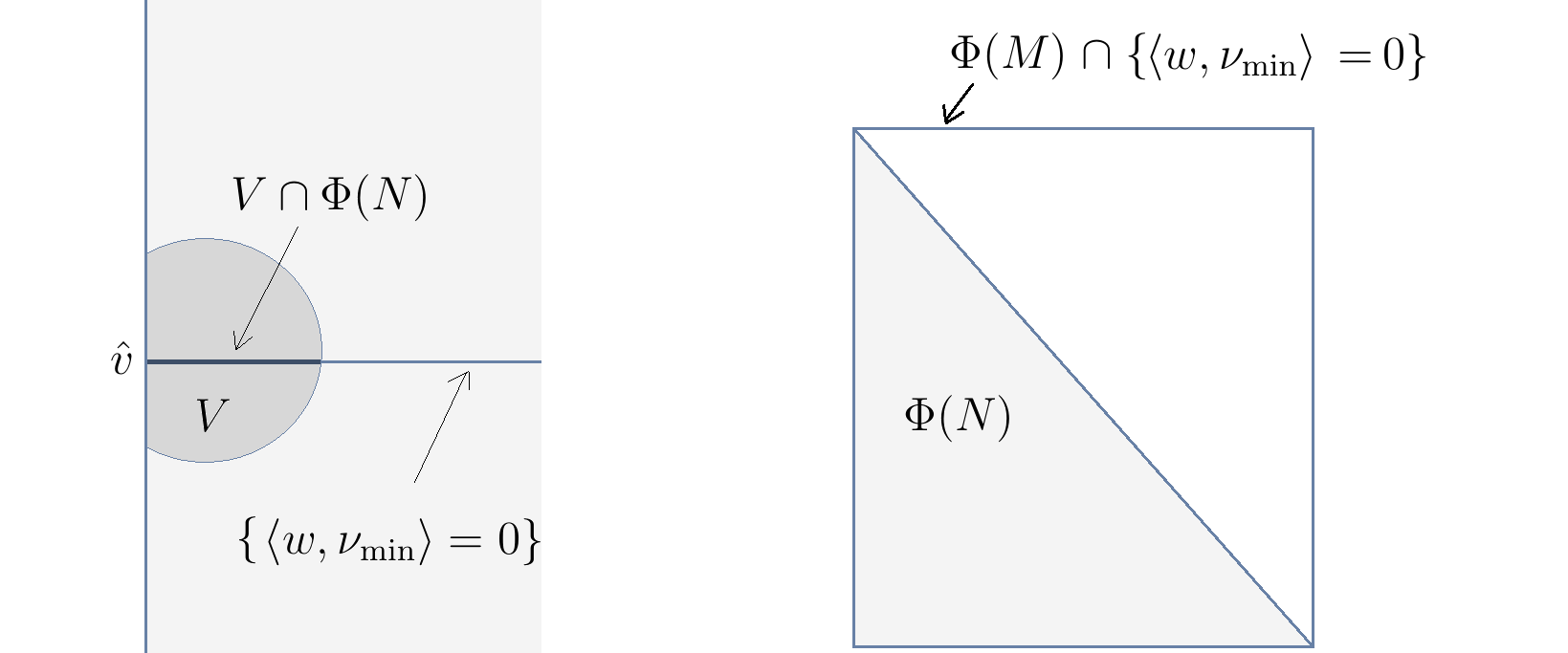}
  		\caption{The two possibilities in the proof of
                  Proposition
                  \ref{lemma:exceptional_sheet_dimension_1}. The
                  figure on the left illustrates the case in which
                  $\hat{v}$ is a vertex of $\Phi(N)$, while that on
                  the right shows the (absurd) case in which not all
                  vertices of $\Phi(M) \cap \{w \in \g^* \mid \langle w, \nu_{\min}
                  \rangle = 0\}$ are vertices of $\Phi(N)$.}
  		\label{Pic2ImageMexc}
  	\end{center}
  \end{figure}
  Suppose that there exists a vertex of $\Phi(M) \cap \{w \in \g^* \mid \langle w, \nu_{\min}
  \rangle = 0\}$ that is not a vertex of $\Phi(N)$. Since both $\Phi(M) \cap \{w \in \g^* \mid \langle w, \nu_{\min}
  \rangle = 0\}$ and $\Phi(N)$ are convex polytopes of full dimension in $\{w \in \g^* \mid \langle w, \nu_{\min}
  \rangle = 0\}$ and since the vertices of the latter are a subset of
  those of the former, there exists a vertex $\hat{v}$ of $\Phi(N)$
  such that for any open neighborhood $V$ of $\hat{v}$
  \begin{equation*}
    V \cap \Phi(N) \subsetneq V \cap \{w \in \g^* \mid \langle w, \nu_{\min}
    \rangle = 0\},
  \end{equation*}
  (see Figure \ref{Pic2ImageMexc}). By \eqref{eq:61}, this is a contradiction.
\end{proof}

Our penultimate aim in this section is to prove Theorem
\ref{cor:exceptional_orbits} below. To this end, first we prove the
following result.

\begin{lemma}\label{lemma:exceptional_in_max_sheet}
  Let $\comp$ be a normalized monotone tall complexity one $T$-space of dimension $2n$. If $q \in M$ is
  exceptional, then there exists a unique exceptional sheet
  $(N,\omega_N,\Phi_N)$ that is stabilized by $H_{\mathcal{F}_{\min}}$ such that $q \in N$. 
\end{lemma}

\begin{proof}
  First we prove uniqueness. Let $(N_i, \omega_i, \Phi_i)$ be an exceptional sheet that is stabilized by $H_{\mathcal{F}_{\min}}$ such that $q \in N_i$ for $i=1,2$. Hence, both $N_1$ and $N_2$ are connected components of
  $M^{H_{\mathcal{F}_{\min}}}$. Since $q \in N_1 \cap N_2$, it follows
  that $N_1 = N_1 \cup N_2 = N_2$, as desired.

  Next we prove existence. Let $(N',\omega',\Phi')$ be the sheet
  through $q$. Since $q$ is exceptional, by Lemma
  \ref{lemma:exceptional}, $(N',\omega',\Phi')$
  is exceptional. Since $N'$ is compact, there exists $p \in
  M^T \cap N'$ that, by Lemma \ref{lemma:fixed_point_exceptional}, is
  isolated. We claim that there exists an exceptional sheet
  $(N,\omega_N,\Phi_N)$ stabilized by $H_{\mathcal{F}_{\min}}$ such
  that $p \in N$. To this end, we use Proposition
  \ref{prop:iso_fixed_pts_monotone} and Lemma \ref{lemma:local_description_exceptional}.
  Let $U$ be the open neighborhood of $p$ given by Lemma
  \ref{lemma:local_description_exceptional} and let $U_1$ be the subset of $U$
  consisting of exceptional points stabilized by
  $H_{\mathcal{F}_{\min}}$, which is path-connected and dense in the subset
  of $U$ consisting of exceptional points. 
  In particular,
  $U_1 \neq \emptyset$. Given any point $q' \in U_1$, we consider the
  sheet $(N,\omega_N,\Phi_N)$ through $q'$. By Lemma
  \ref{lemma:exceptional}, $(N,\omega_N,\Phi_N)$ is
  exceptional, while, by definition, it is stabilized by
  $H_{\mathcal{F}_{\min}}$. Since $U_1$ is dense in the subset
  of $U$ consisting of exceptional points and since $p
  \in U$ is exceptional by Lemma \ref{lemma:fixed_point_exceptional}, it
  follows that $p \in N$, thus proving the claim.

  Hence, $N' \cap N \neq
  \emptyset$. Moreover, since any point in $N'$ is exceptional, by
  Lemma \ref{lemma:global_desc_exceptional} the stabilizer of any point
  in $N'$ contains $H_{\mathcal{F}_{\min}}$. Since both $N'$ and $N$
  are connected and since $N$ is a connected component of
  $M^{H_{\mathcal{F}_{\min}}}$, it follows that $N' \cup N = N$, so
  that $N'$ is contained in $N$. Hence, $q \in N$, as desired.
\end{proof}

In fact, the proof of Lemma \ref{lemma:exceptional_in_max_sheet}
yields a slightly stronger result, namely that, under the hypotheses
of the lemma, any exceptional sheet is contained in one that
stabilized by $H_{\mathcal{F}_{\min}}$.

\begin{theorem}\label{cor:exceptional_orbits}
  Let $\comp$ be a normalized monotone tall complexity one $T$-space of dimension $2n$.
  Each connected component of $M_{\mathrm{exc}}$ is mapped
  homeomorphically to $ \Phi(M) \cap
  \{w \in \g^* \mid \langle w, \nu_{\min} \rangle = 0 \}$
  by the orbital moment map. In particular, each connected component
  of $M_{\mathrm{exc}}$ is contractible. Moreover, 
  \begin{itemize}[leftmargin=*]
  \item if  $m$ is the
    number of vertices of the minimal facet, then the number of
    connected components of $M_{\mathrm{exc}}$ is precisely the number
    of isolated fixed points divided by $m$, and
  \item if $e$ is an
    edge of  $\Phi(M)$ that comes out of
    $\mathcal{F}_{\mathrm{min}}$, then the number of isolated fixed
    points lying on $\Phi^{-1}(e)$ equals the number of connected
    components of $M_{\mathrm{exc}}$.
  \end{itemize}
\end{theorem}

\begin{proof}
  First we show that the image of $\{p \in M^{H_{\mathcal{F}_{\min}}}
  \mid p \text{ is exceptional}\}$  under the quotient map $M \to M/T$
  equals $M_{\mathrm{exc}}$ and that number of connected components of
  the latter equals the number of exceptional sheets that are
  stabilized by $H_{\mathcal{F}_{\min}}$. Clearly, the image of
  $\{p \in M^{H_{\mathcal{F}_{\min}}}
  \mid p \text{ is exceptional}\}$ under the quotient map
  $M \to M/T$ is contained in $M_{\mathrm{exc}}$. Conversely, given an
  exceptional orbit $\mathcal{O} \in M_{\mathrm{exc}}$, every point in
  $\mathcal{O}$ is exceptional by Remarks \ref{rmk:symp_slice_orbit} and \ref{rmk:exc_orbits}.
  Fix a point $p \in \mathcal{O}$. By Proposition 
  \ref{lemma:exceptional_sheet_dimension_1} and Lemma \ref{lemma:exceptional_in_max_sheet}, $p \in
  M^{H_{\mathcal{F}_{\min}}}$. Since $M^{H_{\mathcal{F}_{\min}}}$ is
  $T$-invariant, it follows that $\mathcal{O}$ is contained in
  $M^{H_{\mathcal{F}_{\min}}}$. Hence, $M_{\mathrm{exc}}$ is contained
  in the image of
  $\{p \in M^{H_{\mathcal{F}_{\min}}}
  \mid p \text{ is exceptional}\}$ under the quotient map
  $M \to M/T$ and the first claim follows. To prove the second claim, we observe that the connected
  components of $\{p \in M^{H_{\mathcal{F}_{\min}}}
  \mid p \text{ is exceptional}\}$ are
  exactly the exceptional sheets that are
  stabilized by $H_{\mathcal{F}_{\min}}$. Since sheets are
  $T$-invariant and $T$ is connected, the restriction of the quotient map
  $M \to M/T$ to $\{p \in M^{H_{\mathcal{F}_{\min}}}
  \mid p \text{ is exceptional}\}$ induces a
  bijection between the connected components of
  $\{p \in M^{H_{\mathcal{F}_{\min}}}
  \mid p \text{ is exceptional}\}$ and those of
  $ M_{\mathrm{exc}}$.

  By Proposition \ref{lemma:exceptional_sheet_dimension_1}, the image
  of an exceptional sheet that is stabilized by
  $H_{\mathcal{F}_{\min}}$ is precisely $ \Phi(M) \cap
  \{w \in \g^* \mid \langle w, \nu_{\min} \rangle = 0 \}$. Moreover,
  any such sheet is a compact symplectic toric manifold, so that the
  corresponding orbital moment map image is a homeomorphism onto its
  image. Hence, the first claim follows.
  
  By the above argument, in order to prove the bulleted statements, it
  suffices to show that the number of exceptional sheets that are
  stabilized by $H_{\mathcal{F}_{\min}}$ is precisely the number of
  isolated fixed points divided by $m$. We begin by observing that
  $M_{\mathrm{exc}} = \emptyset$ if and only if there are no isolated
  fixed points. This is a consequence of Lemma
  \ref{lemma:necessary_sufficient_existence_exceptional} and the
  complexity of $\comp$ being one. So we may assume that $M_{\mathrm{exc}} \neq \emptyset$.
  Let $(N,\omega_N,\Phi_N)$ be such an exceptional
  sheet that is stabilized by $H_{\mathcal{F}_{\min}}$. Since $(N,\omega_N,\Phi_N)$ is a compact symplectic toric manifold, the number of vertices
  in $\Phi_N(N) = \Phi(N)$ is precisely the number of fixed points. By Proposition
  \ref{lemma:exceptional_sheet_dimension_1}, 
  $\Phi(N)$ equals $ \Phi(M) \cap
  \{w \in \g^* \mid \langle w, \nu_{\min} \rangle = 0 \}$; moreover,
  there is a bijection between the vertices of $\Phi(N)$ and those of
  $\mathcal{F}_{\min}$. Hence, the cardinality of $M^T \cap
  N$ equals $m$. 

  Moreover, if $(N',\omega',\Phi')$ is another sheet stabilized by
  $H_{\mathcal{F}_{\min}}$, then either $N = N'$ or $N \cap N' =
  \emptyset$. In particular, if $N \neq N'$, the subsets $M^T \cap
  N$ and $M^T \cap
  N'$ are disjoint. Finally, since the set of isolated fixed points is a
  subset of the set of exceptional points by Lemma
  \ref{lemma:fixed_point_exceptional}, by Lemma
  \ref{lemma:exceptional_in_max_sheet}, $M^T_{\mathrm{isolated}}$ equals
  the disjoint union over all exceptional sheets $(N,\omega_N,\Phi_N)$ stabilized
  by $H_{\mathcal{F}_{\min}}$ of the intersections
  $M_{\mathrm{isolated}}^T \cap N$. Since any such sheet is exceptional,
  by Lemma \ref{lemma:fixed_point_exceptional}, $M^T
  \cap N$ equals the intersection of $N$ with the set of isolated
  fixed points for any exceptional sheet $(N,\omega_N,\Phi_N)$ stabilized
  by $H_{\mathcal{F}_{\min}}$. Putting the above facts together, the
  bulleted statements follow.
\end{proof}

We conclude this section with the following important result.

\begin{theorem}\label{prop:trivial_painting}
  The equivalence class of paintings of a compact normalized monotone tall complexity one $T$-space $\comp$ is trivial.
\end{theorem}

\begin{proof}
  By Lemma \ref{lemma:chern}, the genus of $\comp$ is zero. Let $f :
  M_{\mathrm{exc}} \to S^2$ be a painting of $\comp$. We need to
  construct a painting $f' : M_{\mathrm{exc}} \to S^2$ that is
  constant on each connected component of $M_{\mathrm{exc}}$ and that
  is equivalent to $f$. If $M_{\mathrm{exc}} = \emptyset$, there is
  nothing to prove, so we may assume that $M_{\mathrm{exc}} \neq
  \emptyset$.

  Since $\Phi(M) \cap \{ w \in \g^* \mid \langle w, \nu_{\min} \rangle
  = 0\}$ is a convex polytope, it is contractible. We fix a point
  $w_0$ in $\Phi(M) \cap \{ w \in \g^* \mid \langle w, \nu_{\min} \rangle
  = 0\}$ (for instance, the origin). Since $\Phi(M) \cap \{ w \in \g^* \mid \langle w, \nu_{\min} \rangle
  = 0\}$ is contractible, there exists a deformation retraction of $\Phi(M) \cap \{ w \in \g^* \mid \langle w, \nu_{\min} \rangle
  = 0\}$ onto
  $w_0$ that we denote by $H_t$, where $H_0 = \mathrm{id}$ and $H_1(w)
  = w_0$ for all $w \in \Phi(M) \cap \{ w \in \g^* \mid \langle w, \nu_{\min} \rangle
  = 0\}$.

  Let $k >0$ be the number of connected components of
  $M_{\mathrm{exc}}$. We have that
  $$M_{\mathrm{exc}} = \coprod\limits_{i=1}^k M_i,$$
  where $M_i$ is a connected component of $M_{\mathrm{exc}}$ for
  $i=1,\ldots, k$. By Theorem \ref{cor:exceptional_orbits}, the
  restriction of $\overline{\Phi}\colon M/T \to \g^*$ to $M_i$ is a homeomorphism onto $\Phi(M) \cap \{ w \in \g^* \mid \langle w, \nu_{\min} \rangle
  = 0\}$; let $\overline{\Phi}^{-1}_i$ be the inverse to this homeomorphism. We denote the restriction of $f$ to $M_i$ by $f_i$
  and, for any $0 \leq t \leq 1$, we define $f_{i,t} : M_i \to S^2$ as
  the composite $f_i \circ \overline{\Phi}_i^{-1} \circ H_t \circ
  \overline{\Phi}$. Since $M_{\mathrm{exc}}$ is the disjoint union of
  the $M_i$'s, for any $0 \leq t \leq 1$, there exists a unique
  continuous map $f_t : M_{\mathrm{exc}} \to S^2$ such that the
  restriction of $f_t$ to $M_i$ equals $f_{i,t}$ for all $i=1,\ldots,
  k$. 

  We claim that $f_t : M_{\mathrm{exc}} \to S^2$ is a painting of
  $\comp$ for all $0\leq t \leq 1$. To this end, we fix $0 \leq t \leq
  1$, consider the map $(\overline{\Phi},f_t)$ and two distinct points
  $[q_i],[q_j] \in M_{\mathrm{exc}}$ belonging to $M_i$ and $M_j$
  respectively. We aim to show that $(\overline{\Phi}([q_i]),f_t ([q_i])) \neq
  (\overline{\Phi}([q_j]),f_t ([q_j])) $. If $i = j$, since the restriction of $\overline{\Phi}$
  to $M_i = M_j$ is a homeomorphism and since $[q_i] \neq [q_j]$, it
  follows that $\overline {\Phi}([q_i]) \neq \overline{\Phi}([q_j])$,
  so that the result follows. Suppose next that $i \neq
  j$. If $\overline {\Phi}([q_i]) \neq \overline{\Phi}([q_j])$, then
  the result follows, so we may assume that  $\overline {\Phi}([q_i])
  = \overline{\Phi}([q_j]) =: w$. If $f_t ([q_i]) =
  f_t([q_j])$, then the points $\overline{\Phi}^{-1}_i(H_t(w)) \in
  M_i$ and $\overline{\Phi}^{-1}_j(H_t(w)) \in
  M_j$ are distinct (since $i \neq j$), but are such that
  $$
  (\overline{\Phi}(\overline{\Phi}^{-1}_i(H_t(w))),f(\overline{\Phi}^{-1}_i(H_t(w))))
  = 
  (\overline{\Phi}(\overline{\Phi}^{-1}_j(H_t(w))),f(\overline{\Phi}^{-1}_j(H_t(w)))). $$
  \noindent
  This implies that $f$ is not a painting, which is absurd. Hence, $f_t ([q_i]) \neq
  f_t([q_j])$, which implies that $(\overline{\Phi}([q_i]),f_t ([q_i])) \neq
  (\overline{\Phi}([q_j]),f_t ([q_j])) $, as desired.

  We set $f':= f_1$. Since $f_t : M_{\mathrm{exc}} \to S^2$ is a painting of
  $\comp$ for all $0\leq t \leq 1$, $f = f_0$ and $f'$ are homotopic
  through paintings. Moreover, the restriction of $f'$ to $M_i$ is
  constant by construction for any $i=1,\ldots, k$, thus completing
  the proof.
\end{proof}

\subsection{The Duistermaat-Heckman function}\label{sec:duist-heckm-funct-1}
In order to prove the main result of this section (see Theorem
\ref{thm:possibilities_DH}), we prove the following intermediate
result. To this end, we
recall that \eqref{eq:10} and \eqref{eq:59}
hold, that for any vertex $v$ and any edge $e$ of $\Phi(M)$, the
preimages $\Phi^{-1}(v)$ and $\Phi^{-1}(e)$ are a two dimensional
sphere and a four dimensional manifold respectively. Moreover, 
by Lemma \ref{lemma:s_invariant}, there exists $s \in \Z$
such that, if $v \in \mathcal{F}_{\min}$ and $e$ is the edge that
comes out of $\mathcal{F}_{\min}$ that is incident to $v$, then the self-intersection
of $\Phi^{-1}(v)$ in $\Phi^{-1}(e)$ equals $s$ (and does not depend on
$v$).

\begin{proposition}\label{lemma:number_iso_fixed_points}
  Let $\comp$ be a compact normalized monotone tall complexity one $T$-space of dimension $2n$. Let $v \in \mathcal{F}_{\min}$ be a vertex and let $e$ be
  the edge incident to $v$ that comes out of $\mathcal{F}_{\min}$. Let
  $s \in \Z$ be as in Lemma \ref{lemma:s_invariant} and let $k \geq 0$ be the number of isolated fixed
  points contained in $\Phi^{-1}(e)$. Then the restriction of the Duistermaat-Heckman
  function $DH \comp$ to $e$ is the function
  \begin{equation}
    \label{eq:38}
    \begin{split}
      e &\to \R \\
      w &\mapsto 2 - s \langle w, \nu_{\min} \rangle - k \rho(w),
    \end{split}
  \end{equation}
  \noindent
  where $\rho : \g^* \to \R$ is the function given by
  $$ w \mapsto
  \begin{cases}
    0 & \text{if } \langle w, \nu_{\min} \rangle \leq 0 \\
    \langle w, \nu_{\min} \rangle & \text{if } \langle w, \nu_{\min} \rangle \geq 0.
  \end{cases}
  $$
\end{proposition}

\begin{proof}
  Let $\alpha_1,\ldots, \alpha_{n-1}$ be the
  non-zero isotropy weights of $\Phi^{-1}(v)$ ordered so
  that \eqref{eq:59} holds; in particular, $\langle \alpha_{n-1},
  \nu_{\min} \rangle = 1$ (see Lemma
  \ref{lemma:weight_edge_out_facet} and \eqref{eq:16}). Since $e$ is the edge that comes out of
  $\mathcal{F}_{\min}$ and is incident to $v$, by Lemma
  \ref{lemma:weight_edge_out_facet}, there exists $t_{\max} \in \Z_{>0}$
  such that
  $$ e = \{ v + t\alpha_{n-1} \mid 0 \leq t \leq t_{\max}\}. $$

  Let $(M_e, \omega_e,
  \Phi_e)$ be the sheet corresponding to $e$ as in
  \eqref{eq:11}. Since $\alpha_{n-1} \in \ell^*$ is primitive, the
  stabilizer $H_e$ of $(M_e, \omega_e,
  \Phi_e)$ is precisely $\exp\left(\mathrm{Ann}\left( \R \langle
      \alpha_{n-1}\rangle \right) \right)$. Since $\comp$ is tall and compact,
  $(M_e, \omega_e,
  \Phi_e)$ is a compact tall complexity one Hamiltonian $T/H_e \simeq S^1$-space by
  property \ref{item:12} of Corollary
  \ref{cor:comp_preserving} and by Proposition \ref{prop:tall=complexity_preserving}. 
  In what follows, $\Phi_e$ is chosen so that $\Phi_e(M_e)
  = [0,t_{\max}]$. Moreover, the isolated fixed points for the $S^1$-action
  on $M_e$ are precisely the isolated fixed points for the $T$-action contained in
  $M_e$. By Corollary \ref{cor:DH_boundary_comp_pres}, the restriction
  of the Duistermaat-Heckman function $DH \comp$ to $e$ equals $DH
  (M_e,\omega_e,\Phi_e)$. Since $\langle v +
  t\alpha_{n-1},\nu_{\min} \rangle = t-1$, in order to
  prove that \eqref{eq:38} holds, it
  suffices to check that the Duistermaat-Heckman function of
  $(M_e,\omega_e,\Phi_e)$ is the function $[0,t_{\max}] \to \R$ given
  by
  \begin{equation}
    \label{eq:42}
    t \mapsto 2 - s(t-1) -k \Theta(t-1),
  \end{equation}
  where $\Theta: \R \to \R$ is the function 
  $$ t \mapsto
  \begin{cases}
    0 & \text{if } t \leq 0 \\
    t & \text{if } t \geq 0.
  \end{cases}
  $$

  By Proposition
  \ref{prop:iso_fixed_pts_monotone} and the definition of $\Phi_e$, any isolated fixed point $p \in M_e$ for the
  $S^1$-action satisfies $\Phi_e(p) = 1$ and its isotropy weights are
  $+1,-1$. In particular, if $k > 1$, then $t_{\max} > 1$. Since $(M_e,\omega_e,\Phi_e)$ is tall and since the domain
  of the Duistermaat-Heckman function is $\Phi_e(M_e) = [0,t_{\mathrm{max}}]$ (see Definition
  \ref{Def: DHfucntion}), by \cite[Lemma
  2.12]{karshon}, we have that
  \begin{equation}
    \label{eq:39}
    DH (M_e,\omega_e,\Phi_e)(t) =
    \int\limits_{\Phi_e^{-1}(v)} \omega_e - t
    c_1(L_e)[\Phi_e^{-1}(v)] - k \Theta(t-1) \qquad \text{for all } t
    \in [0,t_{\mathrm{max}}],
  \end{equation}
  \noindent
  where $c_1(L_e)$ is the first Chern class of
  the normal bundle $L_e$ to $\Phi_e^{-1}(v)$ in $M_e$.
  Since $\comp$ is normalized monotone, since $\Phi^{-1}_e(v) =
  \Phi^{-1}(v)$ is a sphere by Lemma \ref{lemma:chern}, and
  by Lemma \ref{prop:minimum_facet}, we have that
  \begin{equation}
    \label{eq:40}
    \int\limits_{\Phi_e^{-1}(v)} \omega_e = c_1(M)[\Phi^{-1}(v)] = 2 +
    \sum\limits_{j=1}^{n-1}c_1(L_j)[\Phi^{-1}(v)] = 2 + c_1(L_{n-1})[\Phi^{-1}(v)],
  \end{equation}
  \noindent
  where $L_1\oplus \ldots \oplus L_{n-1}$ is the $T$-equivariant
  splitting of the normal bundle to $\Phi^{-1}(v)$ in $M$, and the
  last equality follows from Proposition \ref{prop:minimum_facet}.
  Combining equations \eqref{eq:39} and \eqref{eq:40}, and observing
  that $L_e = L_{n-1}$, we have that
  $$ DH (M_e,\omega_e,\Phi_e)(t) =  2 - s(t-1) - k \Theta(t-1), $$
  as desired.
\end{proof}

By Theorem \ref{cor:exceptional_orbits}, the number of isolated fixed points
that lies in the preimage of an edge that comes ouf of
$\mathcal{F}_{\min}$ equals the number of connected components of
$M_{\mathrm{exc}}$. Hence, we can use Proposition
\ref{lemma:number_iso_fixed_points} to prove the following result.

\begin{theorem}\label{thm:possibilities_DH}
  Let $\comp$ be a compact normalized monotone tall complexity one $T$-space of dimension $2n$. Let $s \in
  \Z$ be as in Lemma \ref{lemma:s_invariant} and let $k \geq 0$
  be the number of
  connected components of $M_{\mathrm{exc}}$. The Duistermaat-Heckman
  function $DH \comp : \Phi(M) \to \R$ is given by
  \begin{equation}
    \label{eq:57}
    DH \comp (w) = 2 - s \langle w, \nu_{\min} \rangle - k \rho(w),
  \end{equation}
  \noindent
  where $\rho : \g^* \to \R$ is the function
  given by
  \begin{equation}
    \label{eq:58}
    w \mapsto
    \begin{cases}
      0 & \text{if } \langle w, \nu_{\min} \rangle \leq 0 \\
    \langle w, \nu_{\min} \rangle & \text{if } \langle w, \nu_{\min} \rangle \geq 0.
    \end{cases}
  \end{equation}
\end{theorem}

\begin{proof}
  First, we show that the interior of the
  intersection
  $$\Phi(M) \cap \{w \in \g^* \mid \pm \langle w, \nu_{\min} \rangle > 0\}$$
  consists entirely
  of regular values of $\Phi$. To this end, suppose that
  $w \in \g^* \smallsetminus \partial \Phi(M)$ is a singular value of
  $\Phi$. Hence, there exists $q \in \Phi^{-1}(w)$ that has a stabilizer
  has positive dimension. By \cite[Corollary 4.9]{kt1}, $q$ is
  exceptional. Thus, by Proposition
  \ref{lemma:exceptional_sheet_dimension_1} and Lemma
  \ref{lemma:exceptional_in_max_sheet}, $\Phi(q) \in
  \{w \in \g^* \mid \langle w, \nu_{\min} \rangle = 0\}$, as
  desired. Hence, by Remark \ref{rmk:polynomial}, the restriction of $DH \comp$ to
  $\Phi(M) \cap \{w \in \g^* \mid \pm \langle w, \nu_{\min}
  \rangle > 0\}$ is the restriction of an affine
  function $f^{\pm}: \g^* \to \R$ of the form $f^{\pm}(w) =
  c^{\pm} + \langle w,\beta^{\pm}\rangle $ for some $c^{\pm} \in \R$
  and some $\beta^{\pm} \in
  \ell$. 

  We fix a vertex $v \in \mathcal{F}_{\min}$, we let $e$ be the edge
  that comes out of $v$ and $\alpha_1,\ldots, \alpha_{n-1}$ be the non-zero
  isotropy weights of $\Phi^{-1}(v)$ ordered so that \eqref{eq:10} and
  \eqref{eq:59} hold.  Since $DH \comp$ is continuous (by Theorem
  \ref{thm:DH_function_cts}) and 
  since the restriction of $DH \comp$ to $\mathcal{F}_{\min}$ is
  constant by Proposition \ref{prop:minimum_facet}, the
  restriction of the affine function $f^-$ to the affine hyperplane
  supporting $\mathcal{F}_{\min}$ is constant. Since
  $f^-$ is an affine function and since $\mathcal{F}_{\min}$  is
  supported on the hyperplane $v + \R \langle \alpha_1,\ldots,
  \alpha_{n-2} \rangle$, the restriction of the linear part of
  $f^-$ to $\R \langle \alpha_1,\ldots, \alpha_{n-2} \rangle$
  is identically zero. In other words, $\beta^- \in
  \mathrm{Ann}\left(\R\langle \alpha_1,\ldots, \alpha_{n-2}\rangle
  \right)$. Hence, since $\nu_{\min} \in \mathrm{Ann}\left(\R\langle \alpha_1,\ldots, \alpha_{n-2}\rangle
  \right)$, since $\beta^-, \nu_{\min} \in \ell$, and since
  $\nu_{\min}$ is primitive, there exists $\lambda^- \in \Z$ such that
  $\beta^- = \lambda^- \nu_{\min}$. 

  By Lemma \ref{lemma:weight_edge_out_facet}, there exists $t_{\max} \in  \Z_{>0}$
  such that
  $$ e = \{ v + t\alpha_{n-1} \mid 0 \leq t \leq t_{\max}\}. $$
  By \eqref{eq:16} and since $v \in \mathcal{F}_{\min}$, $\langle v + t\alpha_{n-1},\nu_{\min} \rangle =
  -1+t$ for all $0
  \leq t \leq t_{\max}$. Therefore, by Proposition \ref{lemma:number_iso_fixed_points}, the restriction
  of $DH \comp$ to $e \cap \{w \in \g^* \mid  \langle w, \nu_{\min}
  \rangle < 0\}$ is given
  by the function that sends $t$ to $2 + s -st$, where $0 \leq t < 1$. Hence,
  \begin{equation}
    \label{eq:46}
     c^- + \lambda^-(-1+t) = 2 + s -s t \text{ for all } 0 < t < 1.
  \end{equation}
  Equation \eqref{eq:46} readily implies that $\lambda^- = - s$
  and $c^- = 2$. Hence, $f^-(w) = 2 - s\langle w, \nu_{\min} \rangle$. 

  We split the remainder of the proof in two cases, depending on
  whether $t_{\max} = 1$ or $t_{\max} \geq 2$. In the former case, the other vertex $v'$ of $\Phi(M)$ that is incident
  to $e$ lies on the linear hyperplane $\{w \in \g^* \mid \langle
  w,\nu_{\min} \rangle = 0\}$. Hence, by Corollary
  \ref{cor:no_iso_fixed_points}, there are no isolated fixed
  points. Therefore, by Lemma
  \ref{lemma:no_isolated_implies_affine_DH}, the function $DH\comp$ is
  the restriction of an affine function. Since the interior of $\Phi(M) \cap \{w \in
  \g^* \mid - \langle w, \nu_{\min} \rangle > 0\}$ is not empty and
  since the restriction of $DH\comp$ to this subset equals the affine
  function $f^-(w) = 2 - s\langle w, \nu_{\min} \rangle$, it
  follows that
  $$ DH \comp (w) = 2 - s\langle w, \nu_{\min} \rangle = 2
  - s\langle w, \nu_{\min} \rangle - k \rho(w) \quad \text{for all }
  w \in \Phi(M), $$
  where the last equality follows from the fact that $k=0$, since
  there are no isolated fixed points (see Lemma
  \ref{lemma:fixed_point_exceptional} and Corollary \ref{cor:exc=exc}).

  It remains to consider the case $t_{\max} \geq 2$. Since $DH \comp$ is continuous, then the restriction of $f^-$ and
  $f^+$ to $\Phi(M) \cap \{w \in \g^* \mid  \langle w, \nu_{\min}
  \rangle = 0\}$ are equal. Hence, the
  restriction of the affine function $f^+$ to $\Phi(M) \cap \{w \in \g^* \mid  \langle w, \nu_{\min}
  \rangle = 0\}$ equals 2. \color{black} By Proposition \ref{prop:mom_map_image}, $\Phi(M)$ is a reflexive
  Delzant polytope; thus, by Lemma
  \ref{lemma:reflexive_interior}, the origin lies in the interior of
  $\Phi(M)$. Therefore, the interior of $\Phi(M) \cap \{w \in \g^* \mid  \langle w, \nu_{\min}
  \rangle = 0\}$ 
  in $\{w \in \g^* \mid  \langle w, \nu_{\min}
  \rangle = 0\}$ is non-empty. Hence, the restriction of the linear part of $f^+$ to the
  hyperplane $\{w \in \g^* \mid  \langle w, \nu_{\min}
  \rangle = 0\}$ is identically zero. Arguing as above, this implies
  that there exists $\lambda^+ \in \Z$ such that $\beta^+ = \lambda^+ \nu_{\min}$.

  Since $t_{\max} \geq 2$, the intersection $e \cap \{w \in \g^* \mid  \langle w, \nu_{\min}
  \rangle > 0\}$ is not empty. By Proposition \ref{lemma:number_iso_fixed_points}, the restriction
  of $DH \comp$ to $e \cap \{w \in \g^* \mid  \langle w, \nu_{\min}
  \rangle > 0\}$ is
  the function that sends $t$ to $2 + s - st -k(t-1)$ for $ 1 < t \leq
  t_{\max}$. Hence, we have
  that
  \begin{equation}
    \label{eq:56}
    \lambda^+(-1+t) + 2 = 2 + s - st - k(t-1) \text{ for all } 1 < t < t_{\max}. 
  \end{equation}
  Equation \eqref{eq:56} readily implies that $\lambda^+ = - s -
  k$. Hence, $f^+(w) = 2 -s\langle w, \nu_{\min} \rangle-k\langle w, \nu_{\min} \rangle$.

  Thus the restriction of $DH \comp$ to $\Phi(M) \smallsetminus
  \{w \in \g^* \mid  \langle w, \nu_{\min}
  \rangle = 0\}$ is the function given by
  $$ w \mapsto
  \begin{cases}
    2 - s \langle w, \nu_{\min} \rangle & \text{if } \langle w, \nu_{\min} \rangle < 0 \\
    2 - s \langle w, \nu_{\min} \rangle - k \langle w, \nu_{\min} \rangle & \text{if } \langle w, \nu_{\min} \rangle >0.
  \end{cases} $$
  This function equals the restriction of \eqref{eq:57} on the dense
  subset  $\Phi(M) \smallsetminus
  \{w \in \g^* \mid  \langle w, \nu_{\min}
  \rangle = 0\}$ of $\Phi(M)$. Since $DH\comp$ is continuous,
  equation \eqref{eq:57} holds. 
\end{proof}

\subsection{The Duistermaat-Heckman function is a complete invariant}\label{sec:class-finit-result}
We can proceed with the proof of our first main result.

\begin{proof}[Proof of Theorem \ref{thm:DH_classifies}]\label{proof theorem thm:DH_classifies}
  Let $\comp$ and $(M',\omega',\Phi')$ be compact monotone tall complexity one $T$-spaces of
  dimension $2n$. If they are isomorphic, then they have equal
  Duistermaat-Heckman measures. By Theorem
  \ref{thm:DH_function_cts}, they have equal
  Duistermaat-Heckman functions.

  Suppose conversely that they have equal Duistermaat-Heckman
  functions. First, we show that it suffices to show the result under
  the additional assumption that both spaces are normalized. To this
  end, assume this special case. Since $\comp$ and $(M',\omega',\Phi')$ are monotone, by Corollary
  \ref{cor:mono_normalized}, there exist $\lambda, \lambda' >0$ and
  $v, v' \in \g^*$ such that $(M,\lambda\omega, \lambda \Phi+v)$ and
  $(M',\lambda'\omega', \lambda'\Phi'+v')$ are normalized monotone. Moreover,
  since $\comp$ and $(M',\omega',\Phi')$ are tall and have equal
  Duistermaat-Heckman functions, there exists a Delzant polytope
  $\Delta$ such that $\Phi(M) = \Delta = \Phi'(M')$. Since $(M,\lambda\omega, \lambda \Phi+v)$ and
  $(M',\lambda'\omega', \lambda'\Phi'+v')$ are normalized monotone, by
  Proposition \ref{prop:mom_map_image}, the polytopes $\lambda \Delta + v$ and $\lambda' \Delta + v'$ are
  reflexive Delzant. We observe that there exists at most one
  reflexive Delzant polytope obtained from $\Delta$ by rescaling and
  translating because the vertices of a reflexive Delzant polytope
  are determined by the tangent cones to its vertices (see
  Proposition \ref{prop reflexive eq}). Hence, $\lambda = \lambda'$ and $v = v'$; in particular,
  $\lambda \Delta + v = \lambda' \Delta + v'$. On the other hand, we
  observe that, for any $w \in \lambda \Delta + v = \lambda' \Delta +
  v'$,
  \begin{equation}
    \label{eq:82}
    \begin{split}
      DH (M,\lambda\omega, \lambda \Phi+v)(w) & = \lambda
      DH\comp(\lambda w+v) \\
      DH (M',\lambda'\omega', \lambda' \Phi'+v')(w) & = \lambda'
      DH (M',\omega',\Phi') (\lambda' w+v'). 
  \end{split}
  \end{equation}
  This is an immediate consequence of the definition of
  the Duistermaat-Heckman function, of the fact that its restriction to
  the set of regular values is given by \eqref{eq:1}, and of the
  complexity of both spaces being one. Hence, $DH
  (M,\lambda\omega, \lambda \Phi+v) = DH (M',\lambda'\omega', \lambda'
  \Phi'+v')$. By assumption, $(M,\lambda\omega, \lambda \Phi+v)$ and $(M',\lambda'\omega', \lambda'
  \Phi'+v')$ are isomorphic. Therefore, by Lemma
  \ref{lemma:monotone_normalized}, it follows that  $\comp$ and
  $(M',\omega',\Phi')$ are isomorphic, as desired.

  It remains to prove the result under the additional assumption that
  the spaces are normalized. By Remark \ref{rmk:thm_enough} and Theorem
  \ref{thm:DH_function_cts}, the two spaces have equal Duistermaat-Heckman
  measures. Moreover, both spaces
  have genus zero by Lemma
  \ref{lemma:chern}, and equal moment map images, both of which are equal to a reflexive Delzant
  polytope $\Delta$ by Proposition
  \ref{prop:mom_map_image}. Since the Duistermaat-Heckman functions
  are equal , there is a
  facet $\mathcal{F}_{\min}$ of $\Delta$ that is a minimal facet for
  both spaces. Let $s, s' \in \Z$ be the integers as in Lemma
  \ref{lemma:s_invariant} for $\comp$ and $(M',\omega',\Phi')$ respectively.
  Since the Duistermaat-Heckman functions of
  $\comp$ and $(M',\omega',\Phi')$ are equal, by Theorem
  \ref{thm:possibilities_DH},
  \begin{equation}
    \label{eq:18}
     2 - s\langle w ,\nu_{\min} \rangle - k \rho(w) =  2 - s'\langle w ,\nu_{\min} \rangle - k' \rho(w) \text{ for
       all } w \in \Delta,
   \end{equation}
   where $\rho : \g^* \to \R$ is the function of equation
   \eqref{eq:58}. Since the interior of $\Delta$ is not empty and
   by \eqref{eq:10}, there exists a point $w \in \Delta$ such
   that $-1 < \langle w, \nu_{\min} \rangle < 0$. Evaluating both sides of \eqref{eq:18} at
   this point, we obtain that $s = s'$. By the same argument and since $\Delta$ is reflexive so that the origin is an
   interior point of $\Delta$, there exists $w' \in \Delta$
   such that
   $\langle w', \nu_{\min} \rangle > 0$. Evaluating both sides of \eqref{eq:18} at $w'$, we
   obtain that $k = k'$. Hence, either both $M_{\mathrm{exc}}$ and
   $M'_{\mathrm{exc}}$ are empty or neither is.

   We suppose first that neither $M_{\mathrm{exc}}$ nor $M'_{\mathrm{exc}}$
   is empty. As in the proof of Theorem
   \ref{prop:trivial_painting}, we write
   $$ M_{\mathrm{exc}} = \coprod\limits_{j=1}^k M_j \qquad \text{ and } \qquad
   M'_{\mathrm{exc}} = \coprod\limits_{j=1}^k M'_j,$$
   \noindent
   where $M_j$ is a connected component of
   $M_{\mathrm{exc}}$ for all
   $j=1,\ldots, k$, and analogously for $M'_j$ and $M_{\mathrm{exc}}'$. 
   By Theorem
   \ref{prop:trivial_painting}, there exist trivial paintings $f :
   M_{\mathrm{exc}} \to S^2$ and $f':
   M'_{\mathrm{exc}} \to S^2$ of $\comp$ and $(M',\omega',\Phi')$
   respectively (see Definition \ref{defn:trivial_painting}). Since
   $f, f'$ are paintings, if $i \neq j$,
   $f(M_i) \neq f(M_j)$ and $f'(M'_i) \neq f'(M'_j)$. In
   particular, the images of both $f$ and $f'$ consist of $k$
   distinct points in $S^2$. We claim that we may
   assume that $f(M_j) = f'(M'_j)$
   for all $j = 1,\ldots, k$. To see this, we can
   use the argument of \cite{DeVito} to construct an
   orientation-preserving diffeomorphism $\xi : S^2 \to S^2$ such that
   $\xi(f(M_j)) = f'(M'_j)$ for all $j =1,\ldots, k$. The composite
   $\xi \circ f$ is a trivial painting of $\comp$ that is equivalent to $f$. 

   By Theorem \ref{cor:exceptional_orbits}, for all $j=1,\ldots, n$,
   the orbital moment maps $\overline{\Phi}$ (respectively $\overline{\Phi}'$) maps
   each connected component of $M_j$ (respectively
   $M'_j$) homeomorphically onto
   $$ \Delta \cap \{ w \in \g^* \mid \langle w, \nu_{\min} \rangle =
   0\}, $$
   \noindent
   where we use the fact that $\Phi(M) = \Delta = \Phi'(M')$. If $\overline{\Phi}_j$ (respectively
   $\overline{\Phi}'_j$) denotes the restriction of the orbital moment
   map to $M_j$ (respectively $M_j'$), then the composite $i_j:= (\overline{\Phi}'_j)^{-1}
   \circ \overline{\Phi}_j : M_j \to M'_j$ is
   a homeomorphism that satisfies $\overline{\Phi} =
   \overline{\Phi}' \circ i_j$. Hence, since $M_{\mathrm{exc}}$ and
   $M'_{\mathrm{exc}}$ have the same number of connected components,
   there is a unique map $i : M_{\mathrm{exc}} \to M'_{\mathrm{exc}}$
   that, when restricted to $M_j$, equals $i_j$ for all
   $j=1,\ldots,n$. Thus $i$ is a homeomorphism such that $\overline{\Phi} =
   \overline{\Phi}' \circ i$. Moreover, since the moment map images of
   $\comp$ and $(M',\omega',\Phi')$ are equal, by Remark
   \ref{rmk:image_symplectic_slice}, the homeomorphism $i$ maps each orbit to an orbit with
   the same symplectic slice representation. Hence, $i :
   M_{\mathrm{exc}} \to M'_{\mathrm{exc}}$ is an isomorphism of
   exceptional orbits such that $f' = i \circ f$. Thus $\comp$ and
   $(M',\omega',\Phi')$ have equivalent paintings.

   If $M_{\mathrm{exc}} = \emptyset = M'_{\mathrm{exc}}$, then $\comp$ and of
   $(M',\omega',\Phi')$ also have equivalent paintings
   (trivially). Hence, in either case, the result follows by Theorem \ref{sue
     and yael}.
\end{proof}

\begin{remark}
  We observe that, by Corollary \ref{cor:DH_polytope}, Theorem
  \ref{thm:DH_classifies} can be restated equivalently as saying that
  two compact monotone tall complexity one $T$-spaces are isomorphic
  if and only if they have equal Duistermaat-Heckman polytopes.
\end{remark}

\section{Realizability, extension to toric and finiteness results}\label{sec:an-expl-real}

\subsection{Necessary conditions for the realization and a finiteness
  result}\label{sec:necess-cond-finit}
To state the first result of this subsection, we recall that,
by Lemma \ref{lemma:s_invariant}, there exists $s \in \Z$ such that, if $v \in \mathcal{F}_{\min}$ is a
vertex and $e$ is the edge of $\Phi(M)$ that comes out of
$\mathcal{F}_{\min}$ and is incident to $v$, then the
self-intersection  of the sphere $\Phi^{-1}(v)$ in the four-dimensional submanifold $\Phi^{-1}(e)$ is
$s$.

\begin{proposition}\label{prop:necessary}
  Let $\comp$ be a normalized monotone tall complexity one Hamiltonian
  $T$-space, let $s \in \Z$ be as above and let $k \in \Z$ be the
  number of connected components of $M_{\mathrm{exc}}$. The pair
  $(s,k) \in \Z^2$ belongs to the set
  $$\{(0,0), (-1,0), (-1,1), (-1,2) \}.$$
\end{proposition}
\begin{proof}
  We fix a vertex $v \in \mathcal{F}_{\min}$ and we let $e=e_{n-1}$ be the edge
  of $\Phi(M)$ that comes out of $\mathcal{F}_{\min}$ and is incident
  to $v$. The normal bundle $N$ to $\Sigma:=\Phi^{-1}(v)$ splits $T$-equivariantly as 
  $L_1\oplus \cdots \oplus L_{n-1}$. By Proposition
  \ref{prop:minimum_facet}, the first Chern number $c_1(L_j)[\Sigma]$,
  which agrees with the self-intersection of $\Sigma$ in $\Phi^{-1}(e_j)$, equals zero for all $j=1,\ldots,n-2$. 
  Therefore, by Lemma \ref{lemma:chern} and equation \eqref{eq:30} in its proof,
  $$
  0<2+c_1(L_{n-1})[\Sigma]=2+s\,.
  $$
  Moreover, by \eqref{eq: value c1 on Li} in Lemma \ref{Lemma:DHnearfixedSurface}, $s=c_1(L_{n-1})[\Sigma]\leq 0$. Hence,
  we conclude that $s\in \{0,-1\}$.

  Suppose first that $s= 0$. By Theorem \ref{thm:possibilities_DH},
  the Duistermaat-Heckman function $DH \comp : \Phi(M) \to \R$ is given by $w \mapsto
  2 - k\rho(w)$, where $\rho : \g^* \to \R$ is the function of
  \eqref{eq:58}. Since $k \geq 0$, $2 - k\rho(w) \leq 2$
  for all $w \in \Phi(M)$. Since $DH\comp(v) = 2$ and since $v \in
  \mathcal{F}_{\min}$, it follows that $DH\comp(w) = 2$ for all $w \in
  \Phi(M)$. By Proposition \ref{prop:mom_map_image}, $\Phi(M)$ is a
  reflexive (Delzant) polytope. Hence, by Lemma
  \ref{lemma:reflexive_interior}, $\Phi(M)$
  contains the origin in its interior. Thus, by definition of the
  function $\rho$, $k = 0$. 

  Suppose that $s = -1$. We must show that $k \leq 2$. To this end, we
  may assume that $ k >
  0$. Let $\alpha \in \ell^*$ be the isotropy weight of $\Phi^{-1}(v)$
  such that the edge $e$ is contained in the half-ray $v + \R_{\geq 0}
  \langle \alpha \rangle$. By Lemma \ref{lemma:weight_edge_out_facet},
  let $t_{\max} \in \Z_{>0}$ be such that $ e = \{v
  + t \alpha \mid 0 \leq t \leq t_{\max}\}$. First, we prove that
  $t_{\max} \geq 2$. Let $v' \in \Phi(M)$ be the other vertex to which $e$ is incident. If
  $t_{\max} = 1$, then $v'$ is a vertex of
  $\Phi(M)$ that lies on the linear hyperplane $\{w \in \g^* \mid
  \langle w, \nu_{\min} \rangle = 0\}$. By the Convexity Package
  (Theorem \ref{thm:con_pac}), and Proposition
  \ref{prop:iso_fixed_pts_monotone}, there are no isolated fixed
  points in $\Phi^{-1}(e)$. Hence, by Theorem \ref{cor:exceptional_orbits}, $M_{\mathrm{exc}} = \emptyset$,
  a contradiction\. Hence, $t_{\max} \geq 2$, as desired. As a
  consequence, $v + 2\alpha \in \Phi(M)$. To conclude
  the proof, we evaluate $DH\comp$ at $v + 2\alpha$.
  By \eqref{eq:16} and the fact that $v\in \mathcal{F}_{\min}$,
  we obtain that $\langle v+2\alpha, \nu_{\min}\rangle = 1$. Moreover, since $\comp$ is tall,
  $DH\comp(w) > 0 $ for all $w \in \Phi(M)$. Therefore, by
  \eqref{eq:57}, 
  $$DH\comp(v+2\alpha)= 2 + 1- k > 0.$$
  Since $k\in \Z$, it follows that $k \leq 2$.
\end{proof}
  
We proceed with the proof of our second main result. 

\begin{proof}[Proof of Theorem \ref{thm:finitely_many_DH}]
  Suppose that $\Phi(M) = \Delta$ is reflexive Delzant. Since $\comp$
  is monotone, by Lemma \ref{lemma:normalized_complexity_preserving}, $\comp$ is normalized
  monotone. Since there are finitely many facets of $\Phi(M)$ and
  since, by Proposition \ref{prop:necessary}, there are finitely many
  possibilities for $(s,k)$, by Theorem
  \ref{thm:possibilities_DH}, there are finitely many
  possibilities for $DH\comp$. Hence, the result follows from Theorem \ref{thm:DH_classifies}.
\end{proof}

By Theorem \ref{cor:exceptional_orbits} and Proposition
\ref{prop:necessary}, and since there are
precisely as many edges that come out of $\mathcal{F}_{\min}$ as there
are vertices of $\mathcal{F}_{\min}$, we obtain the following bound on
the number of isolated fixed points.

\begin{corollary}\label{cor:total_number_iso_fixed_points}
  Let $\comp$ be a normalized monotone tall complexity one $T$-space of dimension $2n$. If
  $m$ is the number of vertices of a minimal facet, then there are precisely
  either zero, $m$ or $2m$ isolated fixed points in $M$.
\end{corollary}

The next result gives a combinatorial property of the moment map image
of a normalized monotone tall complexity one $T$-space such that 
$M_{\mathrm{exc}}$ has two connected components.

\begin{corollary}\label{cor:k=2}
  Let $\comp$ be a normalized monotone tall complexity one $T$-space of dimension $2n$. Let
  $\mathcal{F}_{\min}$ be a minimal facet of $\Phi(M)$ supported on
  the affine hyperplane $\{w \in \g^* \mid \langle w, \nu_{\min} \rangle =
  -1\}$. If $M_{\mathrm{exc}}$ has two connected components, then
  there exists a minimal facet $\mathcal{F}'_{\min}$ of $\Phi(M)$ supported on the affine hyperplane $\{w \in \g^* \mid \langle w, -\nu_{\min} \rangle =
  -1\}$. In particular, $\Phi(M)$ is contained in the strip $\{w \in
  \g^* \mid -1 \leq \langle w, \nu_{\min} \rangle \leq 1\}$. 
\end{corollary}

\begin{proof}
  We fix a vertex $v \in \mathcal{F}_{\min}$. Since the number $k$ of
  connected components of $M_{\mathrm{exc}}$ is 2, by Proposition
  \ref{prop:necessary}, it follows that $s=-1$. In particular, by
  Theorem \ref{thm:possibilities_DH}, the Duistermaat-Heckman function
  $DH \comp: \Phi(M) \to \R$ of $\comp$ is given by 
  \begin{equation}
    \label{eq:71}
    DH \comp(w) = 2 + \langle w, \nu_{\min} \rangle - 2 \rho(w),
  \end{equation}
  \noindent
  where $\rho : \Phi(M) \to \R$ is the non-negative function given by
  \eqref{eq:58}. Since $v \in \mathcal{F}_{\min}$, $DH
  \comp(v) = 1$. Hence, since $\mathcal{F}_{\min}$ is a minimal facet, the minimal value of $DH \comp$ equals 1.

  Let $e$ be the edge of $\Phi(M)$ that comes out of
  $\mathcal{F}_{\min}$ and is incident to $v$, and let $\alpha \in
  \ell^*$ be the isotropy weight of $\Phi^{-1}(v)$ so that $e$ is
  contained in the half-ray $v + \R_{\geq 0} \langle \alpha
  \rangle$. By \eqref{eq:16} and
  Proposition \ref{prop:minimum_facet}, 
   $\langle \alpha, \nu_{\min}
  \rangle =1$ and there exists $t_{\max} \in
  \Z_{>0}$ such that $ e = \{v + t\alpha \mid 0 \leq t \leq t_{\max}
  \}$. We set $v':= v + t_{\max} \alpha$; this is a vertex of
  $\Phi(M)$. Moreover, we observe that $\langle v', \nu_{\min} \rangle
  = t_{\max} -1$. Since $s=-1$ and $k =2$, arguing as in the last
  paragraph of the proof of
  Proposition \ref{prop:necessary}, $t_{\max} \geq
  2$. In particular, by \eqref{eq:71} and since the minimal
  value of $DH\comp$ is 1, 
  $$ DH\comp (v') = 3 - t_{\max} \geq 1,$$
  \noindent
  whence $t_{\max} \leq 2$. Hence, $t_{\max} = 2$ and $ DH\comp (v')
  =1$, so that $DH \comp$ attains its minimum at $v'$. By Proposition
  \ref{prop:minimum_facet}, $v'$ lies on a minimal facet
  $\mathcal{F}_{\min}'$. In fact, $\mathcal{F}_{\min}'$ is contained
  in the connected component of the level set $(DH \comp)^{-1}(1)$
  that contains $v'$. By \eqref{eq:71}, the latter is given by the
  affine hyperplane $\{w \in \g^* \mid \langle w, -\nu_{\min} \rangle =
  -1\}$. Since the affine span of a facet is an affine hyperplane, the
  first statement follows. The second statement follows at once from
  the first and the fact that $\Phi(M)$ is contained in the
  intersection $\{w \in \g^* \mid \langle w, \nu_{\min} \rangle \geq 
  -1\} \cap \{w \in \g^* \mid \langle w, -\nu_{\min} \rangle \geq
  -1\}$. 
\end{proof}

The next result is the most important building block of the main finiteness result of this paper,
Corollary \ref{prop:finitely_many}. 
\begin{corollary}\label{cor:finite_normalized}
  Given a reflexive Delzant polytope $\Delta$ in $\g^*$, there are
  finitely many isomorphism classes of normalized monotone tall complexity one $T$-spaces with moment map image equal to
  $\Delta$.
\end{corollary}
\begin{proof}
  Let $\mathcal{F}$ be a facet of $\Delta$ and let $(s,k) \in \{(0,0), (-1,0), (-1,1), (-1,2) \}$. Since $\Delta$ has
  finitely many facets and by Proposition \ref{prop:necessary}, it suffices to show that there are finitely
  many isomorphism classes of normalized monotone tall complexity one $T$-spaces with moment map image equal to
  $\Delta$ such that
  \begin{itemize}[leftmargin=*]
  \item $\mathcal{F}$ is a minimal facet of $\Phi(M)$,
  \item given any vertex $v \in \mathcal{F}$ and the edge $e$ of
    $\Phi(M)$ that comes out of $\mathcal{F}$ and is incident to $v$,
    the self-intersection of $\Phi^{-1}(v)$ in $\Phi^{-1}(e)$ equals
    $s$ (cf. Lemma \ref{lemma:s_invariant}), and 
  \item the set of exceptional orbits has precisely $k$ connected
    components. 
  \end{itemize}
  If there is no compact normalized monotone tall complexity one $T$-space with the above properties, there is
  nothing to prove, so we may assume that there exists such a space $\comp$. By Theorem \ref{thm:possibilities_DH}, the data $\Delta,
  \mathcal{F}$ and $(s,k)$ determine uniquely the Duistermaat-Heckman
  function of $\comp$. Hence, by Theorem \ref{thm:DH_classifies},
  there is exactly one isomorphism class of compact normalized monotone tall complexity one $T$-space with the above properties, as desired.
\end{proof}

Theorem \ref{thm:DH_classifies} and Corollary
\ref{cor:finite_normalized} allow us to prove Corollary
\ref{prop:finitely_many}, thus answering a question posed to us by
Yael Karshon. We recall that two Hamiltonian $T$-spaces $(M_1,\omega_1,\Phi_1)$ and
$(M_2,\omega_2,\Phi_2)$ are {\bf equivalent} if there exists a
symplectomorphism $\Psi : (M_1,\omega_1) \to (M_2,\omega_2)$ and an affine transformation $a \in \mathrm{GL}(\ell^*) \ltimes
\mathfrak{t}^*$ such that $\Phi_2
\circ \Psi  = a \circ \Phi_1$. In this case, we write $(M_1,\omega_1,\Phi_1) \sim (M_2,\omega_2,\Phi_2)$.

\begin{remark}\label{rmk:equivalence}
  \mbox{}
  \begin{itemize}[leftmargin=*]
  \item In the above notion of equivalence, the reason why we restrict to elements in $\mathrm{GL}(\ell^*) \ltimes
    \mathfrak{t}^*$ is the following: Given an effective Hamiltonian $T$-space
    $\comp$ and an affine transformation $a$ of $\mathfrak{t}^*$, the
    triple $(M,\omega, a \circ \Phi)$ is an effective Hamiltonian $T$-space if
    and only if $a \in \mathrm{GL}(\ell^*) \ltimes
    \mathfrak{t}^*$.
  \item Isomorphic Hamiltonian $T$-spaces
    in the sense of Definition \ref{def hamiltonian space}
    are necessarily equivalent, but the converse need not hold.
  \end{itemize}
\end{remark}

\begin{proof}[Proof of Corollary \ref{prop:finitely_many}]
  We fix $n$ and we denote the set of equivalence classes of compact tall
  complexity one $T$-spaces of dimension $2n$ with first Chern class equal to the class
  of the symplectic form by $\mathcal{M}_n$. By Definition \ref{monotone ham space}, any normalized monotone tall
  complexity one $T$-space of dimension $2n$ is such that its first
  Chern class equals the class of the symplectic form. We define an
  auxiliary equivalence relation $\approx$ on the set of normalized monotone tall
  complexity one $T$-spaces of dimension $2n$ as follows: Given two
  such spaces
  $(M_1,\omega_1,\Phi_1), (M_2,\omega_2,\Phi_2)$, we say that $
  (M_1,\omega_1,\Phi_1) \approx (M_2,\omega_2,\Phi_2)$ if there exists
  a linear transformation $l \in \mathrm{GL}(\ell^*)$ such
  that  $\Phi_2 \circ \Psi  = l \circ \Phi_1$. We denote the set of
  $\approx$-equivalence classes of normalized monotone tall
  complexity one $T$-spaces of dimension $2n$ by $\mathcal{NM}_n$.

  We observe that there is a natural map $\mathcal{NM}_n \to
  \mathcal{M}_n$ sending the $\approx$-equivalence class of $\comp$ to
  its $\sim$-equivalence class. Moreover, we claim that this map is a
  bijection. First, we show that it is
  injective. Suppose that $(M_1,\omega_1,\Phi_1),
  (M_2,\omega_2,\Phi_2) $ are normalized monotone tall
  complexity one $T$-spaces of dimension $2n$ such that
  $(M_1,\omega_1,\Phi_1) \sim (M_2,\omega_2,\Phi_2)$. Then there
  exists $a \in \mathrm{GL}(\ell^*) \ltimes
  \mathfrak{t}^*$ such that $a (\Phi_1(M_1)) = \Phi_2(M_2)$. We write
  $a = (l, v)$ for unique $l \in \mathrm{GL}(\ell^*)$ and $v \in
  \mathfrak{t}^*$. It suffices to show that $v =
  0$. By Proposition \ref{prop:mom_map_image}, both $\Phi_1(M_1)$ and
  $\Phi_2(M_2)$ are reflexive (Delzant) polytopes. In particular, all
  vertices of $\Phi_1(M_1)$ and of 
  $\Phi_2(M_2)$ lie in $\ell^*$. Since both polytopes have at least
  one vertex and since $l \in \mathrm{GL}(\ell^*)$, $v
  \in \ell^*$. Moreover, by Lemma \ref{lemma:reflexive_interior}, the
  origin is the only interior lattice point in both $\Phi_1(M_1)$ and
  $\Phi_2(M_2)$. Hence, since $a (\Phi_1(M_1)) =
  \Phi_2(M_2)$, the lattice point $v$ lies in the interior of
  $\Phi_2(M_2)$. Thus $v = 0$, as desired. Next we prove
  surjectivity. To this end, let $\comp$ be a compact tall
  complexity one $T$-space of dimension $2n$ with $c_1(M) = [\omega]$. By Proposition \ref{prop:weight_sum}, there exists
  (a unique) $v \in \mathfrak{t}^*$ such that $(M,\omega,\Phi +v)$ is
  normalized. Hence, 
  $\comp \sim (M,\omega,\Phi +v)$, as desired.

  Therefore it suffices to prove that $\mathcal{NM}_n$ is
  finite. To this end, we denote the orbit space of the standard
  $\mathrm{GL}(\ell^*)$-action on the set of reflexive Delzant
  polytopes in $\g^*$ by $\mathcal{RD}_n$. By Proposition
  \ref{prop:mom_map_image}, the map $p: \mathcal{NM}_n \to \mathcal{RD}_n$
  that sends the $\approx$-equivalence class of $\comp$ to the
  $\mathrm{GL}(\ell^*)$-orbit of $\Phi(M)$ is surjective. By Corollary
  \ref{cor:finite}, $\mathcal{RD}_n$ is
  finite. Hence, it suffices to prove that the fibers of the above map
  are finite. We fix a reflexive Delzant polytope $\Delta$ and we
  consider the  map from the set of isomorphism classes of normalized monotone tall complexity one $T$-spaces with moment map image equal to
  $\Delta$ to $p^{-1}([\Delta])$ that sends the isomorphism class of
  $\comp$ to its $\approx$-equivalence class. This map is surjective:
  If $\comp$ is such that $[\Phi(M)] = [\Delta]$, then
  there exists $l \in \mathrm{GL}(\ell^*)$ such that $\Delta =
  l(\Phi(M))$. The isomorphism class of $(M, \omega, l \circ \Phi)$ is
  then mapped to the $\approx$-equivalence class of $\comp$. The
  result now follows from Corollary \ref{cor:finite_normalized}. 
\end{proof}

\subsection{Sufficient conditions for the realization and extension to
a toric action}\label{sec:comb-real-result}
Let $\comp$ be a normalized monotone tall complexity one $T$-space. By
the results of Section
\ref{sec:stuff-that-we}, there exists a quadruple
$(\Delta,\mathcal{F},s,k)$ determined by $\comp$, where $\Delta = \Phi(M)$, $\mathcal{F}$ is
a facet of $\Delta$ that is a minimal facet of $\comp$ (see Definition
\ref{defn:minimal_facet}, $s \in \Z$ is the the self intersection of 
the sphere $\Phi^{-1}(v)$ in $\Phi^{-1}(e)$, where $v \in \mathcal{F}$
is any vertex and $e$ is an edge of $\Delta$ that comes out of
$\mathcal{F}$ and is incident to $v$, and $k \in \Z$ is the number of
connected components of $M_{\mathrm{exc}}$. The overall aim of this section is
to determine which quadruples arise in this fashion (see Corollary \ref{cor:realizability}). So far, we have
established the following necessary conditions:
 \begin{itemize}[leftmargin=*]
 \item[(i)] $\Delta$ is a full-dimensional reflexive Delzant polytope in $\g^*$ (Propositions \ref{prop:mom_map_image} and \ref{prop:tall=complexity_preserving}).
 \item[(ii)] $\mathcal{F}$ is a facet of $\Delta$ that is supported on the affine hyperplane $ \{w \in \g^* \mid \langle w,\nu \rangle =
  -1\}$.
  \item[(iii)] The pair $(s,k)$ belongs to the set $\{(0,0), (-1,0), (-1,1), (-1,2) \}$ (Proposition \ref{prop:necessary}).
  \item[(iv)] If there is a vertex of $\Phi(M)$ on the linear hyperplane $\{w
  \in \g^* \mid \langle w, \nu \rangle =
  0\}$, then $k=0$ (Corollary \ref{cor:no_iso_fixed_points}).
  \item[(v)] If $k=2$ then there exists a facet $\mathcal{F}'$ supported on the hyperplane 
  $\{w\in \g^* \mid \langle w, -\nu \rangle =
  -1\}$ (Corollary \ref{cor:k=2}).
\end{itemize}
The first step towards proving which quadruples are associated to a normalized monotone tall complexity one $T$-space 
and whether the $T$ action extends to a toric action 
(Corollary \ref{cor:realizability})
is
establishing its combinatorial analogue, namely Theorem
\ref{thm:comb_real}. To this end, we introduce the following
terminology.

\begin{definition}\label{defn:types_polytopes}
  We say that a quadruple $(\Delta, \mathcal{F},s,k)$ consisting of a polytope $\Delta$
  in $\g^*$, a facet $\mathcal{F}\subset \Delta$, and integers $s,k$, is {\bf admissible} if it satisfies conditions (i)--(v) above.
\end{definition}

\begin{remark}\label{rmk:k=2}
  If $(\Delta, \mathcal{F}, -1,2)$ is admissible, then, by
  the proof of Corollary \ref{cor:k=2}, the polytope $\Delta$ is
  contained in the strip $\{w \in
  \g^* \mid -1 \leq \langle w, \nu \rangle \leq 1\}$. 
\end{remark}

\begin{lemma}\label{lemma:k_greater_0}
  Let $k \in \{1,2\}$. If $(\Delta, \mathcal{F}, -1,k)$ is admissible,
  then for any edge $e$ of $\Delta$ that intersects the linear
  hyperplane $ \{w \in \g^* \mid \langle w,\nu \rangle =
  0\}$, there exists a vertex $v \in \mathcal{F}$ and a weight
  $\alpha_v \in \ell^*$ at $v$ such that
  $$ e \cap \{w \in \g^* \mid \langle w,\nu \rangle =
  0\} = \{v + \alpha_v\}.$$
  In particular, $ e \cap \{w \in \g^* \mid \langle w,\nu \rangle =
  0\}$ is contained in $\ell^*$.
\end{lemma}

\begin{proof}
  Fix such an edge $e$. Since $(\Delta, \mathcal{F}, -1,k)$ is
  admissible and $k >0$, $\Delta$ has no vertex on the linear
  hyperplane $ \{w \in \g^* \mid \langle w,\nu \rangle =
  0\}$. Hence, $e$ intersects $ \{w \in \g^* \mid \langle w,\nu \rangle =
  0\}$ in the relative interior of $e$, so that the intersection $ e \cap \{w \in \g^* \mid \langle w,\nu \rangle =
  0\}$ consists of one element, which is not a vertex since $(\Delta, \mathcal{F}, -1,k)$ is
  admissible. Let $v \in \Delta$ be the vertex that is incident
  to $e$ and satisfies $\langle v, \nu \rangle < 0$. Since $\Delta$ is
  integral and is contained in the upper-half plane $\{w \in \g^* \mid
  \langle w,\nu \rangle \geq -1\}$, and since $\nu \in \ell^*$ is
  primitive, $\langle v, \nu \rangle = -1$, i.e., $v
  \in \mathcal{F}$. Moreover, $e$ is the edge that comes out of
  $\mathcal{F}$ that is incident to $v$. Let $\alpha_v \in \ell^*$ be
  the weight of $v$ such that $e$ is contained in the half-ray $v +
  \R_{\geq 0} \langle \alpha_v \rangle$. By Lemma
  \ref{lemma:weight_edge_out_facet}, $\langle v+ \alpha_v, \nu \rangle
  = 0$,  whence $ e \cap \{w \in \g^* \mid \langle w,\nu \rangle =
  0\} = \{v + \alpha_v\}$. Since $v, \alpha_v \in \ell^*$, the result
  follows. 
\end{proof}

Admissible quadruples encode the abstract analogs of the functions
given by \eqref{eq:57} in Theorem \ref{thm:possibilities_DH}.

\begin{definition}\label{defn:admissible_DH}
  Let $(\Delta, \mathcal{F}, s,k)$ be an admissible quadruple. The
  {\bf abstract Duistermaat-Heckman function} determined
  by $(\Delta, \mathcal{F}, s,k)$ is the map $\Delta \to \R$ that sends $w \in \Delta$ to
\begin{equation}
    \label{eq:66}
    DH (w) = 2 - s \langle w, \nu \rangle - k \rho(w), 
  \end{equation}
  \noindent
  where $\rho : \g^* \to \R$ is given by
  \begin{equation*}
    w \mapsto
    \begin{cases}
      0 & \text{if } \langle w, \nu \rangle \leq 0 \\
    \langle w, \nu \rangle & \text{if } \langle w, \nu \rangle \geq 0.
    \end{cases}
  \end{equation*}
\end{definition}

\begin{remark}\label{rmk:abstract_determined_quadruple}
  Let $(\Delta, \mathcal{F}, s,k)$ and $(\Delta', \mathcal{F}',
  s',k')$ be admissible quadruples. By the arguments in the proof of
  Theorem \ref{thm:DH_classifies} (see page \pageref{proof theorem thm:DH_classifies}), if the abstract Duistermaat-Heckman
  functions determined by $(\Delta, \mathcal{F}, s,k)$ and $(\Delta', \mathcal{F}',
  s',k')$ are equal, then $(\Delta, \mathcal{F}, s,k) = (\Delta', \mathcal{F}',
  s',k')$.
\end{remark}

In what follows, we fix the map
$\mathrm{pr} : \g^* \times \R \to \g^*$ given by projection to the first
component and we denote the Lebesgue measure on $\R$ by $dy$. Given a polytope $\Delta'$ in $\g^* \times \R$, the projection
$\mathrm{pr}(\Delta')$ is a polytope in $\g^*$. On such a projection we
define the combinatorial analog of the function constructed in Example
\ref{exm:induced_DH} (the terminology used below is not standard).

\begin{definition}\label{defn:height_function}
  Let $\Delta'$ be a polytope in $\g^* \times \R$. The {\bf height
    function} of $\Delta :=\mathrm{pr}(\Delta')$ is the map $\Delta
  \to \R$ that sends $w \in \Delta$ to
  $$ \mathrm{Length}(\Delta'_w):=\int\limits_{\Delta'_w} dy, $$
  \noindent
  where $\Delta'_w:= \mathrm{pr}^{-1}(w) \cap \Delta$.
\end{definition}

The combinatorial realizability result is as follows.
\begin{theorem}\label{thm:comb_real}
  For each admissible quadruple $(\Delta, \mathcal{F}, s,k)$, there exists a
  reflexive Delzant polytope $\Delta'$ in $\g^* \times \R$ such that
  $\mathrm{pr}(\Delta') = \Delta$ and the height function
  of $\Delta$ equals the abstract Duistermaat-Heckman function determined
  by $(\Delta, \mathcal{F}, s,k)$.
\end{theorem}

In Figure \ref{Figure: ExtensionReflexiveSquare} we provide the
complete list of reflexive Delzant polytopes $\Delta'$ such that the
projection is the
reflexive square $\Delta$ of Figure \ref{Figure:ReflexiveSquare}. Before turning to the proof of Theorem \ref{thm:comb_real}, following
\cite[Section 2.4]{mcduff_tolman}, we introduce an important
construction on smooth polytopes.

\begin{definition}\label{defn:comb_blow-up}
  Let $\Delta$ be a Delzant polytope in $\g^*$ given by $\Delta=\bigcap_{i=1}^l \, \{w\in \g^* \mid \langle w,\nu_i \rangle \geq
  c_i\}$, let $\mathcal{F}$ be a face of $\Delta$ of codimension at
  least two, and let $I \subset \{1,\ldots, l\}$ be the subset of those
  indices corresponding to the facets containing $\mathcal{F}$. We set
  $\nu_0:= \sum\limits_{i \in I} \nu_i$ and, given
  $\epsilon > 0$, we also set $c_0:= \epsilon + \sum\limits_{i \in I}
  c_i$. For any $\epsilon >0$ such that any vertex $v$ of $\Delta$ not
  lying on $\mathcal{F}$ satisfies $\langle v, \nu_0 \rangle > c_0$,
  we define the {\bf blow-up of $\Delta$ along $\mathcal{F}$ of size
    $\epsilon$} to be the polytope 
  $$ \Delta \cap \{w \in \g^* \mid \langle w, \nu_0 \rangle \geq c_0\}.$$
\end{definition}

\begin{remark}\label{rmk:blow-up_smooth}
  As remarked in \cite[Section 2.4]{mcduff_tolman}, any blow-up of a Delzant
  polytope (along any face and of any size) is a Delzant polytope (so
  long as the face has codimension at least two and the size satisfies
  the condition stated in Definition \ref{defn:comb_blow-up}). 
\end{remark}

\begin{proof}[Proof of Theorem \ref{thm:comb_real}]
  We fix an admissible quadruple $(\Delta, \mathcal{F}, s,k)$. We
  split the proof in two cases, depending depending on whether $k=0$
  or not.\\
  
  {\bf Case 1: $k=0$}. The abstract Duistermaat-Heckman function is $DH(w) = 2 - s\langle w, \nu
  \rangle$ for all $w \in \Delta$. We set
  \begin{equation}
    \label{eq:67}
    \Delta_{(s,0)}':=\{(w,y) \in \g^* \times \R \mid w \in \Delta \, , \, -1
    \leq y \leq DH(w) -1\}, 
  \end{equation}
  where $DH : \Delta \to \R$ is the abstract Duistermaat-Heckman
  function determined by $(\Delta, \mathcal{F},
  s,k)$ -- see Figure \ref{ToricExtension1}. First, we claim that $\mathrm{pr}(\Delta_{(s,0)}') = \Delta$. To 
  this end, it suffices to prove that 
  $DH(w) \geq 1$ for all $w \in \Delta$. This follows immediately from the fact that,
  by Definition \ref{defn:types_polytopes}, $s
  \in \{0,-1\}$ and $\Delta$ is contained in the upper half-space of $\g^*$ 
  given by $ \{w \in \g^* \mid \langle w, \nu \rangle \geq
  -1\}$. By
  construction, the height function of $\Delta_{(s,0)}'$ equals $DH$, so it
  remains to show that $\Delta_{(s,0)}'$ is a reflexive Delzant polytope. To this
  end, we write $\Delta$ in its minimal representation (see \eqref{eq:3})
  \begin{equation*}
    \Delta=\bigcap_{i=1}^l \, \{w \in \g^* \mid \langle w, \nu_i \rangle \geq
    -1\},
  \end{equation*}
  where $\nu_i \in \ell^*$ is primitive for
  all $i=1,\ldots, l$ and, without loss of generality, the hyperplane supporting
  $\mathcal{F}$ is $\{w\in \g^* \mid \langle w,\nu_1\rangle \geq
  -1\}$, i.e., $\nu_1 = \nu$. By \eqref{eq:67},
  \begin{equation}
    \label{eq:68}
    \begin{split}
      \Delta_{(s,0)}' & =\{(w,y) \in \g^* \times \R \mid \langle (w,y), (0,1)
      \rangle \geq -1\} \\
      & \cap \{(w,y)
      \in \g^* \times \R \mid \langle (w,y),(-s\nu,-1) \rangle \geq -1\} \\
      & \cap \bigcap_{i=1}^l \, \{(w,y) \in \g^* \times \R \mid \langle (w,y),(\nu_i,0) \rangle \geq
      -1\},
    \end{split}
  \end{equation}
  where, by a slight abuse of notation, we denote the natural pairing
  between $\g^* \times \R$ and $\g \times \R$ also by $\langle \cdot,
  \cdot \rangle$. Therefore, $\Delta_{(s,0)}'$ is a polytope (see
  Section \ref{sec:conventions}). Moreover, the vertices of $\Delta_{(s,0)}'$ are precisely the elements of
  the set
  \begin{equation}
    \label{eq:69}
    \{(v, -1) \in \g^* \times \R \mid v \in \Delta \text{
      vertex} \} \, \cup \, \{(v, 1 - s\langle v, \nu\rangle) \in \g^*
    \times \R \mid v \in \Delta \text{
      vertex} \}. 
  \end{equation}
  We observe that, since $\Delta$ is reflexive, any vertex of $\Delta$
  lies in $\ell^*$; since $\nu \in \ell$, it follows that, if $v \in
  \Delta$ is a vertex, then $1 - s\langle v, \nu\rangle \in \Z$. Hence, by \eqref{eq:69}, any vertex of $\Delta_{(s,0)}'$
  lies in $\ell^* \times \Z$, i.e., $\Delta_{(s,0)}'$ is integral. Since
  $\nu_i \in \ell^*$ is primitive, $(\nu_i,0) \in \ell^* \times \Z$ is
  primitive. Moreover, since $\nu
  = \nu_1$ and since $s \in \{0,-1\}$, $(-s\nu,-1) \in \ell^* \times
  \Z$ is also primitive. Hence, by \eqref{eq:68}, $\Delta_{(s,0)}'$ is
  reflexive. Finally, to see that $\Delta_{(s,0)}'$ is Delzant, we fix a
  vertex $v \in \Delta$. By \eqref{eq:68}, the set of inward normals of the facets of
  $\Delta_{(s,0)}'$ that contain $(v,-1)$ (respectively $(v, 1 - s\langle
  v,\nu\rangle)$) consists of $(0,1)$
   (respectively $(-s\nu, -1)$), and of $\{(\nu_v,
  0)\}$, where $\{\nu_v\}$ is the set of inward normals of the facets of
  $\Delta$ that contain $v$. Since $\Delta$ is Delzant, it follows
  that $\Delta_{(s,0)}'$ is smooth at $(v,-1)$ (respectively $(v, 1 - s\langle
  v,\nu\rangle)$). Since any vertex of $\Delta_{(s,0)}'$ is equal to  $(v,-1)$
  or $(v, 1 - s\langle
  v,\nu\rangle)$ for some vertex $v$ of $\Delta$,
  $\Delta_{(s,0)}'$ is Delzant, thus completing the proof in this case. \\
  \begin{figure}[htbp]
 \begin{center}
 \includegraphics[width=8cm]{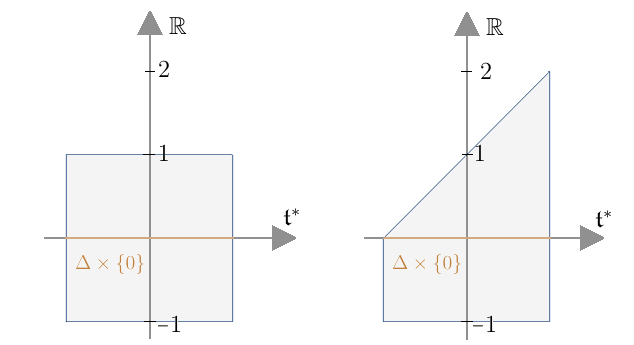}
 \caption{A schematic construction of
   $\Delta_{(0,0)}'$ (left) and of $\Delta_{(-1,0)}'$ (right).}
 \label{ToricExtension1}
 \end{center}
  \end{figure}

{\bf Case 2: $k\neq 0$}. Since $(\Delta, \mathcal{F}, s,k)$ is
  admissible, $(s,k) \in \{(-1,1),(-1,2)\}$. 
  Moreover, by Definition \ref{defn:types_polytopes}, $\Delta$ has no vertices on the linear hyperplane $ \{w \in
  \g^* \mid \langle w, \nu \rangle = 0 \}$ and the quadruple $(\Delta,
  \mathcal{F}, 0,0)$ is also admissible. Hence, by Case 1, there exists a reflexive Delzant polytope $\Delta_{(0,0)}'$ satisfying the conclusions
  of the statement for the admissible quadruple $(\Delta,
  \mathcal{F}, 0,0)$. We deal with the cases $(s,k) = (-1,1)$ and $(s,k) = (-1,2)$
  separately. \\

 {\bf $\bullet$ Suppose that $(s,k) = (-1,1)$.} By \eqref{eq:68}, the reflexive
  Delzant polytope $\Delta_{(0,0)}'$ has a codimension two face
  $\tilde{\mathcal{F}}$ given by the intersection of the facets
  supported by the affine hyperplanes
  $\{(w,y) \in \g^* \times \R \mid \langle (w,y), (0,-1) \rangle = -1 \}$ and
  $\{(w,y) \in \g^* \times \R \mid \langle (w,y), (\nu,0) \rangle = -1
  \}$. This is a copy of $\mathcal{F}$ on the affine hyperplane
  $\{(w,y) \in \g^* \times \R \mid y = 1\}$. We wish to perform the blow-up of $\Delta'_{(0,0)}$ along
  $\tilde{\mathcal{F}}$ of size $1$ (see Figure \ref{ToricExtension2}). To this end, with the
  notation in Definition \ref{defn:comb_blow-up}, $\nu_0 = (\nu,-1)$
  and $c_0 = -1$. By \eqref{eq:69} and since $s = -1$, a vertex of $\Delta'_{(0,0)}$
  that does not lie on $\tilde{\mathcal{F}}$ is either of the form
  $(v,-1)$ for some vertex $v$ of $\Delta$ or of the form $(v,1)$ for
  some vertex $v$ of $\Delta$ that does not lie on
  $\mathcal{F}$. Since $\langle w, \nu \rangle \geq -1$ for any $w \in
  \Delta$, if $v$ is a vertex of $\Delta$, then
  \begin{equation}
    \label{eq:72}
    \langle (v,-1),(\nu,-1) \rangle \geq -1 + 1 = 0 > -1 = c_0. 
  \end{equation}
  On the other hand, since there are no vertices of $\Delta$ lying on
  the linear hyperplane $ \{w \in \g^* \mid \langle w, \nu \rangle = 0
  \}$ and since $\Delta$ is integral, if $v$ is a vertex of $\Delta$
  that does not lie on $\mathcal{F}$, then $\langle v, \nu \rangle
  \geq 1$. Hence, in this case,
  \begin{equation}
    \label{eq:73}
    \langle (v,-1),(\nu,1) \rangle \geq 1 -1 > -1 = c_0. 
  \end{equation}
  By \eqref{eq:72} and \eqref{eq:73}, we can perform the the blow-up of $\Delta'_{(0,0)}$ along
  $\tilde{\mathcal{F}}$ of size $1$ that we denote by
  $\Delta'_{(-1,1)}$, i.e.,
  \begin{equation}
    \label{eq:74}
    \Delta'_{(-1,1)} = \Delta'_{(0,0)} \cap \{(w,y) \in \g^* \times
    \R \mid \langle (w,y),(\nu,-1) \rangle \geq -1\}. 
  \end{equation}
  By \eqref{eq:67} and \eqref{eq:74}, and since $s =
  -1$,
  \begin{equation}
    \label{eq:75}
    \Delta'_{(-1,1)} = \{(w,y) \in \g^* \times \R \mid w \in \Delta \, , \, -1
    \leq y \leq \min(1,1 + \langle w, \nu \rangle)\}. 
  \end{equation}  
  Since $\Delta'_{(0,0)}$ is Delzant, by Remark
  \ref{rmk:blow-up_smooth}, $\Delta'_{(-1,1)}$ is also Delzant. By
  \eqref{eq:75}, it can be checked directly that a vertex of
  $\Delta'_{(-1,1)}$ is one of the following three types:
  \begin{itemize}[leftmargin=*]
  \item $(v, \min(1,1 + \langle v, \nu \rangle))$ for some vertex $v$
    of $\Delta$,
  \item $(w,1)$, where $w$ lies on an edge of $\Delta$ and satisfies $\langle w, \nu
    \rangle = 0$, or
  \item $(v,-1)$ for some vertex $v$ of $\Delta$.  
  \end{itemize}
  Since $\Delta$ is integral and since $\nu \in \ell^*$, if $v$ is a
  vertex of $\Delta$, then $(v, \min(1,1 + \langle v, \nu \rangle))$ and $
  (v,-1)$ belong to $\ell^* \times \Z$. Moreover, by Lemma
  \ref{lemma:k_greater_0}, if $w$ lies on an edge
  of $\Delta$ and satisfies $\langle w, \nu \rangle = 0$, then $w \in
  \ell^*$. Hence, $(w,1) \in \ell^* \times \Z$, so that
  $\Delta'_{(-1,1)}$ is integral. Moreover, by
  \eqref{eq:68} and \eqref{eq:74}, $\Delta'_{(-1,1)}$
  is reflexive. Since $(s,k) = (-1,1)$, it follows that the map $\Delta \to \R$ that
  sends $w$ to $\min(1,1 + \langle w, \nu \rangle)$ equals $DH - 1$,
  where $DH : \Delta \to \R$ is the abstract Duistermaat-Heckman
  function determined by $(\Delta, \mathcal{F},
  -1,1)$. This implies both that $\mathrm{pr}(\Delta'_{(-1,1)}) =
  \Delta$ (since the minimal value of the above map on $\Delta$ is
  $0$), and that the height function of $\Delta$ equals the abstract Duistermaat-Heckman
  function  determined by $(\Delta, \mathcal{F},
  -1,1)$, as desired. \\
  
  \begin{figure}[htbp]
  	\begin{center}
  		\includegraphics[width=8cm]{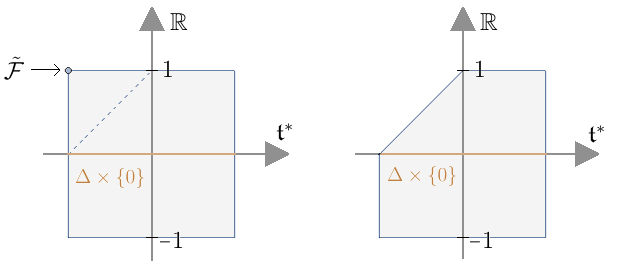}
  		\caption{A schematic representation of $\Delta_{(-1,1)}'$ as a blow-up of $\Delta_{(0,0)}'$ along $\tilde {\mathcal{F}}$
  		of size $1$.}
  		\label{ToricExtension2}
  	\end{center}
  \end{figure}

{\bf $\bullet$ Suppose that $(s,k) = (-1,2)$:} Since $(\Delta, \mathcal{F},
  -1,2)$ is admissible, there exists a facet $\mathcal{F}'$ of
  $\Delta$ supported on the affine hyperplane $ \{w \in \g^* \mid \langle w,-\nu \rangle =
  -1\}$ and the quadruple $(\Delta, \mathcal{F},
  -1,1)$ is also admissible. Let $\Delta'_{(-1,1)}$ be the reflexive
  Delzant polytope constructed from $(\Delta, \mathcal{F},
  -1,1)$ as above. Hence, by \eqref{eq:68} and \eqref{eq:74}, the reflexive
  Delzant polytope $ \Delta'_{(-1,1)}$ has a codimension two face
  $\tilde{\mathcal{F}}'$ given by the intersection of the facets
  supported by the affine hyperplanes $\{(w,y) \in \g^* \times \R \mid \langle (w,y), (0,1) \rangle = -1 \}$ and
  $\{(w,y) \in \g^* \times \R \mid \langle (w,y), (-\nu,0) \rangle =
  -1\}$. This is a copy of $\mathcal{F}'$ on the affine hyperplane
  $\{(w,y) \in \g^* \times \R \mid y = -1\}$. We wish to perform the blow-up of $\Delta'_{(-1,1)}$ along
  $\tilde{\mathcal{F}}'$ of size $1$ (see Figure \ref{ToricExtension3}). To this end, with the
  notation in Definition \ref{defn:comb_blow-up}, $\nu_0 = (-\nu,1)$
  and $c_0 = -1$. A vertex of $\Delta'_{(-1,1)}$
  that does not lie on $\tilde{\mathcal{F}}'$ is of one of three types:
  \begin{itemize}[leftmargin=*]
  \item $(v, \min(1,1 + \langle v, \nu \rangle))$ for some vertex $v$
    of $\Delta$,
  \item $(w,1)$, where $w \in \Delta$ lies on an edge of $\Delta$ and satisfies $\langle w, \nu
    \rangle = 0$, or
  \item $(v,-1)$ for some vertex $v$ of $\Delta$ that does not lie on
    $\mathcal{F}'$,
  \end{itemize}
  (see the proof in the case $(s,k) = (-1,1)$.) In the first case, we have that 
  \begin{equation}
    \label{eq:76}
    \langle (v, \min(1,1 + \langle v, \nu \rangle)), (-\nu,1) \rangle
    = \min(1- \langle v, \nu \rangle,1) > -1 = c_0,
  \end{equation}
  where the inequality follows from the fact that $\Delta$ is
  contained in the strip $\{w \in \g^* \mid -1 \leq \langle w, \nu
  \rangle \leq 1\}$ (see Remark \ref{rmk:k=2}). In the second case, we
  have that
  \begin{equation}
    \label{eq:77}
    \langle (w,1), (-\nu,1)\rangle = 1 > -1 = c_0.
  \end{equation}
  As in the case $(s,k)=(-1,1)$, if $v$ is a vertex of
  $\Delta$, 
  then $\langle v, -\nu
  \rangle \geq 1$. Hence, in the third case, we have that
  \begin{equation}
    \label{eq:78}
    \langle (v,-1), (-\nu,1) \rangle \geq 0 > -1 =c_0.
  \end{equation}
  By \eqref{eq:76}, \eqref{eq:77} and \eqref{eq:78}, we can perform the the blow-up of $\Delta'_{(-1,1)}$ along
  $\tilde{\mathcal{F}}'$ of size $1$ that we denote by
  $\Delta'_{(-1,2)}$, i.e.,
  \begin{equation}
    \label{eq:80}
    \Delta'_{(-1,2)} = \Delta'_{(-1,1)} \cap \{(w,y) \in \g^* \times
    \R \mid \langle (w,y),(-\nu,1) \rangle \geq -1\}. 
  \end{equation}
  By \eqref{eq:75}, we have that
  \begin{equation}
    \label{eq:79}
    \Delta'_{(-1,2)} = \{(w,y) \in \g^* \times \R \mid w \in \Delta \,
    , \, \max(-1,-1 + \langle w, \nu \rangle)
    \leq y \leq \min(1,1 + \langle w, \nu \rangle)\}. 
  \end{equation}
  Since $\Delta'_{(-1,1)}$ is smooth, by Remark
  \ref{rmk:blow-up_smooth}, $\Delta'_{(-1,2)}$ is also
  smooth. Moreover, by 
  \eqref{eq:79}, it can be checked directly that a vertex of
  $\Delta'_{(-1,2)}$ is one of the following three types:
  \begin{itemize}[leftmargin=*]
  \item $(v, \min(1,1 + \langle v, \nu \rangle))$ for some vertex $v$
    of $\Delta$,
  \item $(w, \pm 1)$, where $w$ lies on an edge of $\Delta$ and satisfies $\langle w, \nu
    \rangle = 0$, or
  \item $(v,\max(-1,-1+\langle v, \nu \rangle)$ for some vertex $v$ of $\Delta$.  
  \end{itemize}
  As in the case $(s,k)=(-1,1)$, it follows that
  $\Delta'_{(-1,2)}$ is integral. Moreover, since
  $\Delta'_{(-1,1)}$ is reflexive, by \eqref{eq:74} $\Delta'_{(-1,1)}$
  is reflexive. Since $\Delta$ is contained in the strip $\{w \in \g^* \mid -1 \leq \langle w, \nu
  \rangle \leq 1\}$, the maximal (respectively
  minimal) value of the map
  $\Delta \to \R$ that takes $w$ to $\max(-1,-1 + \langle w, \nu
  \rangle)$ (respectively $\min(1,1 + \langle w, \nu \rangle)$) is
  zero. Hence, $\mathrm{pr}(\Delta'_{(-1,2)}) =
  \Delta$. Moreover, the height function of $\Delta$ is the map
  $\Delta \to \R$ that sends $w \in \Delta$ to
  $$ \min (1,1 + \langle w, \nu \rangle) - \max(-1,-1 + \langle w, \nu
  \rangle) = \min(2 - \langle w, \nu \rangle, 2 + \langle w, \nu
  \rangle). $$
  Since $(s,k) = (-1,2)$, the above map equals the abstract
  Duistermaat-Heckman function determined by $(\Delta, \mathcal{F},
  -1,2)$, as desired.
  \begin{figure}[htbp]
  	\begin{center}
  		\includegraphics[width=8cm]{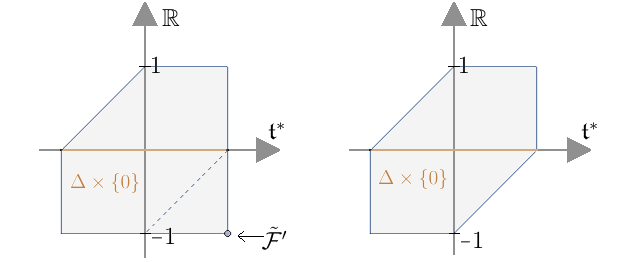}
  		\caption{A schematic representation of $\Delta_{(-1,2)}'$ as a blow-up of $\Delta_{(-1,1)}'$ along $\tilde {\mathcal{F}'}$
  			of size $1$.}
  		\label{ToricExtension3}
  	\end{center}
  \end{figure}
\end{proof}

Theorem \ref{thm:comb_real} and Delzant's classification of compact
symplectic toric manifolds \cite{delzant} yield the following
geometric realizability and extension result.

\begin{corollary}\label{cor:realizability}
If $(\Delta, \mathcal{F}, s,k)$ is an admissible quadruple, then
there exists a normalized monotone tall
  complexity one $T$-space $\comp$ such that
  \begin{itemize}[leftmargin=*]
  \item $\Phi(M) = \Delta$ and the
    Duistermaat-Heckman function of $\comp$ equals the abstract
    Duistermaat-Heckman function determined by $(\Delta,
    \mathcal{F}, s,k)$, and
  \item the Hamiltonian $T$-action extends to an effective Hamiltonian $(T \times
    S^1)$-action. 
  \end{itemize}
\end{corollary}

\begin{proof}
  By Theorem \ref{thm:comb_real}, there exists a reflexive Delzant polytope $\Delta'$ in $\g^* \times \R$ such that
  $\mathrm{pr}(\Delta') = \Delta$ and the height function
  of $\Delta$ equals the abstract Duistermaat-Heckman function determined
  by $(\Delta, \mathcal{F}, s,k)$. By \cite{delzant}, there exists a compact
  complexity zero $(T \times S^1)$-space $(M,\omega, \tilde{\Phi} =
  (\Phi,\Psi))$ such that the moment map image $\tilde{\Phi}(M) =
  \Delta'$, where we identify $(\g \times \R)^*$ with $\g^* \times
  \R$. We claim that $\comp$ satisfies the desired properties. To see this, we observe that, by construction, it is
  tall and has complexity one, and the $T$-action extends to an effective Hamiltonian $(T \times
  S^1)$-action. Moreover, since $\Delta'$ is a reflexive Delzant
  polytope, by Proposition \ref{prop:monotone_toric}, $(M,\omega,
  \tilde{\Phi})$ is normalized monotone so that, in particular,
  $(M,\omega)$ satisfies $c_1 = [\omega]$. Since $\Delta$ is reflexive
  Delzant and since $\mathrm{pr}(\Delta') = \Delta$,
  $\Phi(M) = \Delta$, so that $\Phi$ satisfies the weight sum
  formula. Hence, $\comp$ is normalized monotone. Finally, by Example
  \ref{exm:induced_DH}, the Duistermaat-Heckman function
  of $\comp$ equals the height function of $\Delta$. Since the latter equals the abstract Duistermaat-Heckman function determined
  by $(\Delta, \mathcal{F}, s,k)$, the result follows. 
\end{proof}

\begin{remark}
  In fact, the constructions in the proof of Theorem
  \ref{thm:comb_real} have geometric counterparts that allow to give
  an explicit geometric description of $\comp$ in Corollary
  \ref{cor:realizability}. For instance, the case $k=0$ is described
  explicitly in \cite[Example 1.6]{kt3}, while it is well-known that
  the combinatorial blow-up of a polytope along a face corresponds to
  an equivariant symplectic blow-up (see \cite{mcduff_tolman}). 
\end{remark}

We can prove another important result of this paper.

\begin{proof}[Proof of Theorem \ref{thm:extension}]
  By Corollary \ref{cor:mono_normalized}, we may assume that $\comp$ is normalized monotone. Hence,
  $\Delta := \Phi(M)$ is a reflexive Delzant polytope by Proposition \ref{prop:mom_map_image}. Let
  $\mathcal{F}_{\min} \subset \Phi(M)$ be a minimal facet and let 
  $(s,k)$ be as in the statement of
  Proposition \ref{prop:necessary}. By construction, the quadruple $(\Delta,
  \mathcal{F}_{\min}, s,k)$ is admissible. Hence, by Corollary
  \ref{cor:realizability}, there exists a normalized monotone tall
  complexity one $T$-space $(M',\omega', \Phi')$ such that
  \begin{itemize}[leftmargin=*]
  \item its Duistermaat-Heckman function equals the abstract
    Duistermaat-Heckman function associated to
    $(\Delta, \mathcal{F}_{\min}, s,k)$, and
  \item the Hamiltonian $T$-action extends to an effective Hamiltonian $(T \times
    S^1)$-action.
  \end{itemize}
  By construction, $\comp$ and $(M',\omega', \Phi')$ have equal
  Duistermaat-Heckman functions. Hence, by Theorem
  \ref{thm:DH_classifies}, they are isomorphic and the result follows.
\end{proof}

\subsection{Compact monotone tall complexity one spaces are
  equivariantly Fano}\label{sec:comp-monot-tall}
In this section, we prove the last main result of our paper, Theorem
\ref{thm:Fano}. To this end, we recall that a compact complex manifold
$(Y,J)$ is Fano if and only if there exists a K\"ahler form $\sigma
\in \Omega^{1,1}(Y)$ such that $c_1(Y) = [\omega]$.

\begin{proof}[Proof of Theorem \ref{thm:Fano}]
  By Corollary \ref{cor:mono_normalized}, there is no loss of
  generality in assuming that $\comp$ is normalized monotone. By Theorem \ref{thm:extension}, the Hamiltonian $T$-action extends to an effective
  Hamiltonian $(T \times S^1)$-action. We denote the corresponding
  normalized monotone symplectic toric manifold by $(M,\omega, \tilde{\Phi} =
  (\Phi,\Psi))$. By the classification of compact
  symplectic toric manifolds in \cite{delzant} that there exists an
  integrable almost complex structure $J$ on $M$ that is compatible with $\omega$ and 
  $(T \times S^1)$-invariant; moreover, $\omega$ equals the K\"ahler
  form of $(M,J)$. By Proposition
  \ref{prop:monotone_toric}, $[\omega] = c_1(M) > 0$,
  so that the K\"ahler
  manifold $(M,J)$ is Fano. Finally, by \cite[Theorems 4.4 and
  4.5]{gs-kahler}, the $T$-action extends to an effective holomorphic $T_{\C}$-action,
  as desired.
\end{proof}


\end{document}